\documentclass[12pt]{article}
\usepackage{amssymb,amsmath,amsthm,epsfig}
\usepackage{latexsym, enumerate}
\usepackage{eepic}
\usepackage{epic}
\usepackage{graphicx}
\usepackage{color}
\usepackage{ifpdf}
\usepackage{subfigure}

\usepackage{dsfont}
\usepackage{multirow}
\usepackage{algorithm}

\usepackage{mathrsfs}
\usepackage{dsfont}

\theoremstyle{plain}

\newtheorem{lem}{Lemma}[section]
\newtheorem{thm}[lem]{Theorem}

\theoremstyle{definition}

\theoremstyle{remark}

\begin{document}
\title{ \large\bf Local-global model reduction method  for stochastic optimal control problems constrained by partial differential equations}
\author{Lingling Ma\thanks{College of Mathematics and Econometrics, Hunan University,
Changsha 410082, China. Email: hudalingling@126.com}
\and
Qiuqi Li\thanks{College of Mathematics and Econometrics, Hunan University, Changsha 410082, China. Email:qiuqili@hnu.edu.cn.}
\and
Lijian Jiang\thanks{Institute  of Mathematics, Hunan University, Changsha 410082, China. Email: ljjiang@hnu.edu.cn. Corresponding author}
}

\date{}
\maketitle
\begin{center}{\bf ABSTRACT}
\end{center}\smallskip
In this paper, a local-global model reduction method is presented  to solve stochastic optimal control problems governed by partial
differential equations (PDEs).  If the optimal control problems involve uncertainty,  we need to use a few random variables to
parameterize the uncertainty.  The stochastic optimal control problems require solving coupled optimality system for
a large number of samples in the stochastic space to  quantify the statistics of the system response and explore the uncertainty quantification.
Thus the computation is prohibitively expensive.  To overcome the difficulty, model reduction is necessary to significantly reduce the computation complexity.
We exploit the advantages from both reduced basis method and Generalized Multiscale Finite Element Method (GMsFEM) and develop the local-global model reduction method
for stochastic optimal control problems with PDE constraints.   This local-global model reduction can  achieve much more computation  efficiency than using only local model reduction approach and only global model
reduction approach.  We recast the stochastic optimal problems in the framework of saddle-point problems and analyze the existence and uniqueness of the optimal solutions of the reduced model.
In the local-global approach, most of computation steps are independent of each other.  This is very desirable for scientific computation. Moreover, the online computation for each random sample
is very fast via the proposed model reduction method.  This allows us to compute the  optimality system for a large number of samples.
To demonstrate the performance of the local-global model reduction method,
 a few numerical examples are provided for different stochastic optimal control problems.

\smallskip
{\bf keywords:}
stochastic optimal control problem, reduced basis method, generalized multiscale finite element method,
local-global model reduction

\section{Introduction}
Optimal control problems are often constrained  with partial differential equations (PDEs) when modeling
physical processes in  sciences and engineering.   For problems that involve uncertainty,  stochastic information needs
to be incorporated into the control problems.  This leads to stochastic optimal control problems.  To
characterize the uncertainty,  we often use random variables to parameterize the stochastic functions.
 In practical applications, uncertainties may arise  from various
sources such as the PDE coefficients, boundary conditions, external loadings and shape of physical
domain. The uncertainty  may have significant impact on the optimal solution. For deterministic optimal control problems,
mathematical theories and computational methods have been developed and investigated for many years
(see, e.g. \cite{gl96, jl71, rdw10}), while the development of stochastic optimal control problem governed
by stochastic PDE have gained substantial attention from  the last decades \cite{cqr13, gll11, hlm11, ll13, rw12}.

In this paper, we  consider the stochastic optimal control problems with quadratic cost functional
constrained by stochastic PDEs. For PDE-constrained optimization problems, there is a choice of
whether to discretize-then-optimize or optimize-then-discretize, and there are different opinions regarding
which route to take (see, e.g. \cite{ch02, p13} for more  discussion).   We choose to optimize-discretize-then-reduce approach in this work. Existence and uniqueness of an optimal solution to stochastic optimal control problems
was studied  in \cite{fb74,hlm11, ll13, nrmq13} in the framework of traditional FE method. It is known that  the
numerical solution of PDE-constrained optimization problems is  computationally expensive because it
requires the solutions of a system of PDEs arising from the optimality conditions: the state problem, the adjoint problem and a set of equations ensuring the optimality of the solution. Especially in the many-query context for parameterized PDEs, this computation  becomes much more challenging and may  lead to the ``curse-of-dimensionality" in high dimensional stochastic spaces, which
makes numerical computation very extensive.

There exist  some efficient methods to solve the stochastic optimal control problems. The Monte
Carlo method is one of the most straightforward  schemes to get the approximate  optimal solution in the
stochastic space. However, it is to be blamed for its low convergence rate, thus leading to heavy computational
cost when a full deterministic optimal control problem has to be solved for every sample. Stochastic Galerkin method
has been proved to converge exponentially fast for smooth problems \cite{btz05, hlm11}. Unfortunately, the tensor-product projection scheme produces a large-scale tensor system to be solved, bringing further computational difficulties. The works \cite{cqr13, khr13, rw12} made use of the stochastic collocation method, which can avoid the  tensor-product large algebraic system encountered by the stochastic Galerkin method. However, when
the optimality system is  in high dimensional stochastic spaces, these techniques are needed to solve the optimality system for many times corresponding to the collocation nodes.
To overcome these issues,  model order reduction methods are necessary to solve large-scale stochastic optimal control problems in high dimensional stochastic spaces.
 Roughly speaking, there exist two categories for model order reduction.  One is global model order reduction such as proper orthogonal decomposition and reduce basis (RB) method \cite{gmnp07}.  The other is local model order reduction such as
 coarsen finite element methods and sparse basis methods \cite{hw97, jl17}.   The global model reduction method involves solving a few global problems.  The local model reduction method may still
 have a large number of degree of freedoms and depend on the random parameters.   To exploit the advantages from both local model reduction and global model reduction,  we present a local-global model reduction
 approach to solve the stochastic optimal control problems in the paper.  The local-global model reduction using POD and GMsFEM has been used in flows in heterogeneous porous media \cite{ac15}.
 Our local-global model reduction here is based on reduced basis method and GMsFEM to solve stochastic optimal control problems.

Inspired by \cite{nw14}, we leverage  inexpensive low-fidelity models to provide important  information abut the high-fidelity model outputs and substantially reduce computation complexity.
We build  a low-fidelity model based on a RB method. RB method is one of global model order reduction methods and it allows to recast a
computational demanding problem (``truth" problem)  into a  the reduced problem \cite{bmppg12, chmr10, gk11, qmn15, rhp08} with fast and reliable low-dimensional formulation.
 The main features of RB method \cite{qmn15, rhp08} are: (i) a posterior bound error estimation for choosing some optimal parameters; (ii) a fast  convergent global approximation onto snapshot spaces;  (iii) an offline-online procedure, yielding the RB functions in the offline stage and obtaining the online calculation for each new input parameter with inexpensive cost.

For the RB method, we can use FE method in a fine grid to get the accurate snapshot functions. For the computation of snapshot functions in large-scale or  multiscale models,
the number of degree of freedoms  may be very large  to resolve all scales in fine grid.  Thus, the computation for snapshots may be very expensive.
 So it is desirable  to develop an  inexpensive low-fidelity model to obtain  the snapshot functions.  As a local model reduction  method, Multiscale finite element method (MsFEM)  \cite{hw97, jem10} is an efficient method to achieve a good trade-off between accuracy and efficiency.  The main idea is to divide the fine scale problem into many local problems and use the solutions of the local problems to form a coarse scale model \cite{jem10}.  MsFEM  incorporates the small-scale information to multiscale basis functions and capture the impact of small-scale features on the coarse-scale through a variational formulation. One of the most important features for MsFEM is that the multisacle basis functions can be computed in the offline stage  and used repeatedly for the
model with different source terms, boundary conditions and the  coefficients with similar
multiscale structures \cite{et09}. To capture complex heterogeneities and continuum scales in the models,
one may need to use multiple basis functions per coarse block. To this end, Generalized Multiscale Finite Element Method (GMsFEM) \cite{egh13,egw11} have been developed  in the framework of generalized finite element method \cite{bbo04, mb96}. This approach extends  MsFEM and constructs multiscale basis functions in each coarse element via local spectral problem. In each coarse element, the number  of multiscale basis functions is much less than the  number of fine-scale basis functions.

It is crucial  to select the parameter
samples with an optimal strategy when applying reduce basis method to solve stochastic optimal control problems.  We will adopt the greedy algorithm \cite{lgm96} to choose optimal samples for snapshots in this paper. We
can also use other methods to construct snapshots such as proper orthogonal decomposition and  cross-validation \cite{jl16}.
In Fig.~\ref{Rb-schema}, we describe the schema to get the reduced order model using the proposed  local-global
model reduction  method.
\begin{figure}[!h]
\centering
\includegraphics[width=4.8in, height=3.2in,angle=0.3]{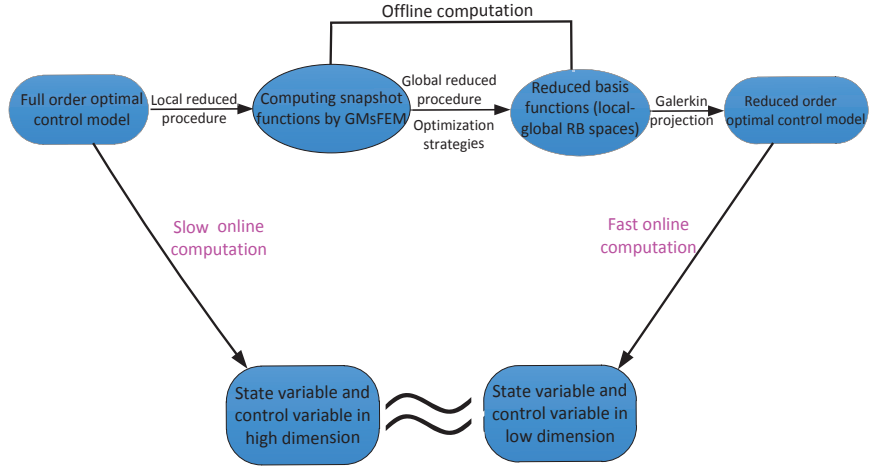}
\caption{\em{Schema of local-global model reduction  for stochastic optimal control problems}}
\label{Rb-schema}
\end{figure}

The paper is structured as follows. In Section \ref{sect2}, we present some preliminaries and notations, the stochastic optimal control problems, and the global existence and uniqueness of optimal solutions. In Section \ref{FEM-va}, FE approximation for the optimal control problems is discussed. Section \ref{sect4} is to present the construction of the global RB method and the equivalent reduced saddle point system
and the local model reduction  method via GMsFEM.  In section \ref{sect5}, we address  the local-global model reduction method and greedy sampling method used for constructing the optimal RB functions.
In Section \ref{sect6},  we present a few numerical examples  for different stochastic optimal control problems to  illustrate the local-global model reduction method. Finally, some conclusions and comments are outlined in the last section.


\section{Stochastic optimal control problems}
\label{sect2}
In this subsection, we first introduce some notations used for the rest of paper and then present
the stochastic optimal control problems. Then, we will recast the original optimal control problems
into the framework of the saddle-point problem and consider the well-posedeness of the saddle-point system.

\subsection{Preliminaries and notations}
Let $(\mathcal{D},\mathcal{F},\mathcal{P})$ be a complete probability space. Here $\mathcal{D}$ is a set of outcomes $\omega \in \mathcal{D}$, $\mathcal{F} \subset 2^{\mathcal{D}}$ is the $\sigma$-algebra of events, and $\mathcal{P}: \mathcal{F} \rightarrow [0,1]$ is a probability measure. We assume that $\Omega$ is a convex bounded polygonal domain in $\mathbb{R}^{d}~(d\geq 1)$ with Lipschitz continuous boundary $\partial \Omega$. Let $v : \Omega \times \mathcal{D} \rightarrow \mathbb{R}$ represent a real-valued random field, i.e., real-valued random function defined in $\mathcal{D}$. For computation, the random field is usually approximated using a prescribed a finite number of random variables,  $\mu(\omega)=\{\mu_1, \mu_2,\cdots, \mu_m\}$, i.e., $\mu(\cdot): \mathcal{D}\rightarrow \Gamma \subset \mathbb{R}^{m}~(m\geq1)$. Thus, we use the random vector $\mu(\omega)$ to characterize the stochastic property of the random field.  Let $\rho$ be the joint probability density function of $\mu(\omega)$.

We define the following tensor-product Hilbert space
\begin{equation*}
\begin{split}
  H^{s}(\Omega)\otimes L^{2}_{\rho} (\Gamma) :=\bigg\{v : \Omega \times \Gamma\rightarrow
\mathbb{R}| v(\cdot,\mu)\in H^{s}(\Omega)~\text{and}~ \|v(\cdot, \mu)\|_{H^{s}(\Omega)}
\in L^{2}_{\rho} (\Gamma), \forall \mu \in \Gamma\bigg\}.
\end{split}
\end{equation*}
We denote $\mathscr{H}^{s}(\Omega):=  H^{s}(\Omega)\otimes L^{2}_{\rho} (\Gamma) $ to shorten the  notation,
and equip it with the following norm
\[\|v\|_{\mathscr{H}^{s}(\Omega)} := \bigg(\int_{\Gamma} \|v(\cdot, \mu)\|^{2}_{H^{s}(\Omega)}\rho d\mu\bigg)^{1/2}.\]
In particular, $\mathscr{H}^{s}_{0}(\Omega)= \{ v\in \mathscr{H}^{s}(\Omega): v |_{\partial \Omega}=0\}$.   When $s=0$, we employ the abbreviated notion $\mathscr{L}^{2}(\Omega)$  to denote  $\mathscr{H}^{0}(\Omega)$ by convention.

\subsection{Problem definition}
 For a stochastic optimal control problem, the  aim is to choose the control function $f(x,\mu(\omega))$ in such a way that the corresponding state function $u(x,\mu(\omega))$ is the best possible approximation to a desired state function $\hat{u}(x,\mu(\omega))$. The stochastic optimal control  problem leads to minimizing an  objective functional with some constraints.
 In the paper, we focus on the following optimal control problem
\begin{equation}
\label{cost-functional}
 \min\limits_{\substack {u\in \mathscr{H}^{1}(\Omega)\\ f \in \mathscr{L}^{2}(\Omega)}} J(u,f)=\frac{1}{2}\|u(x,\mu(\omega))-\hat{u}(x,\mu(\omega))\|
 ^2_{\mathscr{L}^2(\Omega)}+\beta\|f(x,\mu(\omega))\|^2_{\mathscr{L}^2(\Omega)}
\end{equation}
constrained by a stochastic PDE with the weak formulation
\begin{equation}
\label{constraint}
a(u,v;\mu(\omega)) = (f,v;\mu(\omega))~~ \forall v \in \mathscr{H}^{1}_{0}(\Omega),
\end{equation}
subject to the  Dirichlet boundary condition $u|_{\partial \Omega}=g(x)$. Here $J(u,f): \Omega \times \mathcal{D} \rightarrow \mathbb{R}$ is a cost-functional and $\beta$ is a positive constant. In the model problem, we assume that the state  $u(x,\mu(\omega))$ and the control  $f(x,\mu(\omega))$ are random fields represented by the  random vector  $\mu$.

The second  term  in (\ref{cost-functional}) has a regularizing effect and is  called a Tikhonov regularization. The $\beta$ is called a regularization parameter, which  counteracts the tendency of the control to become locally unbounded and the cost functional $J(u,f)$ to approach its minimum.
 For the existence and uniqueness of the solution of the constraint (\ref{constraint}), we  assume that the bilinear form $a(\cdot,\cdot;\mu)$ is coercive and bounded over $H^{1}(\Omega)$ for any $\mu\in \Gamma$, i.e., there exist constants $\tilde{\alpha}_{0}, \tilde{\alpha}_{1} > 0$ such that
\[\tilde{\alpha}(\mu) = \inf\limits_{v\in H^{1}(\Omega)}\frac{a(v,v;\mu)}{\|v\|^{2}_{H^{1}(\Omega)}}\geq \tilde{\alpha}_{0},\quad \forall \mu \in \Gamma,\]
and
\[a(u,v;\mu)\leq \tilde{\alpha}_{1}\|u\|_{H^{1}(\Omega)}\|v\|_{H^{1}(\Omega)},\quad \forall \mu \in \Gamma,\]
for all $u, v\in H^{1}(\Omega)$.\\
\subsection{Saddle-point formulation for optimal control problem}
\label{sad-sys}
In this subsection we introduce the variational formulation of the distributed stochastic optimal control problem. We apply Lagrangian approach \cite{ft10} for the derivation of the stochastic optimality system for the optimal control problem (\ref{cost-functional}-\ref{constraint}). We define  the following stochastic Lagrangian functional as
\begin{equation}
\label{Lagran-fun}
\mathcal{L}(u,f,\lambda)= J(u,f)+a(u,\lambda;\mu)-(f,\lambda;\mu),
\end{equation}
where $\lambda \in \mathscr{H}^{1}_{0}(\Omega)$ is the Lagrangian parameter or adjoint variable.
By taking the Fr$\acute{e}$chet derivative of the Lagrangian functional (\ref{Lagran-fun}) with respect to the variables
$u, f, \lambda$ and evaluating  at $\tilde{u}, \tilde{f}, \tilde{\lambda}$, we can get the first order necessary optimality conditions \cite{cqr13, cq14} of stochastic control problem (\ref{cost-functional}-\ref{constraint}), i.e.,
\begin{equation}
\label{optimal-system1}
\begin{cases}
\begin{split}
a(u,\widetilde{u};\mu)&=(f,\widetilde{u};\mu) ~\forall \widetilde{u} \in \mathscr{H}^{1}_{0}(\Omega) ~~ (\text{state equation}),\\
a'(\lambda,\widetilde{\lambda};\mu)&=-(u-\hat{u},\widetilde{\lambda};\mu) ~\forall \widetilde{\lambda} \in \mathscr{H}^{1}_{0}(\Omega)~~(\text{adjoint equation}),\\
2\beta(f,\widetilde{f};\mu)&=(\widetilde{f},\lambda;\mu) ~\forall \widetilde{f} \in \mathscr{L}^{2}(\Omega)~~(\text {gradient equation}),
\end{split}
\end{cases}
\end{equation}
where $a'(\lambda,\widetilde{\lambda};\mu)= a(\widetilde{\lambda},\lambda;\mu)$ is the adjoint bilinear form, and $(\cdot,\cdot;\mu)$ represents the general $L^{2}$ inner product.

As shown in \cite{gll11, hlm11, rw12}, the optimality system (\ref{optimal-system1}) only has local optimal solutions.
To demonstrate  the global existence and uniqueness of the optimal solution, we will derive a stochastic saddle point
formulation of the optimal control problem (\ref{cost-functional}-\ref{constraint}).

First of all, we introduce the new variables $\underline{u}=(u,f) \in \mathscr{U}$ and $\underline{v}=(v,h)\in \mathscr{U}$, where the tensor space $\mathscr{U}=\mathscr{H}^{1}(\Omega)\times \mathscr{L}^{2}(\Omega)$ is
equipped with graph norm $\|\underline{u}\|_{\mathscr{U}}= \|u\|_{\mathscr{H}^{1}(\Omega)}+\|f\|_{\mathscr{L}^{2}(\Omega)}$.
We define the bilinear forms $\mathcal{A}(\cdot,\cdot): \mathscr{U}\times \mathscr{U}\rightarrow \mathbb{R}$ and $\mathcal{B}(\cdot,\cdot): \mathscr{U}\times \mathscr{H}^{1}_{0}(\Omega)\rightarrow \mathbb{R}$ by
\begin{equation}
\label{bilinear-forms}
\begin{cases}
\begin{split}
\mathcal{A}(\underline{u},\underline{v})&:=(u,v)+2\beta(f,h;\mu),\\
\mathcal{B}(\underline{u},q)&:=a(u,q;\mu)-(f,q;\mu),\\
\end{split}
\end{cases}
\end{equation}
respectively. With the new definitions, we have the following minimization problem equivalent to the original optimal problem  (\ref{cost-functional}-\ref{constraint}), that is
\begin{equation}
\label{mini-formu}
\begin{cases}
\begin{split}
\min\limits_{\underline{u}\in \mathscr{U}_{ad}}\mathcal{J}(\underline{u})&=\frac{1}{2}\mathcal{A}(\underline{u},\underline{u})-(\underline{\hat{u}},
\underline{u}),\\
s.t. ~~\mathcal{B}(\underline{u},\widetilde{u})&=(g,\widetilde{u})_{\partial \Omega} \quad \forall \widetilde{u} \in \mathscr{H}^{1}_{0}(\Omega),
\end{split}
\end{cases}
\end{equation}
where $\underline{\hat{u}}=(\hat{u},0)$, $(\underline{\hat{u}}, \underline{u})=(\hat{u},u)$ and $\mathscr{U}_{ad} \subset \mathscr{U}$ is the admissible set. Moreover, the equivalent saddle point problem for (\ref{mini-formu}) is to find: $(\underline{u}, \lambda) \in \mathscr{U}\times \mathscr{H}^{1}(\Omega)$ such that
\begin{equation}
\label{optimal-system2}
\begin{cases}
\begin{split}
\mathcal{A}(\underline{u},\underline{v})+\mathcal{B}(\underline{v},\lambda)&=(\underline{\hat{u}},\underline{v})\quad \forall \underline{v} \in \mathscr{U},\\
\mathcal{B}(\underline{u},\widetilde{u})&=(g,\widetilde{u})_{\partial \Omega} \quad \forall \widetilde{u} \in \mathscr{H}^{1}_{0}(\Omega).
\end{split}
\end{cases}
\end{equation}
\begin{lem}
\label{equivalence}
Let $\mathscr{U}_{0}:= \{\underline{u}\in \mathscr{U}: \mathcal{B}(\underline{u}, \widetilde{u})=0~~\forall \widetilde{u}\in \mathscr{H}^{1}_{0}(\Omega)\}$ be the kernel space of bilinear form $\mathcal{B}(\cdot, \cdot)$. Then the minimization problem (\ref{mini-formu}) is equivalent to the saddle point formulation (\ref{optimal-system2}). Furthermore, the original minimization (\ref{cost-functional}-\ref{constraint}) and (\ref{optimal-system1}) is also equivalent.
\end{lem}
\begin{proof}
To prove the equivalence, we need to verify the continuity and coercivity properties of $\mathcal{A}(\cdot,\cdot)$ and the inf-sup condition for $\mathcal{B}(\cdot,\cdot)$ \cite{bg09,bf91}. By the definition of bilinear form $\mathcal{A}(\cdot,\cdot)$ in (\ref{bilinear-forms}), we have $\mathcal{A}(\underline{u},\underline{v})
= \mathcal{A}(\underline{v},\underline{u})$ and $\mathcal{A}(\underline{u},\underline{u})\geq 0$. So
$\mathcal{A}$ is symmetric and nonnegative. Because
\begin{equation*}
\begin{split}
\mathcal{A}(\underline{u},\underline{v})&\leq \|u\|_{\mathscr{L}^{2}(\Omega)}\|v\|_{\mathscr{L}^{2}(\Omega)}+\gamma
\|f\|_{\mathscr{L}^{2}(\Omega)}\|h\|_{\mathscr{L}^{2}(\Omega)}\\
&\leq \|u\|_{\mathscr{H}^{1}(\Omega)}\|v\|_{\mathscr{H}^{1}(\Omega)}+\gamma
\|f\|_{\mathscr{L}^{2}(\Omega)}\|h\|_{\mathscr{L}^{2}(\Omega)}\\
&\leq \gamma_{a}\|\underline{u}\|_{\mathscr{U}}\|\underline{v}\|_{\mathscr{U}},
\end{split}
\end{equation*}
$\mathcal{A}(\cdot,\cdot)$ is continuous on $\mathscr{U}\times \mathscr{U}$ and $\gamma_{a}$ is a constant depending on $\gamma$. By the definition of kernel space $\mathscr{U}_{0}$, we have $\mathcal{B}(\underline{u},\widetilde{u})=0$, that is $a(u,\widetilde{u};\mu)=(f,\widetilde{u};\mu),~\forall \widetilde{u} \in \mathscr{H}^{1}_{0}(\Omega)$. Hence, by the coercivity property of $a(\cdot,\cdot;\mu)$ and Cauchy-Schwarz inequality, it holds that $\|u\|_{\mathscr{H}^{1}(\Omega)} \leq C_{1} \|f\|_{\mathscr{L}^{2}(\Omega)}$ with appropriate constant $C_{1}$. Then the coercivity of $\mathcal{A}(\cdot,\cdot)$ follows,
\begin{equation*}
\begin{split}
\mathcal{A}(\underline{u},\underline{u})&=\|u\|^{2}_{\mathscr{L}^{2}(\Omega)}+ \beta \|f\|^{2}_{\mathscr{L}^{2}(\Omega)}\\
&\geq\|u\|^{2}_{\mathscr{L}^{2}(\Omega)}+\frac{\beta}{2C_{1}}\|u\|^{2}_{\mathscr{H}^{1}(\Omega)}+\frac{\beta}{2} \|f\|^{2}_{\mathscr{L}^{2}(\Omega)}\\
&\geq \gamma_{b} \|\underline{u}\|^{2}_{\mathscr{U}},
\end{split}
\end{equation*}
where the coefficient $\gamma_{b}=\max\{\frac{\beta}{2C_{1}},\frac{\beta}{2}\}$.

Next, by the definition of $\mathcal{B}(\cdot,\cdot)$ and the continuity property of bilinear form $a(\cdot,\cdot;\mu)$, we have
\begin{equation*}
\begin{split}
\mathcal{B}(\underline{u},\widetilde{u})&\leq \tilde{\alpha}_{1}\|u\|_{\mathscr{H}^{1}(\Omega)}\|\widetilde{u}\|_{\mathscr{H}^{1}(\Omega)}+
\|f\|_{\mathscr{L}^{2}(\Omega)}\|\widetilde{u}\|_{\mathscr{L}^{2}(\Omega)}\\
&\leq \max \{\tilde{\alpha}_{1},1\}\|\underline{u}\|_{\mathscr{U}}\|\widetilde{u}\|_{\mathscr{H}^{1}(\Omega)},
\end{split}
\end{equation*}
where $\tilde{\alpha}_{1}$ is the continuity constant of the bilinear form $a(\cdot,\cdot;\mu)$.

Finally, we exploit the fact that state variable and adjoint variable belong to the same space
$\mathscr{H}^{1}(\Omega)$ and the coercivity property of the bilinear form $a(\cdot,\cdot;\mu)$. The inf-sup condition of $\mathcal{B}(\cdot,\cdot)$ on $\mathscr{U}\times\mathscr{H}^{1}(\Omega)$ follows
\begin{equation*}
\begin{split}
\sup\limits_{0\neq\underline{u}\in \mathscr{U}} \frac{\mathcal{B}(\underline{u},\widetilde{u})}{\|\underline{u}\|_{\mathscr{U}}}
&=\sup\limits_{0\neq(u,h)\in \mathscr{H}^{1}(\Omega)\times\mathscr{L}^{2}(\Omega)}\frac{a(u,\widetilde{u};\mu)-(h,\widetilde{u};\mu)}
{\|u\|_{\mathscr{H}^{1}(\Omega)}+\|h\|_{\mathscr{L}^{2}(\Omega)}}\\
&\geq \sup\limits_{0\neq(u,0)\in \mathscr{H}^{1}(\Omega)\times\mathscr{L}^{2}(\Omega)}\frac{a(u,\widetilde{u};\mu)}
{\|u\|_{\mathscr{H}^{1}(\Omega)}}\\
&\geq \tilde{\alpha}(\mu)\|\widetilde{u}\|_{\mathscr{H}^{1}(\Omega)}.
\end{split}
\end{equation*}
Here $\tilde{\alpha}(\mu)$ is the infimum of the bilinear form $a(\cdot,\cdot;\mu)$.\\
\end{proof}
\begin{lem}
The saddle point problem (\ref{optimal-system2}) is equivalent to the stochastic optimality system (\ref{optimal-system1}).
\end{lem}
\begin{proof}
Equation (\ref{optimal-system2}) amounts to finding $(u,f,\lambda)\in \mathscr{H}^{1}(\Omega)\times\mathscr{L}^{2}(\Omega)\times\mathscr{H}^{1}_{0}(\Omega)$ such that
\begin{equation}
\label{full-optimal-system}
\begin{cases}
\begin{split}
(u,v)+2\beta(f,h;\mu)+a(v,\lambda;\mu)-(h,\lambda;\mu)&=(\hat{u},v)  \quad \forall v \in \mathscr{H}^{1}(\Omega),~\forall h \in \mathscr{L}^{2}(\Omega) \\
a(u,\widetilde{u};\mu)-(f,\widetilde{u};\mu)&=(g,\widetilde{u})
_{\partial \Omega}\quad \forall \widetilde{u} \in \mathscr{H}^{1}_{0}(\Omega).
\end{split}
\end{cases}
\end{equation}
As we can see, $(\ref{full-optimal-system})_{2}$ coincides with the state equation in (\ref{optimal-system1}).
Furthermore, we can obtain the adjoint equation of system (\ref{optimal-system1}) by taking $h=0$ in $(\ref{full-optimal-system})_{1}$. With $v=0$, we can recover the optimality conditions $(\ref{optimal-system1})_{3}$.
Conversely, $(\ref{full-optimal-system})_{1}$ is retrieved by adding the adjoint equation and gradient equation in (\ref{optimal-system1}).
\end{proof}
\begin{thm}
Provided that assumptions in Lemma \ref{equivalence} are satisfied, we can obtain the global existence of a unique solution to the minimization problem (\ref{optimal-system2}). Furthermore, we have the following estimates:
\begin{equation*}
\begin{split}
\|\underline{u}\|_{\mathscr{U}}\leq \gamma_{1}\|\hat{u}\|_{\mathscr{L}^{2}(\Omega)}+\varrho_{1}\|g\|_{\mathscr{L}^{2}(\partial \Omega)},\\
\|\lambda\|_{\mathscr{U}}\leq
\gamma_{2}\|\hat{u}\|_{\mathscr{L}^{2}(\Omega)}+\varrho_{2}\|g\|_{\mathscr{L}^{2}(\partial \Omega)},\\
\end{split}
\end{equation*}
where $\gamma_{1}$, $\gamma_{2}$, $\varrho_{1}$ and $\varrho_{2}$ are all positive constants.
\end{thm}
Due  to the above equivalence, it is sufficient to compute the solution  of
(\ref{optimal-system2}) or (\ref{optimal-system1}) to solve the original optimal control
problem (\ref{cost-functional}-\ref{constraint}).
\section{FE approximation for optimal control problem}
\label{FEM-va}
Before giving the framework of FE approximation for stochastic optimal control problems, we will make an assumption that both the parametric bilinear form $a(\cdot,\cdot;\mu)$ and the parametric
linear form $(f,\cdot;\mu)$ are affine with respect to $\mu$ , i.e.,
\begin{equation}
\label{affine}
\begin{cases}
\begin{split}
a(u,v;\mu) &=\sum\limits^{Q_{a}}_{q=1} Q^{q}_{a}(\mu)a^{q}(u,v) \quad \forall u,v \in H^{1}(\Omega),~\forall \mu \in \Gamma,\\
(f,v;\mu) &=\sum\limits^{Q_{f}}_{q=1} Q^{q}_{f}(\mu)(f^{q},v)   \quad \forall f^{q}\in L^{2}(\Omega),~v \in H^{1}(\Omega),~\forall \mu \in \Gamma.\\
\end{split}
\end{cases}
\end{equation}

In the above, for $q=1,\cdots,Q_{a}$, each $Q^{q}_{a}(\mu): \Gamma \rightarrow \mathbb{R}$ is a $\mu$-dependent
function and $a^{q}(\cdot,\cdot): H^{1}(\Omega)\times H^{1}(\Omega) \rightarrow \mathbb{R}$ is a symmetric bilinear form independent of $\mu$. Similarly, $Q^{q}_{f}(\mu): \Gamma \rightarrow \mathbb{R}, q=1,\cdots,Q_{f},$ are $\mu$-dependent functions and $(f^{q},\cdot): L^{2}(\Omega)\times H^{1}(\Omega) \rightarrow \mathbb{R}, q=1,\cdots,Q_{f}$,  are continuous functionals independent of $\mu$. The affine assumptions $(\ref{affine})$ will play a crucial computational role in the offline-online computational procedure. For the nonaffine parametric bilinear form $a(\cdot,\cdot;\mu)$ and parametric linear form $(f,\cdot;\mu)$, we can use the Empirical Interpolation Method (EIM), see, e.g. \cite{bmna04, gmnp07, qmn15}, to approximate them  with an affine representation.

Let $\mathcal{T}_{h}$ be a uniform partition of the physical domain $\Omega$. Let $V^{h}(\Omega)\subset H^{1}(\Omega)$ be the FE space on the fine grid $\mathcal{T}_{h}$ and $V^{h}_{0}(\Omega) \subset V^{h}(\Omega)$ with vanish boundary values. We let $N_{h}$ be the number of vertices, $N_{e}$ be the number of elements in the fine mesh and the dimension of FE space be $\mathcal{N}$. Furthermore, we assume that $M_{h}(\Omega)$ is the finite dimensional subspace of $L^{2}(\Omega)$ and $\mathcal{M}_{h}(\Omega)$ is also the finite dimensional subspace of $\mathscr{L}^{2}(\Omega)$ with $\mathcal{M}_{h}(\Omega)= M_{h}(\Omega)\otimes L^{2}_{\rho}(\Gamma)$.

Given any $\mu \in \Gamma$, by applying Galerkin projection of $\mathscr{U}_{h}\otimes\mathscr{V}^{h}_{0}(\Omega) \subset \mathscr{U}\otimes\mathscr{H}^{1}_{0}(\Omega)$, where $\mathscr{U}_{h}:= \mathscr{V}^{h}_{0}(\Omega) \otimes \mathcal{M}_{h}(\Omega)$, we obtain FE formulation of the saddle point problem (\ref{optimal-system2}) as: $(\underline{u}_{h}, \lambda_{h}) \in \mathscr{U}_{h}\times \mathscr{V}^{h}_{0}(\Omega)$ such that
\begin{equation}
\label{opt-sys-fem}
\begin{cases}
\begin{split}
\mathcal{A}(\underline{u}_{h},\underline{v}_{h})+\mathcal{B}(\underline{v}_{h},\lambda_{h})
&=(\underline{\hat{u}},\underline{v}_{h})\quad \forall \underline{v}_{h} \in \mathscr{U}_{h},\\
\mathcal{B}(\underline{u}_{h},\widetilde{u}_{h})&=(g,\widetilde{u}_{h})_{\partial \Omega} \quad \forall \widetilde{u}_{h} \in \mathscr{V}^{h}_{0}(\Omega).
\end{split}
\end{cases}
\end{equation}
Mimicking the proofs in subsection \ref{sad-sys}, we can easily show that $\mathcal{A}(\cdot,\cdot)$ is bounded and coercive. Moreover, the bilinear form $\mathcal{B}(\cdot,\cdot)$ satisfies inf-sup condition.

The  Galerkin formulation of (\ref{opt-sys-fem}) is to find $(u_{h},f_{h},\lambda_{h}) \in \mathscr{V}^{h}_{0}(\Omega)\otimes \mathcal{M}_{h}(\Omega)\otimes \mathscr{V}^{h}_{0}(\Omega)$
such that
\begin{equation}
\label{optimal-system-fem}
\begin{cases}
\begin{split}
a(u_{h},\widetilde{u}_{h};\mu)&=(f_{h},\widetilde{u}_{h};\mu) \quad\forall \widetilde{u}_{h} \in \mathscr{V}^{h}_{0}(\Omega),\\
a(\lambda_{h},\widetilde{\lambda}_{h};\mu)&=-(u_{h}-\hat{u},\widetilde{\lambda}_{h};\mu) \quad \forall \widetilde{\lambda}_{h} \in \mathscr{V}^{h}_{0}(\Omega),\\
2\beta(f_{h},\widetilde{f}_{h};\mu)&=(\widetilde{f}_{h},\lambda_{h};\mu) \quad \forall \widetilde{f}_{h} \in \mathcal{M}_{h}(\Omega),
\end{split}
\end{cases}
\end{equation}
where $\mathscr{V}^{h}_{0}(\Omega):= V^{h}_{0}(\Omega)\otimes L^{2}_{\rho}(\Gamma)$.
If the basis functions of spaces $\mathcal{M}_{h}(\Omega)$ and $V^{h}_{0}(\Omega)$ are denoted by
$\{\phi_{k}\}^{N_{e}}_{k=1}$ and $\{\psi_{i}\}^{N_{h}}_{i=1}$,   we can represent the variables in (\ref{optimal-system-fem}) with the linear combination of the corresponding basis functions as
\[
u_{h}=\sum\limits^{N_{h}}_{i=1}u_{h,i}\psi_{i},\quad  f_{h}=\sum\limits^{N_{e}}_{j=1}f_{h,j}\phi_{j}, \quad
\lambda_{h}=\sum\limits^{N_{h}}_{k=1}\lambda_{h,k}\psi_{k}.
\]
Furthermore, if affine assumption (\ref{affine})
holds, the optimality system (\ref{optimal-system-fem}) will be
\begin{equation*}
\begin{cases}
\begin{split}
&\sum\limits^{Q_{a}}_{q=1} \sum\limits^{N_{h}}_{i=1}Q^{q}_{a}(\mu) u_{h,i}(\mu) a^{q}(\psi_{i},\psi_{i'})= \sum\limits^{N_{e}}_{j=1}f_{h,j}(\mu)(\phi_{j},\psi_{i'}),\\
&\sum\limits^{Q_{a}}_{q'=1} \sum\limits^{N_{h}}_{k=1}Q^{q'}_{a}(\mu)\lambda_{h,k}(\mu)
a^{q'}(\psi_{k},\psi_{k'})+\sum\limits^{N_{h}}_{i=1}u_{h,i}(\mu)(\psi_{i},\psi_{k'})=
-\sum\limits^{Q_{u}}_{p=1}\widehat{u}_{p}(\mu)(\overline{\hat{u}}_{p},\psi_{k'}),\\
&2\beta\sum\limits^{N_{e}}_{j=1}\sum\limits^{N_{e}}_{j'=1}f_{h,j}(\mu)f_{h,j'}(\mu)
(\phi_{j},\phi_{j'})=\sum\limits^{N_{h}}_{k=1}\lambda_{h,k}(\mu)(\phi_{j'},\psi_{k}).
\end{split}
\end{cases}
\end{equation*}

Then the algebraic formulation
of (\ref{optimal-system-fem}) reads
\begin{equation}
\label{full-matrix}
\underbrace{\left[
  \begin{array}{ccc}
    2\beta M_{1,h}(\mu) & 0              & -M^{T}_{2,h}(\mu)\\
    0                    & M_{3,h}(\mu)  & K^{T}_{h}(\mu)\\
    -M_{2,h}(\mu)       &K_{h}(\mu)       &0
  \end{array}
\right]}\limits_{\Lambda_{h}(\mu)\in \mathbb{R}^{(2N_{h}+N_{e})\times(2N_{h}+N_{e})}}
\left[
\begin{array}{c}
F(\mu)\\
u(\mu)\\
\lambda(\mu)
\end{array}
\right]
=  \left[
\begin{array}{c}
0\\
\overline{\widehat{U}}\\
d
\end{array}
\right],
\end{equation}
where
\begin{equation}
\label{matrix-def}
\begin{split}
&(M_{1,h})_{j,j'} =(\phi_{j},\phi_{j'}), 1\leq j,j'\leq N_{e},\\
&(M_{2,h})_{k,j'} = (\phi_{j'},\psi_{k}), 1\leq j'\leq N_{e}, 1\leq k\leq N_{h},\\
&(M_{3,h})_{i,k'} =(\psi_{i},\psi_{k'}), (K^{q}_{h})_{k,k'}
=a^{q}(\psi_{k},\psi_{k'}), 1\leq i,k,k'\leq N_{h},\\
&(\widehat{U}_{p})_{k'}=(\overline{\hat{u}}_{p}, \psi_{k'}), 1\leq k'\leq N_{h},\\
&K_{h}=\sum^{Q_{a}}_{q=1}Q^{q}_{a}(\mu)K^{q}_{h}, \overline{\widehat{U}}=\sum^{Q_{u}}_{p=1}\widehat{u}_{p}(\mu)\widehat{U}_{p}.
\end{split}
\end{equation}
Here $u(\mu)$, $F(\mu)$ and $\lambda(\mu)$ denote the vectors of the coefficients in the expansion of $u_{h}(\mu)$, $f_{h}(\mu)$, $\lambda_{h}(\mu)$, respectively. The term coming from the boundary values of $u_{h}$ is denoted by $d$. For simplicity of notation, we have suppressed the size of zero-blocks. In the following, we will change the size of zeros-blocks according to the size of the related  nonzero-blocks.

With the affine assumption (\ref{affine}), stiffness-matrix and mass-matrices are performed only once at the offline stage with expensive computational cost. At the online stage,  for each new parameter sample $\mu \in \Gamma$, all the coefficients $Q^{q}_{a}(\mu)$ and  $\widehat{u}_{p}(\mu)$ are evaluated, and the $(2N_{h}+N_{e})\times (2N_{h}+N_{e})$ linear system (\ref{full-matrix}) is assembled and solved. The online operation count is $O(Q_{a}\mathcal{N}^{2})+ O(Q_{u}\mathcal{N})$ to perform the sum of the last line in (\ref{matrix-def}), and is $O((3\mathcal{N})^{3})$ to invert the matrix $\Lambda_{h}(\mu)$. Thus, the total online operation count to get the FE optimal solutions for each sample $\mu$ is
\[O(Q_{a}\mathcal{N}^{2})+ O(Q_{u}\mathcal{N})+O((3\mathcal{N})^{3}),\]
which depends on the dimension of FE space $\mathcal{N}$.
We define the spaces for the state, control, and adjoint variable, respectively, as
\begin{equation}
\label{FE-space}
\begin{cases}
\begin{split}
X^{\mathcal{N}}_{h}(\Omega)= \text{span} \{u_{h}(\mu), \forall \mu \in \Gamma\},\\
Y^{\mathcal{N}}_{h}(\Omega)= \text{span} \{f_{h}(\mu), \forall \mu \in \Gamma\},\\
Z^{\mathcal{N}}_{h}(\Omega)= \text{span} \{\lambda_{h}(\mu), \forall \mu \in \Gamma\}.
\end{split}
\end{cases}
\end{equation}
\section{The global reduced basis method and the local model reduction method}
\label{sect4}
In Section \ref{sad-sys} and Section \ref{FEM-va}, we can see that both the minimization problems and the equivalent optimality systems are related to the parameter sample $\mu$. This means that we need to compute the optimality system for one time with a given random sample. This is a many-query problem  and the computational cost will be very expensive. To overcome the difficulty, we build  a reduced  model to improve computation efficiency. In this section, we introduce a local-global model reduction method to construct a low-fidelity optimal control model. When performing the optimal process for a new given configuration, the solution will be computed efficiently.

\subsection{The global reduced basis approximation}
\label{glo-RB}
In this section, we follow the framework \cite{hrsp07, jl16, jl17,  qrm11, rhp08} to present the global RB method for solving stochastic optimal control problems. The essential components of the RB method: RB projection and optimality system.   An offline-online computational stratagem will be presented.

\subsubsection{Construction of global reduced basis approximation spaces and Galerkin projection}
\label{RB-space}
 We assume that we have been given FE approximation spaces $V^{h}(\Omega)$, $M_{h}(\Omega)$ in a fine grid with (typically very large) dimensions $\mathcal{N}$ and $N_{e}$. The spaces $\mathscr{V}_{h}(\Omega)$ and $\mathcal{M}_{h}(\Omega)$ are automatically given.
 Given a positive integer $N_{max}$, we then introduce an associated sequence of (what shall ultimately be reduced basis) approximation spaces: for $N=1, \cdots, N_{max}$, $X^{N}_{h}(\Omega)$ is an N-dimensional subspace of $X^{\mathcal{N}}_{h}(\Omega)$. Let $\{X^{n}_{h}\}^{N_{max}}_{n=1}$ be nested, that is, $X^{1}_{h} \subset X^{2}_{h} \subset \cdots \subset X^{N_{max}}_{h} \subset X^{\mathcal{N}}_{h}(\Omega)\subset \mathscr{V}^{h}(\Omega)$. For the control variable $f$, we also define the similar subspaces: $Y^{1}_{h} \subset Y^{2}_{h} \subset \cdots \subset Y^{N_{max}}_{h} \subset Y^{\mathcal{N}}_{h}(\Omega)\subset \mathcal{M}_{h}(\Omega)$. For the adjoint variable $\lambda$, the set of nested subspaces is: $Z^{1}_{h} \subset Z^{2}_{h} \subset \cdots \subset Z^{N_{max}}_{h} \subset Z^{\mathcal{N}}_{h}(\Omega)\subset \mathscr{V}^{h}(\Omega)$. These nested subspaces are crucial  in ensuring (memory) efficiency of the resulting RB approximation.

In order to define a  sequence of Lagrange spaces $\{X^{N}_{h}(\Omega), Y^{N}_{h}(\Omega), Z^{N}_{h}(\Omega), 1\leq N\leq N_{max}\}$, we first use
\[S_{N}=\{\mu^{1},\cdots,\mu^{N},1\leq N\leq N_{max}\}\]
to denote the set of properly selected parameter samples from $\Gamma$. The associated Lagrange RB spaces are
\begin{equation}
\label{rb-space}
\begin{split}
\begin{cases}
X^{N}_{h}(\Omega)=\text{span} \{u_{h}(\mu^{n}), 1\leq n \leq N\},\\
Y^{N}_{h}(\Omega)=\text{span} \{f_{h}(\mu^{n}), 1\leq n \leq N\},\\
Z^{N}_{h}(\Omega)=\text{span} \{\lambda_{h}(\mu^{n}), 1\leq n \leq N\},
\end{cases}
\end{split}
\end{equation}
and denote $\mathcal{U}^{N}_{h}(\Omega)= X^{N}_{h}(\Omega) \times Y^{N}_{h}(\Omega)$.

The $u_{h}(\mu^{n}), f_{h}(\mu^{n}), \lambda_{h}(\mu^{n})$, $1\leq n \leq N_{max}$, are often referred to as ``snapshots" of the parametric manifolds and are obtained by solving (\ref{optimal-system-fem}) for $\mu^{n}$, $1\leq n\leq N_{max}$. The choice of the parameters $\{\mu^{n}\}^{N_{max}}_{n=1}$ will effect on  the fidelity of the reduced model to approximate the original model. To  choose the set of optimal parameter samples, we will use a sampling strategy, greedy algorithm,  in Section \ref{greedy-sample}.

By using  Galerkin projection onto the low-dimensional subspace $\mathscr{U}^{N}_{h}(\Omega)\times Z^{N}_{h}(\Omega)$, the following RB approximation can be obtained: given $\mu \in \Gamma$, find $(\underline{u}^{N}_{g}, \lambda^{N}_{g}) \in \mathscr{U}^{N}_{h}(\Omega)\times Z^{N}_{h}(\Omega)$ such that
\begin{equation}
\label{opt-sys-grb}
\begin{cases}
\begin{split}
\mathcal{A}(\underline{u}_{g},\underline{v}_{g})+\mathcal{B}(\underline{v}_{g},\lambda_{g})
&=(\underline{\hat{u}},\underline{v}_{g})\quad \forall \underline{v}_{g} \in \mathscr{U}^{N}_{h},\\
\mathcal{B}(\underline{u}_{g},\widetilde{u}_{g})&=(g,\widetilde{u}_{g})_{\partial \Omega} \quad \forall \widetilde{u}_{g} \in Z^{N}_{h}(\Omega).
\end{split}
\end{cases}
\end{equation}

Next, we discuss the well-posedness of the global RB approximation (\ref{opt-sys-grb}).
The continuity  of the bilinear forms is automatically inherited from the FE spaces.
Similarly, with the kernel space of bilinear form $\mathcal{B}(\cdot,\cdot;\mu)$ as $\mathscr{U}^{N}_{0}:= \{ \underline{u}_{g}\in \mathscr{U}^{N}_{h}: \mathcal{B}(\underline{u}_{g}, \widetilde{u}_{g})=0~~\forall \widetilde{u}_{g}\in Z^{N}_{h}(\Omega)\}$, the coercivity of bilinear form $\mathcal{A}(\cdot,\cdot;\mu)$ can be guaranteed. In the continuous case (or approximated by FE method), both the state and adjoint spaces belong to  $\mathscr{H}^{1}(\Omega)$ (or $\mathscr{V}^{h}(\Omega)$) and the bilinear form $\mathcal{B}(\cdot,\cdot;\mu)$ satisfies the inf-sup condition. But with the choice (\ref{rb-space}), we lose this property on the  Lagrange RB spaces, i.e., $X^{N}_{h}(\Omega)\neq Z^{N}_{h}(\Omega)$.

In order to recovery the inf-sup condition for system (\ref{opt-sys-grb}), we therefore need to enrich in some way
at least one of the RB spaces involved. With the method considered in some previous works \cite{ld10, nrmq13},
we use an enriched RB space $Q^{N}_{h}(\Omega)$ as the union $X^{N}_{h}(\Omega)$ and $Z^{N}_{h}(\Omega)$, i.e.,
\begin{equation}
\label{enrich-space}
Q^{N}_{h}(\Omega) = X^{N}_{h}(\Omega) \cup Z^{N}_{h}(\Omega) = \text{span} \{u_{h}(\mu^{n}), \lambda_{h}(\mu^{n}), 1\leq n \leq N\},
\end{equation}
and we let
\begin{equation}
\label{new-space}
X^{N}_{h}(\Omega)=Q^{N}_{h}(\Omega),\quad \mathcal{U}^{N}_{h}(\Omega)=X^{N}_{h}(\Omega)\times Y^{N}_{h}(\Omega),\quad Z^{N}_{h}(\Omega)=Q^{N}_{h}(\Omega).
\end{equation}
\begin{lem}
With the definition of reduced spaces $\mathcal{U}^{N}_{h}(\Omega)$ and $Z^{N}_{h}(\Omega)$ in (\ref{enrich-space})-(\ref{new-space}), then the bilinear form $\mathcal{B}(\cdot,\cdot;\mu)$ satisfies the inf-sup
condition. Moreover, for the inf-sup value $\beta^{N}(\mu)$, we have
\[\beta^{N}(\mu)\geq \tilde{\alpha}^{\mathcal{N}}(\mu),\]
where $\tilde{\alpha}^{\mathcal{N}}(\mu)$ is the coercivity constant of the bilinear form $\mathcal{A}(\cdot,\cdot;\mu)$ with the FE approximation.
\end{lem}
\begin{proof}
 In fact,
\begin{equation*}
\begin{split}
\sup\limits_{0\neq\underline{u}_{g}\in \mathcal{U}^{N}_{h}(\Omega)}\frac{\mathcal{B}(\underline{u}_{g},\tilde{u}_{g};\mu)}
{\|\underline{u}_{g}\|_{\mathcal{U}^{N}_{h}(\Omega)}}
&=\sup\limits_{0\neq (u_{g},h_{g})\in X^{N}_{h}(\Omega)\times Y^{N}_{h}(\Omega)}\frac{a(u_{g},\tilde{u}_{g};\mu)-(h_{g},\tilde{u}_{g};\mu)}
{\|u_{g}\|_{X^{N}_{h}(\Omega)}+\|h_{g}\|_{Y^{N}_{h}(\Omega)}}\\
&\geq \sup\limits_{0\neq (u_{g},0)\in X^{N}_{h}(\Omega)\times Y^{N}_{h}(\Omega)}\frac{a(u_{g},\tilde{u}_{g};\mu)}
{\|u_{g}\|_{X^{N}_{h}(\Omega)}}\\
&\geq \tilde{\alpha}^{N}(\mu)\|\tilde{u}_{g}\|_{Z^{N}_{h}(\Omega)}.
\end{split}
\end{equation*}
Here $\tilde{\alpha}^{N}(\mu)$ represents the coercivity constant of the bilinear form $\mathcal{A}(\cdot,\cdot;\mu)$ with the FE approximation. Thus, we complete the proof.
\end{proof}

\subsubsection{Algebraic formulation and offline-online computation}
\label{offline-online}
Given the spaces $X^{N}_{h}(\Omega), Y^{N}_{h}(\Omega), Z^{N}_{h}(\Omega)$, the associated optimality system
is: given $\mu \in \Gamma$, find $u^{N}_{h} \in X^{N}_{h}(\Omega), f^{N}_{h} \in Y^{N}_{h}(\Omega), \lambda^{N}_{h}\in Z^{N}_{h}(\Omega)$ such that
\begin{equation}
\label{optimal-system-rb}
\begin{cases}
\begin{split}
a(u^{N}_{h},\widetilde{u}'_{h};\mu)&=(f^{N}_{h},\widetilde{u}'_{h};\mu) ~\forall \widetilde{u}'_{h} \in X^{N}_{h}(\Omega),\\
a(\lambda^{N}_{h},\widetilde{\lambda}'_{h};\mu)&=-(u^{N}_{h}-\hat{u},\widetilde{\lambda}'_{h};\mu) ~\forall \widetilde{\lambda}'_{h} \in Z^{N}_{h}(\Omega),\\
2\beta(f^{N}_{h},\widetilde{f}'_{h};\mu)&=(\widetilde{f}'_{h},\lambda^{N}_{h};\mu) ~\forall \widetilde{f}'_{h} \in Y^{N}_{h}(\Omega).
\end{split}
\end{cases}
\end{equation}
For the sake of algebraic stability in assembling the RB matrices and performing Galerkin projection \cite{rhp08}, we orthonormalize the snapshots in the RB spaces $X^{N}_{h}(\Omega)$, $Y^{N}_{h}(\Omega)$ and $Z^{N}_{h}(\Omega)$ by the Gram--Schmidt process with respect to the $L^{2}$-inner products, yielding
\[ X^{N}_{h}(\Omega)=\{\delta_{i}, 1\leq i\leq 2N\},
\quad Y^{N}_{h}(\Omega)=\{\xi_{j}, 1\leq j \leq N\},
\quad Z^{N}_{h}(\Omega)=\{\delta_{i}, 1\leq i\leq 2N\}.
\]
So the solutions $(u^{N}_{h},f^{N}_{h},\lambda^{N}_{h})$ can be represented  by
\begin{equation}
\label{Rb-repre}
u^{N}_{h} = \sum\limits^{2N}_{i=1} u^{N}_{i}(\mu) \delta_{i},~~
f^{N}_{h} = \sum\limits^{N}_{j=1} f^{N}_{j}(\mu) \xi_{j},~~
\lambda^{N}_{h} = \sum\limits^{2N}_{k=1} \lambda^{N}_{k}(\mu) \delta_{k}.
\end{equation}
By plugging $\widetilde{u}'_{h}=\delta_{i'}, \widetilde{f}'_{h}=\xi_{j'}, \widetilde{\lambda}'_{h}=\delta_{k'}$
into model (\ref{optimal-system-rb}), we  get
\begin{equation}
\label{os-dis1}
\begin{cases}
\begin{split}
&\sum\limits^{2N}_{i=1}u^{N}_{i}(\mu)a(\delta_{i},\delta_{i'};\mu)
=\sum\limits^{N}_{j=1}f^{N}_{j}(\mu)(\xi_{j},\delta_{i'}),\\
&\sum\limits^{2N}_{k=1}\lambda^{N}_{k}(\mu)a(\delta_{k},\delta_{k'};\mu)+
\sum\limits^{2N}_{i=1}u^{N}_{i}(\mu)(\delta_{i},\delta_{k'})=(\hat{u},\delta_{k'};\mu),\\
&2\beta\sum\limits^{N}_{j=1}f^{N}_{j}(\mu)(\xi_{j},\xi_{j'})
=\sum\limits^{2N}_{k=1} \lambda^{N}_{k}(\mu) (\xi_{j'},\delta_{k}).
\end{split}
\end{cases}
\end{equation}

The system (\ref{os-dis1}) implies a linear algebraic system with $5N$ unknowns. The stiffness matrix, mass matrices
and load vector from system (\ref{os-dis1}) involve the computation of inner products $a(\delta_{i},\delta_{i'})$, $(\delta_{i},\delta_{k'})$, $(\xi_{j},\xi_{j'})$, $(\xi_{j},\delta_{i'})$ and $(\hat{u}, \delta_{k'})$. For a new parameter sample $\mu\in \Gamma$, we need to compute the matrices and the load vector for one time. When the number  of parameter samples is large, the computation of the system will be expensive. If the assumption (\ref{affine}) of affine assumption holds, the system will be
\begin{equation}
\label{os-dis}
\begin{cases}
\begin{split}
&\sum\limits^{Q_{a}}_{q=1}\sum\limits^{2N}_{i=1}u^{N}_{i}(\mu)Q^{q}_{a}(\mu)a^{q}(\delta_{i},\delta_{i'})
=\sum\limits^{N}_{j=1}f^{N}_{j}(\mu)(\xi_{j},\delta_{i'}),\\
&\sum\limits^{Q_{a}}_{q'=1}\sum\limits^{2N}_{k=1}\lambda^{N}_{k}(\mu) Q^{q'}_{a}(\mu)a^{q'}(\delta_{k},\delta_{k'})+
\sum\limits^{2N}_{i=1}u^{N}_{i}(\mu)(\delta_{i},\delta_{k'})=
\sum\limits^{Q_{u}}_{p=1}\widehat{u}_{p}(\mu)(\overline{\widehat{u}}_{p},\delta_{k'}),\\
&2\beta\sum\limits^{N}_{j=1}f^{N}_{j}(\mu)(\xi_{j},\xi_{j'})
=\sum\limits^{2N}_{k=1} \lambda^{N}_{k}(\mu) (\xi_{j'},\delta_{k}).
\end{split}
\end{cases}
\end{equation}
Because basis function $\delta_{i}$ belongs to the FEM space $V^{h}_{0}(\Omega)$, it can be written as
\[\delta_{i}= \sum^{M}_{k=1}(\mathcal{Z}_{1})_{i,k}\psi_{k},~ i=1,\cdots, 2N.\]
Similarly, with the piecewise constant basis of $\mathcal{M}_{h}(\Omega)$ as $\{\phi_{j}\}^{N_{f}}_{j=1}$,
we can obtain
\[\xi_{i}= \sum^{N_{f}}_{s=1}(\mathcal{Z}_{2})_{i,s}\phi_{s},~ i=1,\cdots, N.\]
The linear system (\ref{os-dis}) will become
\begin{equation}
\label{aff-rep-glo}
\begin{cases}
\begin{split}
&\sum\limits^{Q_{a}}_{q=1}\sum\limits^{2N}_{i=1}\sum\limits^{M}_{\ell=1}\sum\limits^{M}_{\ell'=1}u^{N}_{i}(\mu)
Q^{q}_{a}(\mu)(\mathcal{Z}_{1})_{i,\ell}(\mathcal{Z}_{1})_{i',\ell'}a^{q}(\psi_{i,\ell},\psi_{i',\ell'})\\
&=\sum\limits^{N}_{j=1}\sum\limits^{M}_{\ell'=1}\sum\limits^{N_{f}}_{s=1}(\mathcal{Z}_{1})_{i',\ell'}
(\mathcal{Z}_{2})_{j,s}f^{N}_{j}(\mu)(\phi_{j,s},\psi_{i',\ell'}),\\
&\sum\limits^{Q_{a}}_{q'=1}\sum\limits^{2N}_{k=1}\sum\limits^{M}_{\ell=1}\sum\limits^{M}_{\ell'=1}
\lambda^{N}_{k}(\mu)Q^{q'}_{a}(\mu)(\mathcal{Z}_{1})_{k,\ell}(\mathcal{Z}_{1})_{k',\ell'}a^{q'}
(\psi_{k,\ell},\psi_{k',\ell'})\\
&+\sum\limits^{2N}_{i=1}\sum\limits^{M}_{\ell=1}\sum\limits^{M}_{\ell'=1}u^{N}_{i}(\mu)(\mathcal{Z}_{1})_{i,\ell}
(\mathcal{Z}_{1})_{k',\ell'}(\psi_{i,\ell},\psi_{k',\ell'})=\sum\limits^{Q_{u}}_{p=1}\sum\limits^{M}_{\ell'=1}
\widehat{u}_{p}(\mu)(\overline{\widehat{u}}_{p},\psi_{k',\ell'}),\\
&2\beta\sum\limits^{N}_{j=1}\sum\limits^{N_{f}}_{s=1}\sum\limits^{N_{f}}_{s'=1}(\mathcal{Z}_{2})_{j,s}
(\mathcal{Z}_{2})_{j',s'}f^{N}_{j}(\mu)(\phi_{s},\phi_{s'})
=\sum\limits^{2N}_{k=1}\sum\limits^{N_{f}}_{s'=1}\sum\limits^{M}_{\ell=1} (\mathcal{Z}_{2})_{j',s'}(\mathcal{Z}_{1})_{k,\ell}\lambda^{N}_{k}(\mu) (\phi_{j'},\psi_{k}).
\end{split}
\end{cases}
\end{equation}
This gives rise to the matrix form
\begin{equation}
\label{matrix-sep}
\begin{cases}
\begin{split}
&\sum\limits^{Q_{a}}_{q=1} Q^{q}_{a}(\mu)(\mathcal{Z}^{T}_{1}K^{q}\mathcal{Z}_{1})\mathbf{u}^{N}(\mu)
= \mathcal{Z}^{T}_{1}M^{T}_{2}\mathcal{Z}_{2} \mathbf{F}^{N}(\mu),\\
&\sum\limits^{Q_{a}}_{q'=1} Q^{q'}_{a}(\mu)(\mathcal{Z}^{T}_{1}{K}^{q'}\mathcal{Z}_{1})\pmb{\lambda}^{N}(\mu)
+\mathcal{Z}^{T}_{1} M_{3}\mathcal{Z}_{1} \mathbf{u}^{N}(\mu)
=\sum\limits^{Q_{u}}_{p=1}\mathbf{\widehat{u}_{p}}(\mu)(\mathcal{Z}^{T}_{1}\overline{\widehat{U}}_{p}),\\
&2\beta \mathcal{Z}^{T}_{2}M_{1}\mathcal{Z}^{T}_{2} \mathbf{F}^{N}(\mu) = \mathcal{Z}^{T}_{2}M_{2} \mathcal{Z}_{1} \pmb{\lambda}^{N}(\mu),
\end{split}
\end{cases}
\end{equation}
where $(K^{q})_{i,i'}=a^{q}(\delta_{i},\delta_{i'})$,
$(M_{1})_{j,j'}=(\xi_{j},\xi_{j'})$, $(M_{2})_{j,i'}=(\xi_{j},\delta_{i'})$,
$(M_{3})_{i,k'}=(\delta_{i},\delta_{k'})$.
Thanks to the affine assumption, these inner products are computed for only one time in the offline phase.
For a given parameter sample $\mu$, we only need to update the coefficients $Q^{q}_{a}(\mu)$ and $\mathbf{\widehat{u}_{p}}(\mu)$ in the online phase.
With the notations
\begin{equation*}
\begin{cases}
\begin{split}
&K_{g}(\mu)=\sum\limits^{Q_{a}}_{q=1} Q^{q}_{a}(\mu)(\mathcal{Z}^{T}_{1}K^{q})\mathcal{Z}_{1},\\
&\overline{\widehat{U}}_{g}=\sum\limits^{Q_{u}}_{p=1} \mathbf{\widehat{u}_{p}}(\mu)(\mathcal{Z}^{T}_{1}\overline{\widehat{U}}_{p}),\\
&M_{1,g}=\mathcal{Z}^{T}_{2}M_{1}\mathcal{Z}^{T}_{2}, M_{2,g}=\mathcal{Z}^{T}_{1}M^{T}_{2}\mathcal{Z}_{2},\\ &M_{3,g}=\mathcal{Z}^{T}_{1} M_{3}\mathcal{Z}_{1},
\end{split}
\end{cases}
\end{equation*}
 we can further get
\begin{equation}
\label{matrix}
\underbrace{\left[
  \begin{array}{ccc}
    2\beta M_{1,g}(\mu) & 0              & -M^{T}_{2,g}(\mu)\\
    0                    & M_{3,g}(\mu)  & K^{T}_{g}(\mu)\\
    -M_{2,g}(\mu)       &K_{g}(\mu)       &0
  \end{array}
\right]}_{\Lambda_{g}(\mu)\in \mathbb{R}^{5N\times5N}}
\left[
\begin{array}{c}
\vec{F}_{g}(\mu)\\
\vec{u}_{g}(\mu)\\
\vec{\lambda}_{g}(\mu)
\end{array}
\right]
=  \left[
\begin{array}{c}
0\\
\overline{\widehat{U}}_{g}\\
d_{g}
\end{array}
\right],
\end{equation}
where $d_{g}$ represents the corresponding boundary part with $d_{g}=K_{g,bd} \vec{u}_{g}(\mu)$.

Similarly, the online operation count depends on $N$, $Q^{a}$, $Q^{u}$ but independent of $\mathcal{N}$. At the online
stage, we need $O(Q^{a}N^{2})+ O(Q_{u}N)$ operations to assemble the matrix $\Lambda_{g}(\mu)$, right vector in
(\ref{matrix}) and $O((5N)^{3})$ to invert the matrix. Thus, for a new random sample, the computation complexity is
\[O(Q^{a}N^{2})+ O(Q_{u}N)+O((5N)^{3})\]
to get the global RB optimal solutions.  Due to the hierarchical condition among $\{X^{h}_{n}\}^{Nmax}_{n=1}$,  the online storage is  only
\[O(Q^{a}N^{2}_{max})+O(Q^{u}N_{max}).\]
The online computation cost (operation count and storage) to evaluate $\mu\rightarrow (u_{g}(\mu), f_{g}(\mu), \lambda_{g}(\mu))$ is thus independent of $\mathcal{N}$. When $N\ll \mathcal{N}$, we can get the solutions very rapidly.

If we defined the matrix $\mathcal{Z}$ as
\begin{equation}
\label{matrix-rb}
\mathcal{Z}=\left[
  \begin{array}{ccc}
    \mathcal{Z}^{T}_{2} & 0               & 0\\
    0               &\mathcal{Z}^{T}_{1}  & 0\\
    0               &0                &\mathcal{Z}^{T}_{1}
  \end{array}
  \right]\in \mathbb{R}^{5N\times(N_{e}+2N_{h})},
\end{equation}
the reduced optimality matrix can be obtained by $\Lambda_{g}(\mu)=\mathcal{Z}\Lambda_{h}(\mu)\mathcal{Z}^{T}$.

Then we downscale the RB solution to fine scale solution by using RB functions with
\[u_{g}(\mu)=\mathcal{Z}^{T}_{1} \vec{u}_{g}(\mu), \quad f_{g}(\mu)=\mathcal{Z}^{T}_{2} \vec{F}_{g}(\mu), \quad \lambda_{g}(\mu)=\mathcal{Z}^{T}_{1} \vec{\lambda}_{g}(\mu).\]

\subsection{Local model order reduction via GMsFEM}
\label{gmsfem}
To construct the Lagrange RB spaces in Section \ref{RB-space}, we need to compute the system (\ref{optimal-system-fem}) for many training samples. In general, the number of the training  samples is large and it will require a demanding computational cost. Furthermore, the constraint in the optimal problem may be PDEs with multiscale structure. To overcome the difficulty, we will construct a local reduced  model with suitable fidility. As a local model reduction approach, generalized multiscale finite element method (GMsFEM)  is one of efficient model reduction methods. In this subsection, we will give a brief overview of the local model reduction using  GMsFEM. For motivation and details regarding GMsFEM, we refer the readers to \cite{bo83} and the references therein.

First, we set the scene of the numerical discretization, as demonstrated in Fig.~\ref{mesh-GMs}.
\begin{figure}[hbtp]  \centering
  \includegraphics[width=0.4\textwidth]{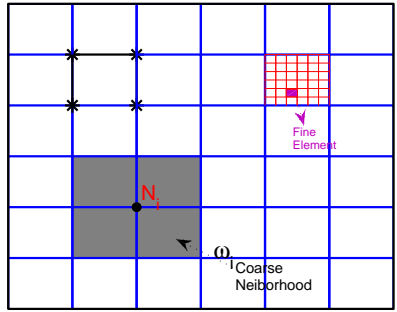}
  \caption{\em{Illustration of the discretization configuration. The computational domain $\Omega$ is equipped with a coarse mesh partition (the element marked by $\ast$) and a fine mesh partition (depicted by the red lines for a  coarse element). The gray region illustrates the neighborhood $\omega_{i}$ associated with the coarse node $N_{i}$.}}
  \label{mesh-GMs}
\end{figure}
We assume that the computational domain $\Omega$ is partitioned uniformly by a coarse mesh $\mathcal{T}_{H}$ and a fine mesh $\mathcal{T}_{h}$, with mesh size $H$ and $h$, respectively. Moreover, $\mathcal{T}_{h}$ is obtained by refining the coarse mesh $\mathcal{T}_{H}$. The nodes of the the coarse mesh are denoted by $\{N_{i}\}^{\mathcal{N}_{c}}_{i=1}$, where
$\mathcal{N}_{c}$ represents the number of coarse nodes. The neighborhood $\omega_{i}$ of the node $N_{i}$ consists of
all the coarse mesh elements for which node $N_{i}$ is a vertex, i.e.,
\[\omega_{i} = \cup \{K_{s} \in \mathcal{T}_{H} | x_{i} \in \overline {K_{s}}\}.\]

Generalized multiscle finite element method uses two stages: offline stage and online stage. At the offline stage of GMsFEM, we firstly construct the space of ``snapshots", $V^{(i)}_{snap}$, and then introduce the construction of local reduced basis functions. For the snapshot space, we can construct it by various ways \cite{ceh16}: (1) all fine-grid functions; (2) harmonic snapshots; (3) oversampling harmonic snapshots; and (4) forced-based snapshots. In this paper, we adopt the second choice to form a snapshot space. For each fine-grid
function $\delta^{h}_{j}(x)$, which is defined by $\delta^{h}_{j}(x_{k})=\delta_{j,k}$, $\forall j,k \in J_{h}(\omega_{i})$, where $J_{h}(\omega_{i})$ denotes the fine-grid nodes on the boundary $\partial \omega_{i}$. We obtain a snapshot function $\zeta^{\omega_{i}}_{j}(x)$ by solving
\[\mathcal{L}(\zeta^{\omega_{i}}_{j}(x))=0 ~\text{in} ~\omega_{i} \]
with the boundary condition, $\zeta^{\omega_{i}}_{j}(x)= \delta^{h}_{j}(x)$ on $\partial \omega_{i}$, and $\delta_{j,k}=1$ if $j=k$ and $\delta_{j,k}=0$ if $j\neq k$. Here, $\mathcal{L}$ is the differential operator corresponding to the state equation and adjoint equation in (\ref{optimal-system1}).
For brevity of notation we now omit the superscript $\omega_{i}$, yet it is assumed throughout this section that the local reduced space computations are localized to respective coarse subdomains.
We let $l_{i}$ be the number of functions in the snapshot space, and
\[V_{snap}(\omega_{i}) =\text{span} \{\zeta_{j}(x), 1\leq j\leq l_{i}\},\]
for each subregion $\omega_{i}$.
Components of $V_{snap}$ are linear combinations of the fine-grid basis functions with coefficients stored in $R_{snap}$, i.e.,
\[R_{snap}=[\zeta_{1},\zeta_{2},\cdots,\zeta_{l_{i}}.]\]

After obtaining the snapshot space $V_{snap}$, we move on to the construction of the local reduced space $V_{lr}(\Omega)$ following the similar procedure. We will solve local eigenvalue problem in the snapshot space $V_{snap}(\omega_{i})$. Given the neighborhood $\omega_{i}$, the local eigenvalue problem is defined by
\begin{equation}
\label{m-eig}
\begin{split}
 A^{lr} \phi^{lr}_{i,\ell} = \lambda_{\ell} S^{lr} \phi^{lr}_{i,\ell}.
 \end{split}
\end{equation}
If the state equation is a diffusion equation with diffusion coefficient $\kappa(x)$, we specify the matrixes in (\ref{m-eig}) by
 \[A^{lr}:=[a^{lr}_{mn}]=\int_{\omega_{i}}\kappa(x) \nabla \zeta_{m}\nabla \zeta_{n}dx, \quad
S^{lr}:=[s^{lr}_{mn}]=\int_{\omega_{i}}\kappa(x) \zeta_{m} \zeta_{n}dx.
\]
To accelerate the procedure of solving the local eigenvalue problem, we can replace the coefficient $\kappa(x)$ by $\widetilde{\kappa(x)}=\sum\limits^{\mathcal{N}_{c}}_{i=1}H^{2}|\nabla \chi_{i}|^{2}$ and $\{\chi_{i}\}$ is a set of partition of unity functions  \cite{bbo04, bl11} corresponding to the grid the grid $\mathcal{T}_{H}$.

Let $\{L_{p}(x)\}$ be the standard basis functions defined on the fine mesh, which belong to the FE approximation space $V_{h}(\Omega)$. Since the snapshot functions in $V_{snap}(\omega_{i})$ can be represented by the linear combinations of standard basis functions, $A^{lr}$ and $S^{lr}$ can be written as
\begin{equation*}
A^{lr}=(R_{snap})^{T}A R_{snap} \quad and \quad S^{lr}=(R_{snap})^{T}S R_{snap},
\end{equation*}
respectively.
Here $A$ and $S$ are the counterparts of $A^{lr}$ and $S^{lr}$ built with the fine-grid basis functions with the following expressions:
\begin{equation*}
A:=[a_{pq}]=\int_{\omega_{i}}\kappa(x) \nabla L_{p}\nabla L_{q}dx \quad and \quad
S:=[s_{pq}]=\int_{\omega_{i}}\kappa(x) L_{p} L_{q}dx.
\end{equation*}

We then let $\lambda^{i}_{1}\leq \lambda^{i}_{2} \leq \cdots$ be the eigenvalues and let $\phi^{lr}_{i,1}, \phi^{lr}_{i,2}, \cdots$ be the corresponding eigenvectors. Finally, we downscale the eigenvectors generated by snapshot functions to local fine scale solution by using basis functions in $V_{snap}$ and denote the eigenfunctions by $\{\rho_{i,\ell}\}$ on local fine grid $\omega_{i}$.

We note that $\{\zeta_{m}(x)\}$ (or $\{L_{m}(x)\}$) are functions of the spatial variables $x$,  whereas $\{\phi^{lr}_{i,\ell}\}$ (or $\{\rho_{i,\ell}\}$) are discrete vectors. The linear combination of the snapshot functions (or the fine grid basis functions), i.e., $\zeta^{i,\ell}$ (or $L^{i,\ell}$ ), with $\phi^{lr}_{i,\ell}$ (or $\rho_{i,\ell}$) as the coefficient vector can be understood as the eigenfunction of the continuous problem corresponding to (\ref{m-eig}).

The number of eigenvectors that satisfy (\ref{m-eig}) is the same as the number of the fine-grid basis functions defined
on the neighborhood $\omega_{i}$. However, we only retain a few of them that correspond to the smallest eigenvalues.
We choose the $M_{i}$ lowermost eigenvalues and the corresponding eigenvectors of the eigenvalue problem (\ref{m-eig}),
i.e. the eigenvalues and eigenvectors denoted by
$\{\lambda^{(i)}_{\ell}\}^{M_{i}}_{\ell=1}$ and $\{\rho^{(i)}_{\ell}\}^{M_{i}}_{\ell=1}$.
The local reduced space $V^{(i)}_{lr} := \{\rho^{(i)}_{\ell}, 1 \leq \ell \leq M_{i}\}$ on the local region $\omega_{i}$.
We  use partition of unity functions $\{\chi_{i}\}$ to paste the snapshot functions and get the multiscale basis function space
\[
V_{lr}(\Omega) := \text{span} \{\rho_{k}: 1\leq k\leq M\} = \text{span} \{\rho_{i,\ell}: \rho_{i,\ell} =
\chi_{i} \rho^{(i)}_{\ell}, 1\leq i\leq \mathcal{N}_{c}, 1\leq \ell \leq M_{i}\},
\]
where $M = \sum\limits^{\mathcal{N}_{c}}_{i=1} M_{i}$ is the total number of eigenvectors for reduced space.
Components of $V_{lr}$ are linear combinations of the fine-grid basis functions with coefficients stored
in $R^{l}$, i.e.,
\begin{equation}
\label{convert-matrix}
R^{l} = [\rho_{1}, \rho_{2}, \cdots, \rho_{M}].
\end{equation}
At the online stage, the multiscale basis functions can be repeatedly used.
Compared with direct numerical simulation on fine grid, GMsFEM can significantly improve the computation efficiency.

 We note that GMsFEM is only applied to the approximation space for  state variable and adjoint variable. While for control variable,
  we do not use GMsFEM approximation.

\section{Local-global model reduction method}
\label{sect5}
In this section, we will discuss the construction of the local-global model reduction framework and present  a greedy algorithm to get  the optimal samples from a training set.

\subsection{Local-global model reduction method for optimality system}

\label{lg-framework}
We apply local model reduction method into the optimality system (\ref{optimal-system-fem}), i.e., finding $u_{l}(\mu)\in \mathscr{V}_{lr}(\Omega)$, $f_{l}(\mu)\in \mathcal{M}_{h}(\Omega)$, $\lambda_{l}(\mu) \in \mathscr{V}_{lr}(\Omega)$ such that
\begin{equation}
\label{optimal-system-gms}
\begin{cases}
\begin{split}
a(u_{l},\widetilde{u}_{l};\mu)&=(f_{l},\widetilde{u}_{l};\mu) ~\forall \widetilde{u}_{l} \in \mathscr{V}_{lr}(\Omega),\\
a(\lambda_{l},\widetilde{\lambda}_{l};\mu)&=-(u_{l}-\hat{u},\widetilde{\lambda}_{l};\mu) ~\forall \widetilde{\lambda}_{l} \in \mathscr{V}_{lr}(\Omega),\\
2\beta(f_{l},\widetilde{f}_{l};\mu)&=(\widetilde{f}_{l},\lambda_{l};\mu) ~\forall \widetilde{f}_{l} \in \mathcal{M}_{h}(\Omega).
\end{split}
\end{cases}
\end{equation}
Here $\mathscr{V}_{lr}(\Omega):=L^{2}_{\rho}(\Gamma)\otimes V_{lr}(\Omega) \subset
\mathscr{V}_{h}(\Omega)$.

In the matrix notation, the reduced optimality matrix corresponding to (\ref{optimal-system-gms}) is defined by
\begin{equation}
\underbrace{\label{full-matrix}
\left[
  \begin{array}{ccc}
    2\beta M_{1}(\mu)        & 0                              & -M^{T}_{2}(\mu)R^{l}\\
    0                           & (R^{l})^{T}M_{3}(\mu)R^{l}  & (R^{l})^{T}K^{T}(\mu)R^{l}\\
    -(R^{l})^{T}M_{2}(\mu)   &(R^{l})^{T}K(\mu)R^{l}       &0
  \end{array}
\right]}\limits_{\Lambda_{l}(\mu)\in \mathbb{R}^{(N_{e}+2M)\times (N_{e}+2M)}}
\left[
\begin{array}{c}
\vec{F}_{l}(\mu)\\
\vec{u}_{l}(\mu)\\
\vec{\lambda}_{l}(\mu)
\end{array}
\right]
=  \left[
\begin{array}{c}
0\\
(R^{l})^{T}\widehat{U}\\
(R^{l})^{T}d
\end{array}
\right].
\end{equation}
Here if we set
\begin{equation*}
R=\left[
  \begin{array}{ccc}
    I     &0           &0\\
    0 &(R^{l})^{T}     &0\\
    0     &0     &(R^{l})^{T}\\
  \end{array}
\right] \in \mathbb{R}^{(N_{e}+2M)\times(N_{e}+2N_{h})},
\end{equation*}
then  $\Lambda_{l}(\mu)=R \Lambda_{l}(\mu) R^{T}$. Thus,
\[
u_{l}(\mu)= (R^{l})^{T} \vec{u}_{l}(\mu), \quad  f_{l}(\mu)= (R^{l})^{T} \vec{F}_{l}(\mu), \quad \lambda_{l}(\mu)= (R^{l})^{T} \vec{\lambda}_{l}(\mu).
\]

Similarly in Section \ref{glo-RB}, we can define the local-global reduced spaces for the state, control, adjoint variables, respectively, by
\begin{equation}
\label{lg-rb-space}
\begin{split}
\begin{cases}
X^{N}_{l}(\Omega)=\text{span} \{u_{l}(\mu^{n}), 1\leq n \leq N\}\\
Y^{N}_{l}(\Omega)=\text{span} \{f_{l}(\mu^{n}), 1\leq n \leq N\}\\
Z^{N}_{l}(\Omega)=\text{span} \{\lambda_{l}(\mu^{n}), 1\leq n \leq N\}.
\end{cases}
\end{split}
\end{equation}
Moreover, we define the enriched space by
\[Q^{N}_{lg}(\Omega)= X^{N}_{l}(\Omega) \cup Z^{N}_{l}(\Omega) = \text{span} \{u_{l}(\mu^{n}), \lambda_{l}(\mu^{n}), 1\leq n \leq N\},\]
and set
\[ X^{N}_{lg}(\Omega)=Q^{N}_{lg}(\Omega),\quad Y^{N}_{lg}(\Omega)=Y^{N}_{l}(\Omega),\quad Z^{N}_{lg}(\Omega)=Q^{N}_{lg}(\Omega),\quad \mathcal{U}^{N}_{lg}(\Omega)= X^{N}_{lg}(\Omega)\times Y^{N}_{lg}(\Omega).\]

On the low-dimensional subspace $\mathscr{U}^{N}_{lg}(\Omega)\times Z^{N}_{lg}(\Omega)$, the local-global reduced approximation is: for $\forall \mu \in \Gamma$, find $(\underline{u}^{N}_{lg}, \lambda^{N}_{lg}) \in \mathscr{U}^{N}_{lg}(\Omega)\times Z^{N}_{lg}(\Omega)$ such that
\begin{equation}
\label{opt-sys-glrb}
\begin{cases}
\begin{split}
\mathcal{A}(\underline{u}_{lg},\underline{v}_{lg})+\mathcal{B}(\underline{v}_{lg},\lambda_{lg})
&=(\underline{\hat{u}},\underline{v}_{lg})\quad \forall \underline{v}_{lg} \in \mathscr{U}^{N}_{lg},\\
\mathcal{B}(\underline{u}_{lg},\widetilde{u}_{lg})&=(g,\widetilde{u}_{lg})_{\partial \Omega} \quad \forall \widetilde{u}_{lg} \in Z^{N}_{lg}(\Omega).
\end{split}
\end{cases}
\end{equation}
In a similar way, we can show the continuity and coercivity properties for the local-global reduced saddle point problem (\ref{opt-sys-glrb}).

The corresponding optimality system is: for $\forall \mu \in \Gamma$, find $(u^{N}_{lg}, f^{N}_{lg}, \lambda^{N}_{lg}) \in X^{N}_{lg}(\Omega)\times Y^{N}_{lg}(\Omega)\times Z^{N}_{lg}(\Omega)$ such that
\begin{equation}
\label{optimal-system-lg}
\begin{cases}
\begin{split}
a(u_{lg},\widetilde{u}_{lg};\mu)&=(f_{lg},\widetilde{u}_{lg};\mu) ~\forall \widetilde{u}_{lg} \in X^{N}_{lg}(\Omega),\\
a(\lambda_{lg},\widetilde{\lambda}_{lg};\mu)&=-(u_{lg}-\hat{u},\widetilde{\lambda}_{lg};\mu) ~\forall \widetilde{\lambda}_{lg} \in Z^{N}_{lg}(\Omega),\\
2\beta(f_{lg},\widetilde{f}_{lg};\mu)&=(\widetilde{f}_{lg},\lambda_{lg};\mu) ~\forall \widetilde{f}_{lg} \in Y^{N}_{lg}(\Omega).
\end{split}
\end{cases}
\end{equation}

Let $\{\tau_{j}\}^{2N}_{j=1}=\{u_{l}(\mu^{n})\}^{N}_{n=1}\cup\{\lambda_{l}(\mu^{n})\}^{N}_{n=1}$ such that
$Q^{N}_{lg}(\Omega)= \text{span} \{\tau_{j}, j=1,\cdots,2N\}$, and we can express the local-global reduced state, adjoint, and control solutions as
\begin{equation*}
\begin{split}
\begin{cases}
u_{lg}=\sum\limits^{2N}_{i=1}u^{i}_{lg}\tau_{i}=\sum\limits^{2N}_{i=1}
\sum\limits^{2M}_{k=1}u^{i}_{lg}Z^{i}_{1,k}\rho_{k}
=\sum\limits^{2N}_{i=1}\sum\limits^{2M}_{k=1}\sum\limits^{N_{h}}_{n=1}u^{i}_{lg}Z^{i}_{1,k}R_{k,n}\psi_{n},\\
f_{lg}=\sum\limits^{N}_{j=1}f^{j}_{lg}f_{l}(\mu^{j})=\sum\limits^{N}_{j=1}
\sum\limits^{N_{e}}_{\ell=1}f^{j}_{lg}Z_{2,\ell}\phi_{\ell},\\
\lambda_{lg}=\sum\limits^{2N}_{m=1}\lambda^{m}_{lg}\tau_{m}=\sum\limits^{2N}_{m=1}
\sum\limits^{2M}_{n=1}\lambda^{m}_{lg}Z^{m}_{1,n}\rho_{n}=\sum\limits^{2N}_{m=1}
\sum\limits^{2M}_{n=1}\sum\limits^{N_{h}}_{s=1}\lambda^{m}_{lg}Z^{m}_{1,n}R_{n,s}\psi_{s}.
\end{cases}
\end{split}
\end{equation*}
With the affine assumption (\ref{affine}), we can get the similar linear system to (\ref{aff-rep-glo})
in the local-global reduced case. Hence, given a random sample $\mu$, the reduced linear system associated
to the system (\ref{optimal-system-lg}) can be written as:
\begin{equation}
\label{system-lg}
\underbrace{\left[
  \begin{array}{ccc}
    2\beta M_{1,lg}(\mu) & 0              & -M^{T}_{2,lg}(\mu)\\
    0                    & M_{3,lg}(\mu)  & K^{T}_{lg}(\mu)\\
    -M_{2,lg}(\mu)       &K_{lg}(\mu)       &0
  \end{array}
\right]}\limits_{\Lambda_{N}(\mu)\in \mathbb{R}^{5N \times 5N}}
\left[
\begin{array}{c}
F_{lg}(\mu)\\
u_{lg}(\mu)\\
\lambda_{lg}(\mu)
\end{array}
\right]
=  \left[
\begin{array}{c}
0\\
\overline{\widehat{U}}_{lg}\\
d_{lg}
\end{array}
\right].
\end{equation}
With the definition of the basis matrix: \begin{equation*}
\mathcal{Z}=\left[
  \begin{array}{ccc}
    \mathcal{Z}^{T}_{2}     &0       &0\\
    0         &\mathcal{Z}^{T}_{1}   &0\\
    0         &0       &\mathcal{Z}^{T}_{1}\\
  \end{array}
\right] \in \mathbb{R}^{5N\times(N_{e}+2M)},
\end{equation*}
$\Lambda_{N}(\mu)$ is given by $\Lambda_{N}(\mu)= \mathcal{Z}(R\Lambda_{}(\mu)R^{T})\mathcal{Z}^{T}$. Although being dense rather than sparse as in the FE case, the system matrix $\Lambda_{N}(\mu)$ is very small and still symmetric with saddle-point structure. Finally,  we can  downscale the local-global reduced solutions to fine scale solution by
\[
\vec{u}_{lg}(\mu)=Z_{1}R_{l}u_{lg}(\mu),\quad  \vec{f}_{lg}(\mu)=Z_{2}f_{lg}(\mu),\quad  \vec{\lambda}_{lg}(\mu)=Z_{1}R_{l}\lambda_{lg}(\mu).
\]

\subsection{Sampling strategy}
\label{greedy-sample}
In order to find a few optimal parameter samples $\mu^{n}, 1\leq n \leq N$ to construct  the hierarchical Lagrange RB approximation spaces and to assure the fidelity of the reduced model to approximate the original model, we use the sampling strategy based on the  greedy algorithm \cite{rhp08,qrm11}. We denote the finite-dimensional sample set by $\Xi_{train} \subset \Gamma$. The cardinality of $\Xi_{train}$ will be denoted $|\Xi_{train}|= n_{train}$ and we assume that $\Xi_{train}$ is a good surrogate for the set $\Gamma$. The idea of the greedy procedure is that, starting with a train sample $\Xi_{train}$,  we adaptively select $N_{max}$ parameters $\mu_{1}, \cdots, \mu_{N_{max}}$ and form the hierarchical sequence of RB spaces $X^{N}_{lg}(\Omega), Y^{N}_{lg}(\Omega), Z^{N}_{lg}(\Omega)$. At the $N-$th iteration, the greedy algorithm enriches  the retained snapshots by a particular candidate snapshot over all candidates snapshots $(u_{l}(\mu), f_{l}(\mu), \lambda_{l}(\mu)), \mu \in \Xi_{train}$, which  is least well approximated by $X^{N-1}_{lg}(\Omega) \times Y^{N-1}_{lg}(\Omega) \times \lambda^{N-1}_{lg}(\Omega)$.

Firstly, we consider the residual errors for local reduced  model. By the first equation of system (\ref{optimal-system-gms}), we can get
\[a(u_{l}-u^{N}_{l}+u^{N}_{l}, \widetilde{u}_{l};\mu)=(f_{l}-f^{N}_{l}+f^{N}_{l};\mu),\quad \forall \widetilde{u}_{l} \in \mathscr{V}_{lr}(\Omega),\]
that is,
\[a(u_{l}-u^{N}_{l}, \widetilde{u}_{l};\mu)-(f_{l}-f^{N}_{l}, \widetilde{u}_{l};\mu)= -a(u^{N}_{l}, \widetilde{u}_{l}; \mu)+(f^{N}_{l}, \widetilde{u}_{l};\mu),\quad \forall \widetilde{u}_{l} \in \mathscr{V}_{lr}(\Omega).\]

Let the error $e_{u}(\mu):= u_{l}(\mu)-u^{N}_{l}(\mu), e_{f}(\mu):= f_{l}(\mu)-f^{N}_{l}(\mu)$ and $r_{1}(\widetilde{u}_{l};\mu) \in \mathscr{V}^{*}_{lr}(\Omega)$ (the dual space of $\mathscr{V}_{lr}(\Omega)$) be the residual
\[ r_{1}(\widetilde{u}_{l};\mu)= -a(u^{N}_{l}, \widetilde{u}_{l}; \mu)+(f^{N}_{l},\widetilde{u}_{l};\mu),\quad \forall \widetilde{u}_{l} \in \mathscr{V}_{lr}(\Omega).\]
Then we can obtain
\[a(e_{u}(\mu),\widetilde{u}_{l};\mu)-(e_{f}(\mu),\widetilde{u}_{l};\mu) = r_{1}(\widetilde{u}_{l};\mu).\]
By the Riesz representation theory, there exists a function $\hat{e}_{1}(\mu)$ such that
\begin{equation}
\label{error-function}
(\hat{e}_{1}(\mu),\widetilde{u}_{l})_{\mathscr{V}_{l}} = r_{1}(\widetilde{u}_{l};\mu), \quad \forall \widetilde{u}_{l} \in \mathscr{V}_{lr}(\Omega).
\end{equation}
Consequently, the dual norm of the residual $r_{1}(\widetilde{u}_{l};\mu)$ can be evaluated by the Riesz representation,
\[\|r_{1}(\widetilde{u}_{l};\mu)\|_{\mathscr{V}^{*}_{lr}}:= \sup_{\widetilde{u}_{l}\in\mathscr{V}_{lr}}\frac{r_{1}(\widetilde{u}_{l};\mu)}{\|\widetilde{u}_{l}\|_{\mathscr{V}_{lr}}}
=\|\hat{e}_{1}(\mu)\|_{\mathscr{V}_{lr}}.\]

The computation of the residual is important to perform the offline-online computation decomposition. Combined by  (\ref{affine}) and (\ref{Rb-repre}), the residual can be expressed as
\begin{equation}
\label{Residual}
\begin{split}
r_{1}(\widetilde{u}_{l};\mu)&= -a(u^{N}_{l}, \widetilde{u}_{l};\mu)+(f^{N}_{l},\widetilde{u}_{l};\mu)\\
&=-\sum^{2N}_{i=1}u^{\mathcal{N}}_{i}(\mu) a(\delta_{i}, \widetilde{u}_{l};\mu)+ \sum^{N}_{j}f^{\mathcal{N}}_{j=1}(\mu)(\xi_{j},\widetilde{u}_{l})\\
&=-\sum^{Q_{a}}_{q=1}\sum^{2N}_{i=1}\theta_{q}(\mu) u^{\mathcal{N}}_{i}(\mu) a_{q}(\delta_{i}, \widetilde{u}_{l})+ \sum^{N}_{j=1}f^{\mathcal{N}}_{j=1}(\mu)(\xi_{j},\widetilde{u}_{l}).
\end{split}
\end{equation}
By (\ref{error-function}) and (\ref{Residual}), we have
\[(\hat{e}_{1}(\mu),\widetilde{u}_{l})_{\mathscr{V}_{lr}}=-\sum^{Q_{a}}_{q=1}\sum^{2N}_{i=1}\theta_{q}(\mu) u^{\mathcal{N}}_{i}(\mu) a_{q}(\delta_{i},\widetilde{u}_{l})
+\sum^{N}_{j=1}f^{\mathcal{N}}_{j}(\mu)(\xi_{j},\widetilde{u}_{l}).
\]
This implies that
\begin{equation}
\label{residual-expression}
\begin{split}
\hat{e}_{1}(\mu)=-\sum^{Q_{a}}_{q=1}\sum^{2N}_{i=1}\theta_{q}(\mu) u^{\mathcal{N}}_{i}(\mu)L^{\delta}_{i}+
\sum^{N}_{j=1}f^{\mathcal{N}}_{j}(\mu) L^{\xi}_{j},
\end{split}
\end{equation}
where $L^{\delta}_{i}, L^{\xi}_{j}$ are Riesz representations of $a_{q}(\delta_{i},\widetilde{u}_{l})$ and $(\xi_{j},\widetilde{u}_{l})$, i.e., $(L^{\delta}_{i},\widetilde{u}_{l})_{\mathscr{V}_{lr}}= a_{q}(\delta_{i},\widetilde{u}_{l})$ and
$(L^{\xi}_{j},\widetilde{u}_{l})_{\mathscr{V}_{lr}}=(\xi_{j},\widetilde{u}_{l})$ for all $\widetilde{u}_{l} \in \mathscr{V}_{lr}(\Omega)$, respectively. We thus obtain
\begin{equation}
\label{residual-norm}
\begin{split}
\|\hat{e}_{1}(\mu)\|^{2}_{\mathscr{V}_{lr}}&= \sum^{Q_{a}}_{q=1}\sum^{Q_{a}}_{q'=1}\sum^{2N}_{i=1}\sum^{2N}_{i'=1} \theta_{q}(\mu)\theta_{q'}(\mu)u^{\mathcal{N}}_{i}u^{\mathcal{N}}_{i'}(L^{\delta}_{i}, L^{\delta}_{i'})_{\mathscr{V}_{lr}}\\
&+\sum^{N}_{j=1}\sum^{N}_{j'=1}f^{\mathcal{N}}_{j}(\mu)f^{\mathcal{N}}_{j'}(\mu)
(L^{\xi}_{j}, (L^{\xi}_{j'})_{\mathscr{V}_{lr}}
-2\sum^{Q_{a}}_{q=1}\sum^{2N}_{i=1}\sum^{N}_{j=1}\theta_{q}(\mu)f^{\mathcal{N}}_{j}(\mu)
(L^{\delta}_{i},L^{\xi}_{j})_{\mathscr{V}_{lr}}.
\end{split}
\end{equation}
Combining (\ref{error-function}) and (\ref{residual-norm}) gives the calculation of the dual norm of the residual $\|r_{1}(\widetilde{u}_{l};\mu)\|_{\mathscr{V}^{*}_{lr}}$.
With the similar way, we can define residuals $r_{2}(\widetilde{\lambda}_{l};\mu)$ and $r_{3}(\widetilde{f}_{l};\mu)$ for the rest equations of system (\ref{optimal-system-gms}) with the corresponding functions $\hat{e}_{2}(\mu)$ and  $\hat{e}_{3}(\mu)$.

Next, we define the error estimator as
\begin{equation}
\label{error-estimator}
\begin{split}
\vartriangle_{N}(\mu)&= \bigg(\|r_{1}(\widetilde{u}_{l};\mu)\|^{2}_{\mathscr{V}^{*}_{lr}}+\|r_{2}(\widetilde{u}_{l};\mu)\|^{2}_{\mathscr{V}^{*}_{lr}}
+\|r_{3}(\widetilde{u}_{l};\mu)\|^{2}_{L^{2}}\bigg)^{\frac{1}{2}}\\
&=(\|\hat{e}_{1}(\mu)\|^{2}_{\mathscr{V}_{lr}}+\|\hat{e}_{2}(\mu)\|^{2}_{\mathscr{V}_{lr}}
+\|\hat{e}_{3}(\mu)\|^{2}_{L^{2}})^{\frac{1}{2}}.
\end{split}
\end{equation}
Let  $\varepsilon^{*}_{tol}$ be
a chosen tolerance for the stopping criterium. The greedy sampling strategy is described  in
Algorithm \ref{greedy-alg}.\\
\begin{algorithm}
\caption{Greedy algorithm for selecting optimal parameter samples $S_N$}
\textbf{Input}:  A training set $\Xi_{train}\subset \Gamma$, a tolerance $\varepsilon^{*}_{tol}$, a maximum number $N_{max}$\\
 \textbf{Output}:  $S_{N}$ and $X^{N}_{lg}, Y^{N}_{lg}, Z^{N}_{lg}$\\
~1:~~  \textbf{Initialization}: $\mu_{1}= \text{mean}(\Xi_{train})$, $S_{1}=\{\mu_{1}\}$;\\
~2:~~  compute $u_{l}(\mu_{1}), f_{l}(\mu_{1}), \lambda_{l}(\mu_{1})$ by local model reduction method;\\
~3:~~  obtain $X^{1}_{lg}=\text{span}\{u_{l}(\mu_{1}),\lambda_{l}(\mu_{1})\}$, $Y^{1}_{lg}=\text{span}\{f_{l}(\mu_{1})\}$, $Z^{1}_{lg}=X^{1}_{lg}$;\\
~4:~~  $\Xi_{train}$ = $\Xi_{train}\setminus \mu_{1}$;\\
~6:~~  apply Gram-Schmidt process to $X^{1}_{lg}$, $Y^{1}_{lg}$, $Z^{1}_{lg}$;\\
~7:~~  $\varepsilon_{1}=\max\limits_{\mu \in \Xi_{train}}\vartriangle_{1}(\mu)$;\\
~8:~~  N=1;\\
~9:~~  $\mathbf{while}~\varepsilon_{N} \leq \varepsilon^{*}_{tol}$ or $N \leq N_{max}$\\
~10:~~  $N=N+1$;\\
~11:~~  $\mu_{N}= \text{argmax}_{\mu \in \Xi_{train}} \vartriangle_{N-1}(\mu)$;\\
~12:~~  $S_{N} =S_{N-1} \cup \{\mu_{N}\}$; \\
~13:~~  compute $u_{l}(\mu_{N}), f_{l}(\mu_{N}), \lambda_{l}(\mu_{N})$ by local model reduction method;\\
~14:~~  update the reduced spaces $X^{N}_{lg}=X^{N-1}_{lg}\cup\text{span}\{u_{l}(\mu_{N}),\lambda_{l}(\mu_{N})\}$,\\
       $~~~~~$ $f^{N}_{lg} =f^{N-1}_{lg}\cup\text{span}\{f_{l}(\mu_{N})\}$, $Z^{N}_{lg} = X^{N}_{lg}$;\\
~15:~~  apply Gram-Schmidt process to $X^{N}_{lg}$, $Y^{N}_{lg}$, $Z^{N}_{lg}$;\\
~16:~~  $\Xi_{train}$ = $\Xi_{train}\setminus \mu_{N}$;\\
~17:~~  $\varepsilon_{N}=\max\limits_{\mu \in \Xi_{train}}\vartriangle_{N}(\mu)$;\\
 \textbf{end~while}
 \label{greedy-alg}
\end{algorithm}

With the model reduction framework shown in Section \ref{lg-framework} and the above greedy sampling strategy, we
outline the local-global model reduction method in  algorithm \ref{local-global}.

\begin{algorithm}[htb]
\caption{Local-global model reduction method for stochastic optimal control problems.}
\textbf{Input}: Training set $\Xi_{train} \subset \Gamma$, Testing set $\Xi_{test} \subset \Gamma$\\
\textbf{Offline Stage:}\\
~1:~~ get spaces $X^{N}_{lg}(\Omega), Y^{N}_{lg}(\Omega), Z^{N}_{lg}(\Omega)$ by Algorithm \ref{greedy-alg};\\
~2:~~  calculate the local reduced matrix $R$ and the global reduced matrix $Z$;\\
 ~3:~~ \textbf{if} affine assumptions (\ref{affine}) holds, then \\
  $~~~~~$  compute the $\mu$-independent matrices $K^{q}_{h}$, $M_{1,h}$, $M_{2,h}$, $M_{3,h}$ and $\mu$-independent load vectors\\
  $~~~~~$  $\widehat{U}_{p}$ in (\ref{matrix-def});\\
  $~~~~~$ \textbf{else} \\
  $~~~~~$ use EIM to $a(\cdot,\cdot;\mu)$ and $(f,\cdot;\mu)$;\\
  ~4:~~ Go back to the step 4;\\
 $~~~~~$ \textbf{end if}  \\
\textbf{Online Stage:} For the $k_{th}$ test parameter $\mu_{k}$ in $\Xi_{test}$\\
~5:~~ update the $\mu$-dependent coefficients $Q^{q}_{a}(\mu_{k})$, $\widehat{u}_{p}(\mu_{k})$ in (\ref{matrix-def});\\
~6:~~ assemble the full matrix $\Lambda_{h}(\mu_{k})$ and right term $[0;\overline{\widehat{U}};d]^{T}$;\\
~7:~~ compute the reduced matrix $\Lambda_{N}(\mu_{k})$ and the reduced right hand $[0;\overline{\widehat{U}}_{lg};d_{lg}]^{T}$;\\
~8:~~ solve the system (\ref{system-lg}) for the reduced solutions $u_{lg}(\mu_{k})$, $F_{lg}(\mu_{k})$, $\lambda_{lg}(\mu_{k})$;\\
~9:~~ get the local-global RB solutions $\vec{F}_{lg}(\mu_{k})$, $\vec{u}_{lg}(\mu_{k})$ , $\vec{\lambda}_{lg}(\mu_{k})$ defined on fine grid;
\label{local-global}
\end{algorithm}

In the paper, we have focused on   the global existence and uniqueness of the optimal solution and the local-global model reduction framework for the distributed optimal control problem.
The existence and uniqueness of global optimal solution for boundary optimal control problem can be proved in a similar way, and the local-global model reduction method can also be used in the boundary optimal control problem.
We present a numerical example for the boundary optimal control problem in Section \ref{Ex3}.

\section{Numerical experiments}
\label{sect6}
A variety of numerical examples are presented to demonstrate the efficiency of the local-global model reduction method for stochastic optimal control problems. In this section, we focus on stochastic optimal control problems constrained by elliptic PDEs. In the following, we are going to show the relative errors and computation performance  for the control, state, and adjoint variables. In Section \ref{Ex1}, we consider the stochastic optimal control problem, where  the diffusion coefficient and target function are both random fields. The effect  of the regularization parameter $\beta$ will be discussed. In Section \ref{Ex2}, we consider the control problem which is defined on a random domain. The last example in Section \ref{Ex3}, we compute a Neumann boundary control problem with the local-global model reduction method. In this case, the diffusion coefficient and target function are random fields with high dimensional parameters. To demonstrate the efficiency of our presented model reduction method, we will list the detailed CPU time in the last two examples.

\subsection{Stochastic optimal control problem defined on deterministic domain}
\label{Ex1}
In the first example, we consider the target function $\hat{u}(x,\mu)$ and coefficient $\kappa(x,\mu)$ are both related to the parameter sample $\mu \in \Gamma$ and numerically explore the approximation of optimal solution using the  local-global model reduction method.

The computational domain is a two-dimensional unit square $\Omega = [0,1]^2$. We consider the optimal control problem
\begin{equation}
\label{model-ddom}
\begin{cases}
\begin{split}
 &\min\limits_{u,f} J=\frac{1}{2}\|u(x,\mu)-\hat{u}(x,\mu)\|
 ^2_{\mathscr{L}^2(\Omega)}+\beta\|f(x,\mu)\|^2_{\mathscr{L}^2(\Omega)}\\
&s.t.~~-div (\kappa(x,\mu)\nabla u)=f(x,\mu) ~\text{in}~\Omega\\
&\quad \quad \quad u|_{\partial \Omega}=g(x),
\end{split}
\end{cases}
\end{equation}
where  the diffusion coefficient $\kappa(x,\mu)$  and the desired state function $\hat{u}(x,\mu)$  are
\begin{equation*}
\begin{cases}
\begin{split}
\kappa(x,\mu)&= (\mu^2+(\mu+0.5)^2)\kappa_{1}(x)+(1+\exp(\mu)\cos(\mu/3))^2\kappa_{2}(x),\\
\hat{u}(x,\mu)&= x_{1}x_{2}(x_{1}+1)(x_{2}-1)\mu+x^{2}_{1}x_{2}(x_{1}-1)(x_{2}+1)\cos(\mu)\\
&\quad +x_{1}x^{3}_{2}(x_{1}-1)(x_{2}-1)\mu^2+\exp(x_{1}/3)x^{2}_{2}\sin(\mu).
\end{split}
\end{cases}
\end{equation*}
Here $x=(x_{1},x_{2})\in \Omega$ and the random variable $\mu \sim Beta(\theta_{1},\theta_{2})$ obey beta distribution with
two shape parameters $\theta_{1},\theta_{2} \in \mathds{N}_{+}$. In this case, $(\theta_{1},\theta_{2})$ take the value (1,1). $\kappa_{1}(x)$ and $\kappa_{2}(x)$ are independent of $\mu$.
The $\kappa_{1}(x)$ and $\kappa_{2}(x)$ are  high-contrast functions and their  maps  are depicted in Fig.~\ref{osci-coeff}.

\begin{figure}[!h]
\centering
\subfigure{\includegraphics[width=3.1in, height=2.1in,angle=0]{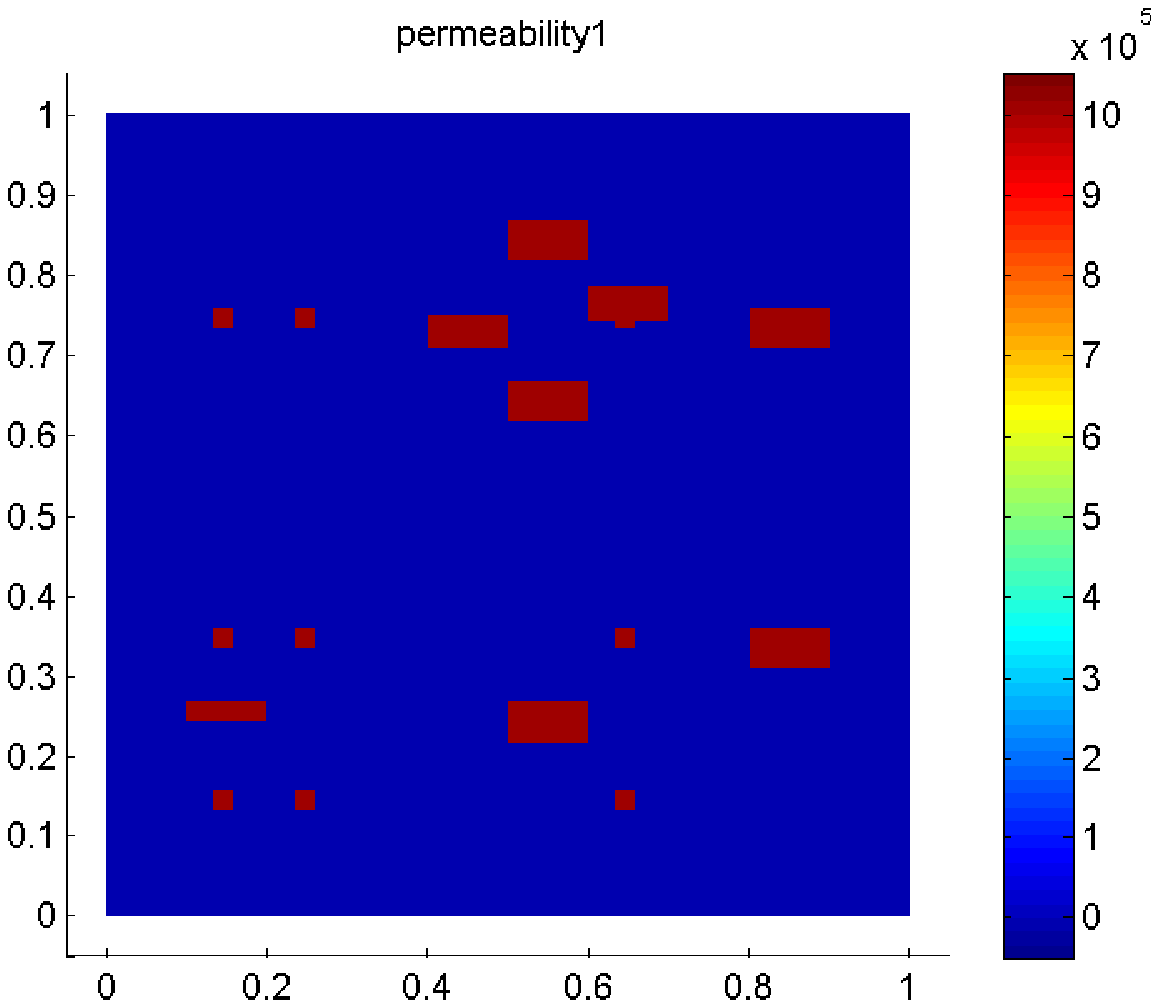}}
\subfigure{\includegraphics[width=3.1in, height=2.1in,angle=0]{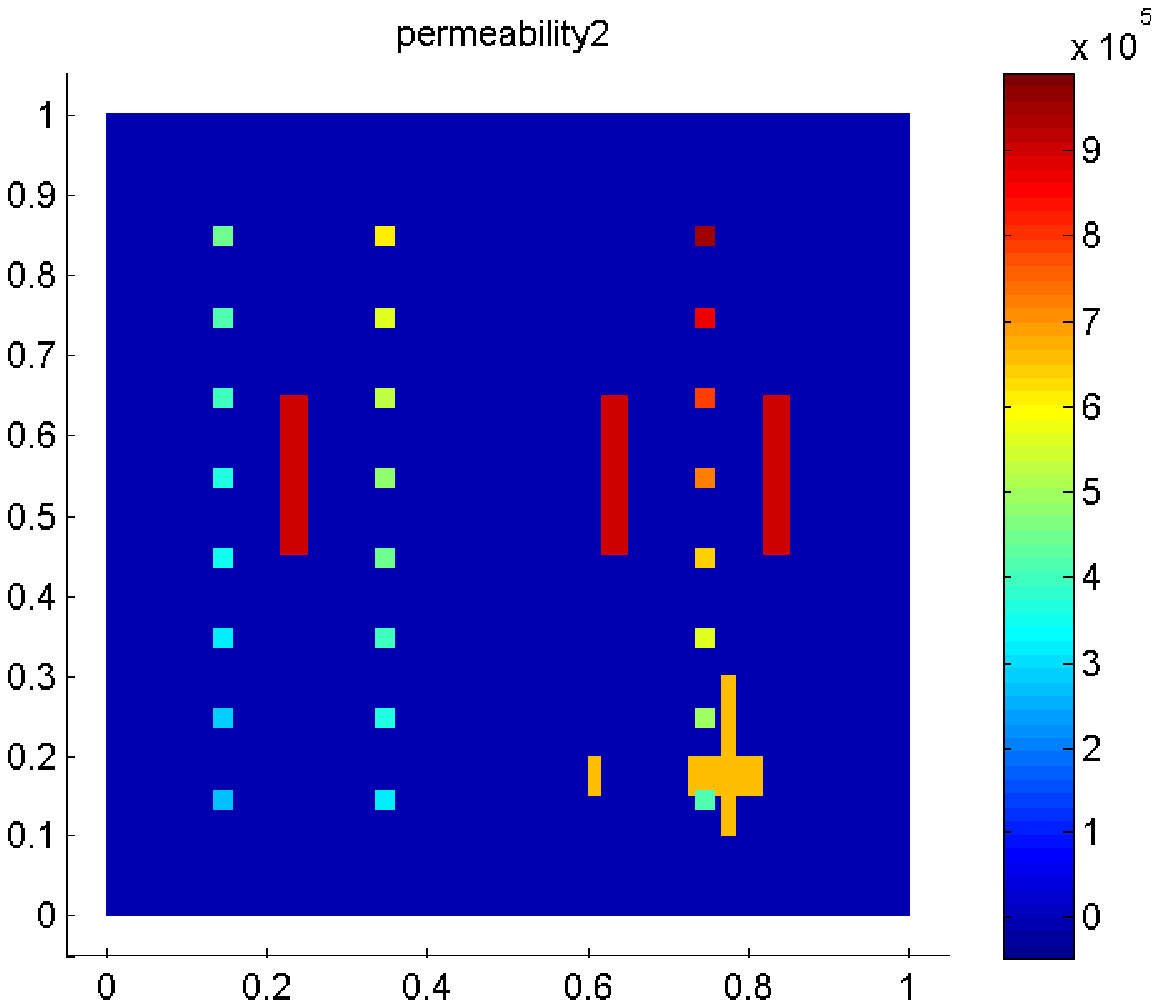}}
\caption{\em{High-contrast coefficients $\kappa_{1}$ (left) and $\kappa_{2}$ (right).} }
\label{osci-coeff}
\end{figure}

In this example, we use $120\times 120$ uniform fine grid to compute the reference optimal solutions $(u_{ref}, f_{ref}, \lambda_{ref})$. The local-global model reduction solutions $(u_{lg}, f_{lg}, \lambda_{lg})$ are computed on $10\times 10$ coarse mesh. We define the  relative errors  for the state variable $u$, the control variable $f$ and the adjoint variable by
\begin{equation*}
\begin{cases}
\begin{split}
  &e^{2}_{u}=\frac{1}{N}\sum\limits^{N}_{i=1}\frac{\|u_{ref}(x,\mu_{i})-u_{lg}(x,\mu_{i})\|
  _{L^{2}(\Omega)}}{\|u_{ref}(x,\mu_{i})\|_{L^{2}(\Omega)}},\\
  &e^{2}_{f}=\frac{1}{N}\sum\limits^{N}_{i=1}\frac{\|f_{ref}(x,\mu_{i})-f_{lg}(x,\mu_{i})\|_{L^{2}(\Omega)}}
  {\|f_{ref}(x,\mu_{i})\|_{L^{2}(\Omega)}},\\
  &e^{2}_{\lambda}=\frac{1}{N}\sum\limits^{N}_{i=1}\frac{\|\lambda_{ref}(x,\mu_{i})-\lambda_{lg}(x,\mu_{i})\|
  _{L^{2}(\Omega)}}{\|\lambda_{ref}(x,\mu_{i})\|_{L^{2}(\Omega)}}.\\
\end{split}
\end{cases}
\end{equation*}

The contour plot of the target solution $\hat{u}$ is given in Fig.~\ref{target-fun}, and contour plots of the control $f$ and the state $u$ for the three values of $\beta$ are given in Fig.~\ref{diff-beta}. For the optimal control problem governed by PDE, the regularization parameter plays an important role. For small $\beta$, the control variable is not heavily penalized, and so the state may be closed to the desired state. However, given a large $\beta$, it is hard for the state variable to be near to the desired state in the relative $L^{2}$-norm because  the input of control contributes more heavily into the cost functional. In  Table \ref{beta-data}, we demonstrate the relative $L^{2}$ errors about the state variable $u$ and the control variable $f$ for different $\beta$. Moreover, we compute the minimal values of cost functional. From the figures and the table, we can find that it is necessary to find a suitable regularization parameter for the optimal solution.
\begin{figure}[!h]
\centering
\includegraphics[width=3.1in, height=2.1in,angle=0]{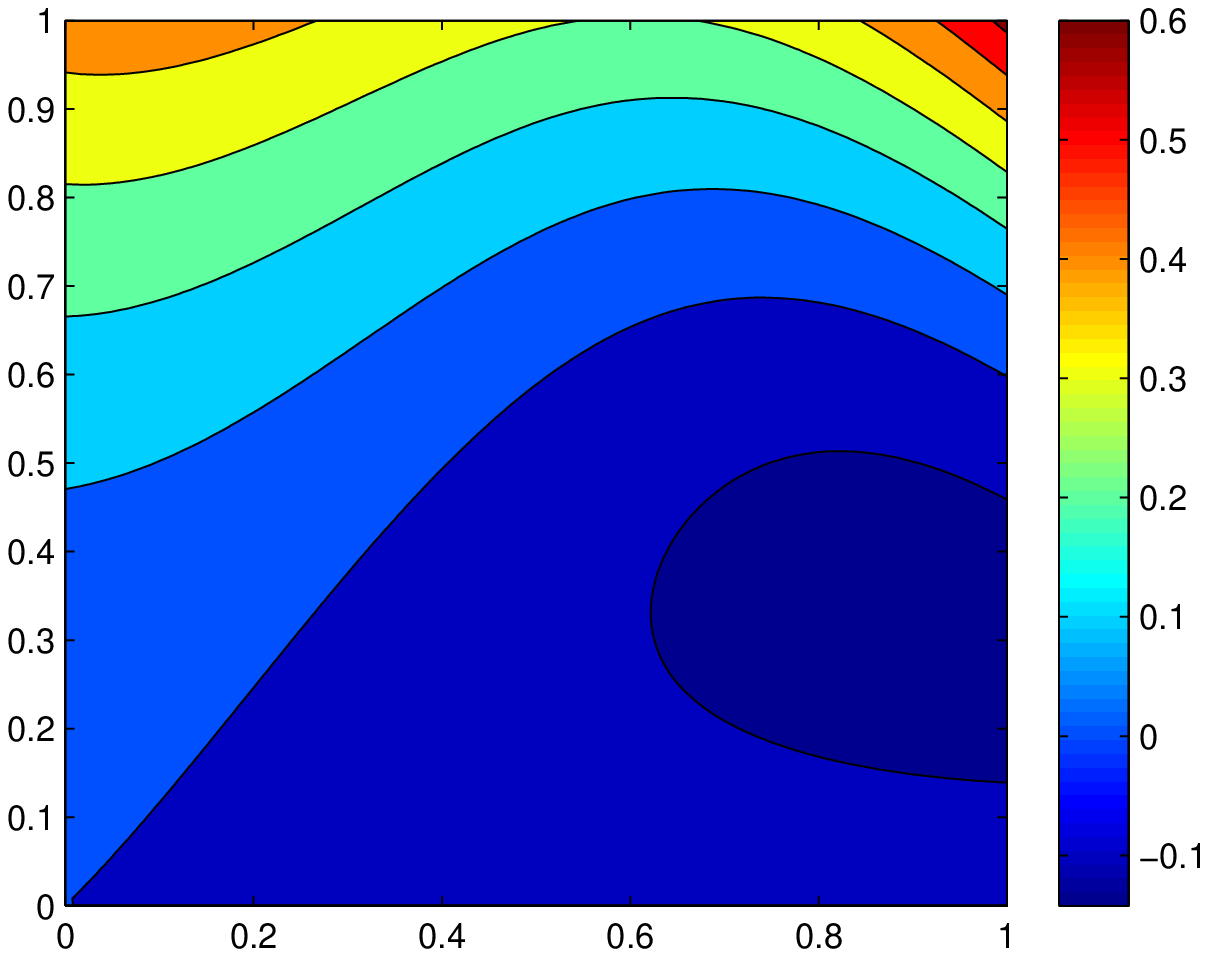}
\caption{\em{Contour plot of $\hat{u}(x,\mu)$ at the $\overline{\mu}$.}}
\label{target-fun}
\end{figure}

\begin{figure}[!h]
\centering
\subfigure[u, $\beta=1\times 10^{-2}$]{\includegraphics[width=2in, height=2in,angle=0.3]{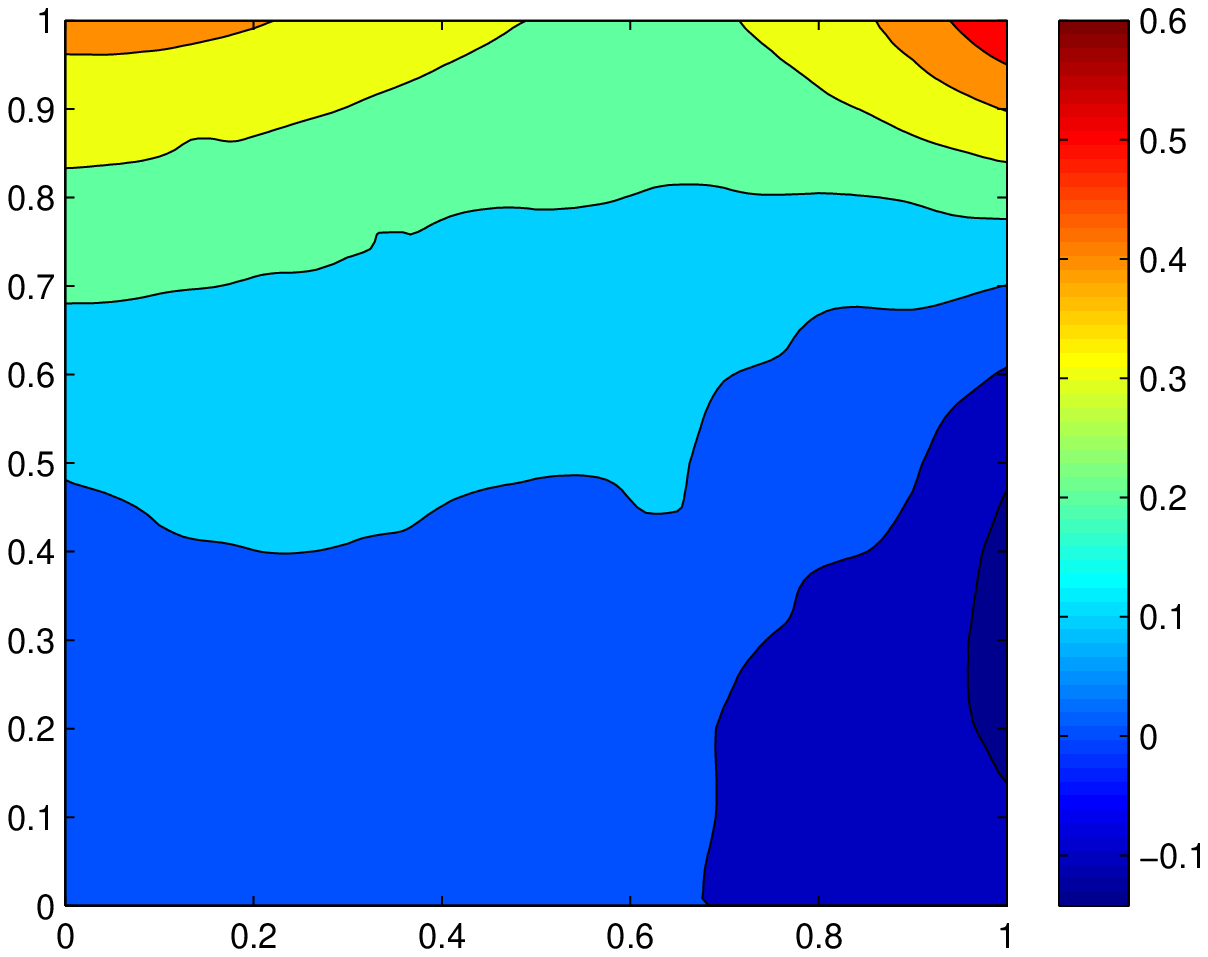}}
\subfigure[u, $\beta=2\times10^{-4}$]{\includegraphics[width=2in, height=2in,angle=0.3]{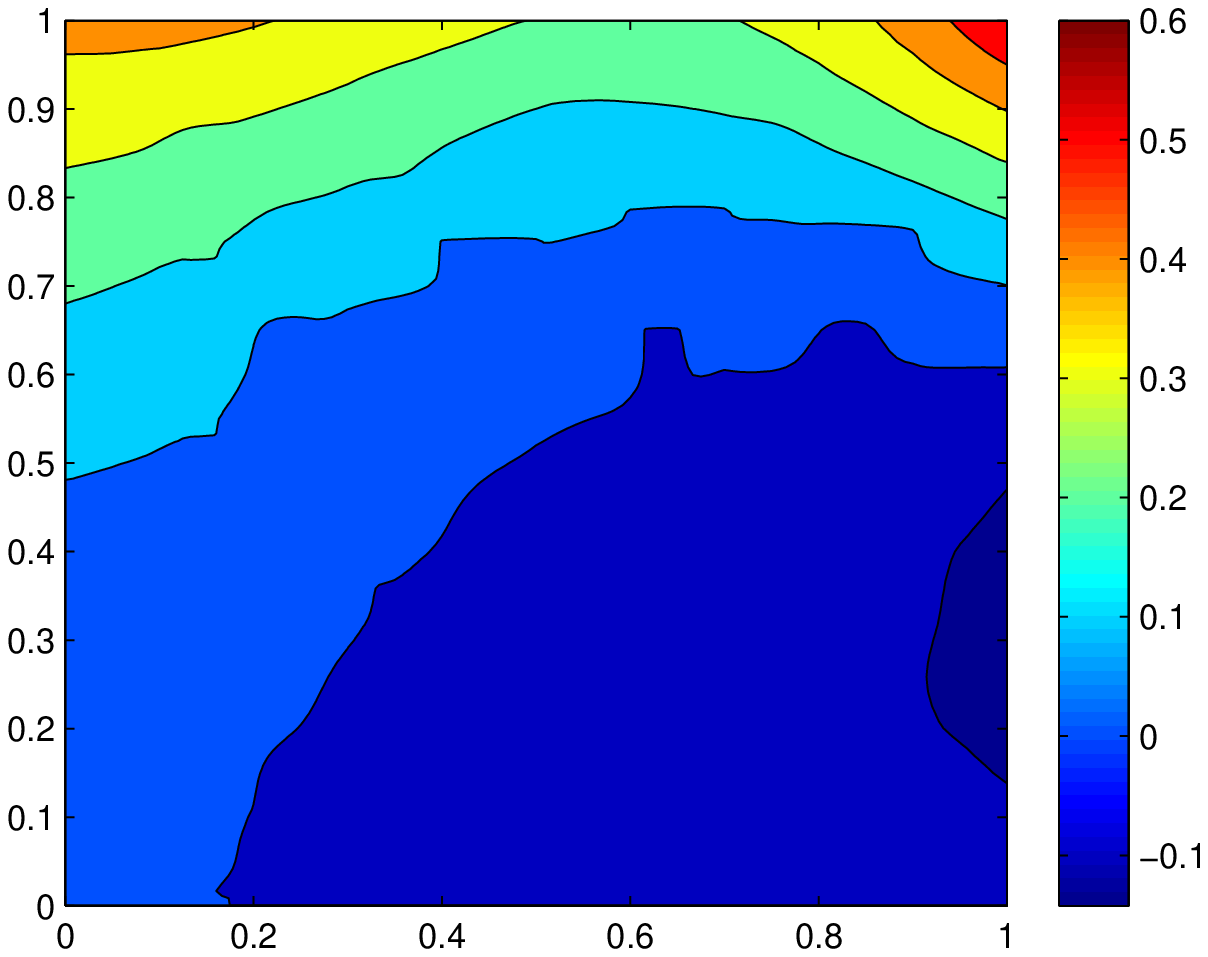}}
\subfigure[u, $\beta=0.5\times10^{-5}$]{\includegraphics[width=2in, height=2in,angle=0.3]{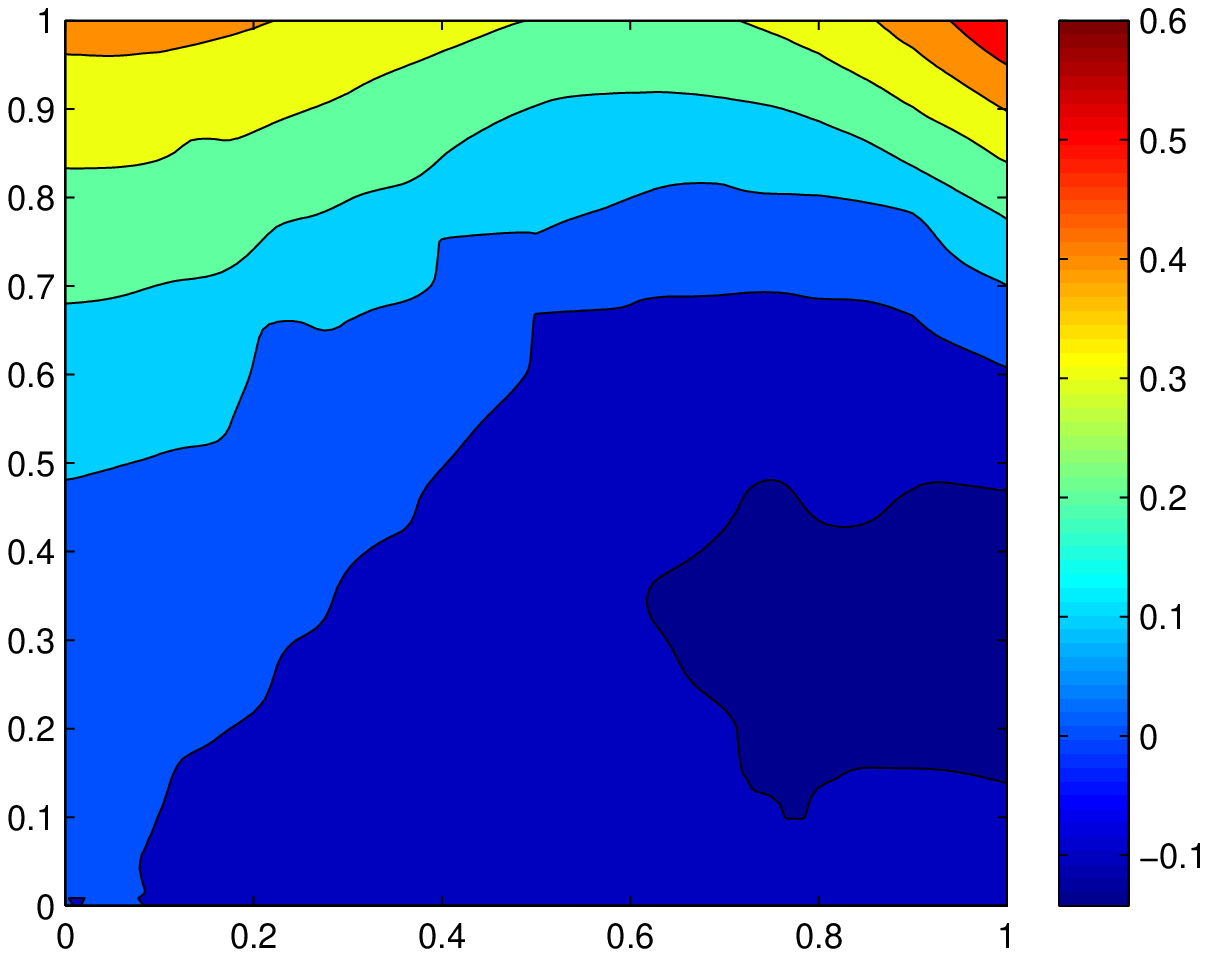}}
\subfigure[f, $\beta=1\times 10^{-2}$]{\includegraphics[width=2in, height=2in,angle=0.3]{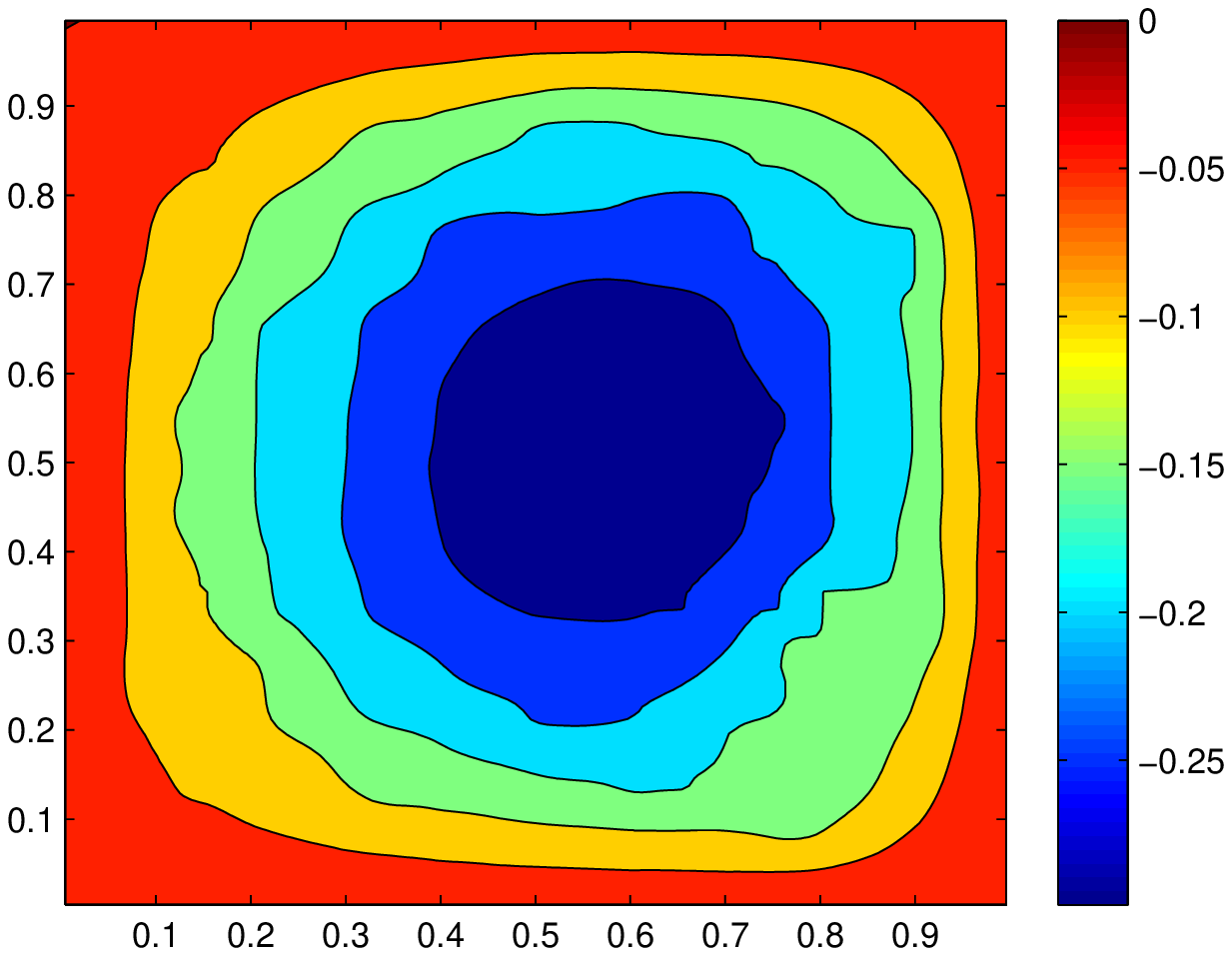}}
\subfigure[f, $\beta=2\times10^{-4}$]{\includegraphics[width=2in, height=2in,angle=0.3]{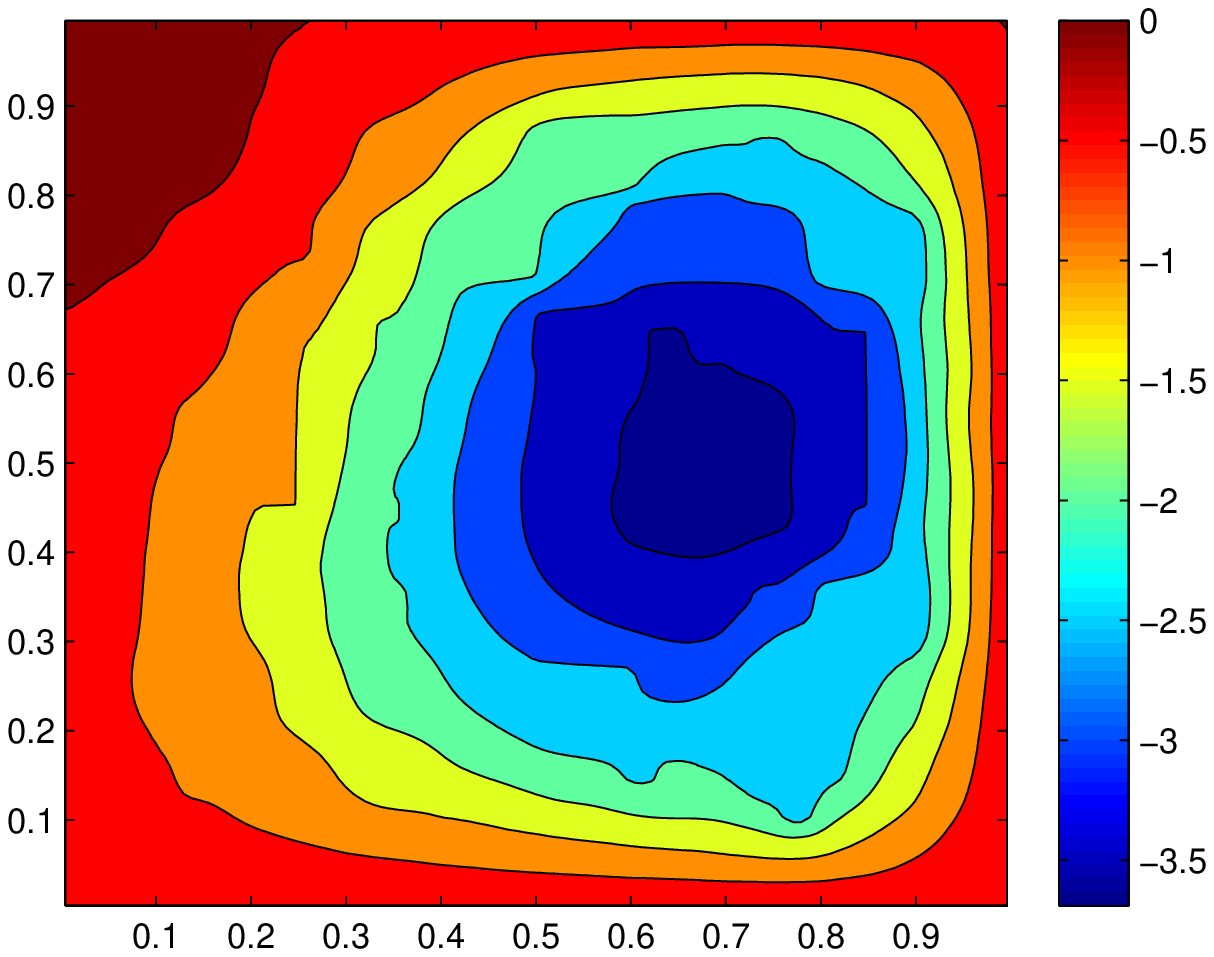}}
\subfigure[f, $\beta=0.5\times10^{-5}$]{\includegraphics[width=2in, height=2in,angle=0.3]{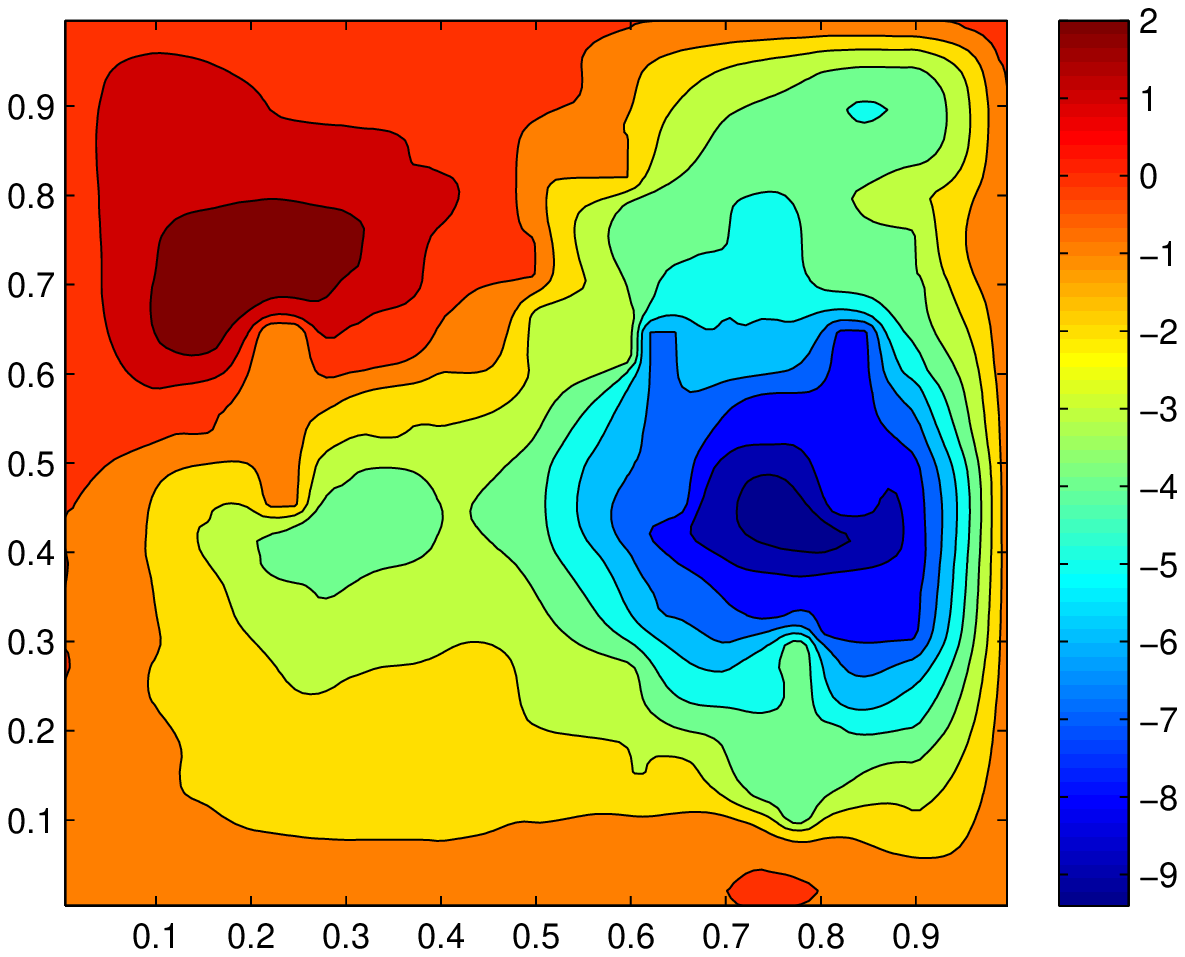}}
\caption{\em{Contour plots of the state $u$,  the control $f$ for $\beta = 10^{-2}$,  $2 \times10^{-4}$,  $0.5\times10^{-5}$.}}
\label{diff-beta}
\end{figure}

\begin{table}[!htbp]\center
\renewcommand\arraystretch{1.2}
\scriptsize
    \begin{tabular}{c| c c c c}
    \hline
    $\beta$  & $1\times 10^{-2}$ &$2\times10^{-4}$ &$0.5\times10^{-5}$ \\
   \hline
    $e^{2}_{u}$     & 1.057E-02 & 1.033E-02       & 9.547E-03 \\
    $e^{2}_{f}$     & 1.779E-02 & 2.627E-02       & 5.586E-02 \\
    $J_{min}$     & 2.048E-02 & 3.974E-06       & 3.753E-05 \\
    \hline
    \end{tabular}
    \caption{\em{The relative $L^{2}$ errors with different
    regularization parameter $\beta$ for the state variable
    $u$, the control variable $f$ and the cost functional $J$.}}
    \label{beta-data}
\end{table}
Next we fix the regularization parameter $\beta$, the number of local basis functions $L$ for each coarse element and choose five global basis functions. To discuss the effect of coarse mesh size, we consider some different coarse mesh sizes in the example,  $H=\{1/5, 1/6, 1/8, 1/10, 1/12\}$. The relative $L^{2}$ errors and the corresponding
minimal values of cost functional are listed in Table \ref{H-data}. From the table, we can see that the approximation for the state variable $u$ and the control variable $f$ are improved as the coarse gird is refined. The minimal values of the cost functional $J$ get smaller as the coarse mesh becomes finer.

\begin{table}[!htbp]\center
\renewcommand\arraystretch{1.2}
\scriptsize
    \begin{tabular}{c|c c c c c}
    \hline
    $Coarse~mesh~size$ & $H=1/5$     &$H=1/6$      &$H=1/8$      &$H=1/10$     &$H=1/12$\\
   \hline
    $e^{2}_{u}$     & 5.229E-01 & 3.829E-01 & 2.425E-01 & 1.057E-02 & 7.743E-03 \\
    $e^{2}_{f}$     & 6.920E-01 & 7.117E-01 & 4.569E-01 & 1.779E-02 & 1.445E-02 \\
    $J_{min}$     & 4.003E-03 & 2.039E-03 & 9.756E-04 & 4.111E-04 & 4.103E-04 \\
    \hline
    \end{tabular}
    \caption{\em{The relative $L^{2}$ errors with different
    coarse mesh size $H$ for the state variable
    $u$, the control variable $f$ and the minimal value $J_{min}$.}}
    \label{H-data}
\end{table}

\subsection{Stochastic optimal control problem defined on random domain}
\label{Ex2}
In this subsection,  we consider the stochastic optimal control problem described by (\ref{model-rdom1}) defined in  a random domain $\Omega(\mu)=\{(x_{1},x_{2})| 0\leq x_{1}\leq 1, s(x_{1},\mu)\leq x_{2}\leq 1, \mu \in \Gamma\}$,
i.e.,
\begin{equation}
\label{model-rdom1}
\begin{cases}
\begin{split}
 &\min\limits_{u,f} J=\frac{1}{2}\|u(x,\mu)-\hat{u}(x)\|
 ^2_{\mathscr{L}^2(\Omega(\mu))}+\beta\|f(x,\mu)\|^2_{\mathscr{L}^2(\Omega(\mu))},\\
&s.t.~~-div (\kappa(x)\nabla u)=f(x) ~\text{in}~\Omega(\mu),\quad u|_{\partial \Omega}=g(x).
\end{split}
\end{cases}
\end{equation}
Here the coefficient $\kappa(x)$ and target function $\hat{u}(x)$ are defined by
\[\kappa(x)=|x_{1}x_{2}|+1, \quad
\hat{u}(x)=x_{1}x_{2}(x_{1}-1)(x_{2}-\frac{x_{1}}{2}-1)+1.
\]

To be specific, we treat the rough bottom boundary as a random field $s(x_{1},\mu)$ with zero mean
$\mathbb{E}(s(x_{1},\mu))=0$ and an exponential two-point covariance function
\[ C_{s}(x_{1},z_{1})=\mathds{E}(s(x_{1},\mu),s(z_{1},\mu))=\exp(-|x_{1}-z_{1}|).
\]
With the finite-term Karhunen-Lo$\grave{e}$ve expansion (K-L expansion), $s(x_{1},\mu)$ can be approximated by
\[
s(x_{1},\mu)\approx \sigma \sum\limits^{N}_{n=1}\sqrt{\lambda_{n}}\phi_{n}(x)X_{n}(\omega),
\]
where $\{(\lambda_{n},\phi_{n})\}^{N}_{n=1}$ are solutions of the eigenvalue problem,
\[ \int_{\Gamma}C(x_{1},x_{2})\phi_{n}(x_{1})dx_{1}=\lambda_{n}\phi_{n}(x_{2}),\quad \forall x \in \Omega.\]
We set $\{X_{n}(\omega)\}\sim U(-1,1)$ to be independent uniform random variables and use the parameter
$0< \sigma<1$ to control the maximum deviation of the rough surface.

In Fig.~\ref{random-domain}, we employ the K-L expansion to represent the random boundary and show some realizations of the boundary.
\begin{figure}[!h]
\centering
\subfigure[]{\includegraphics[width=2in, height=2in,angle=0.3]{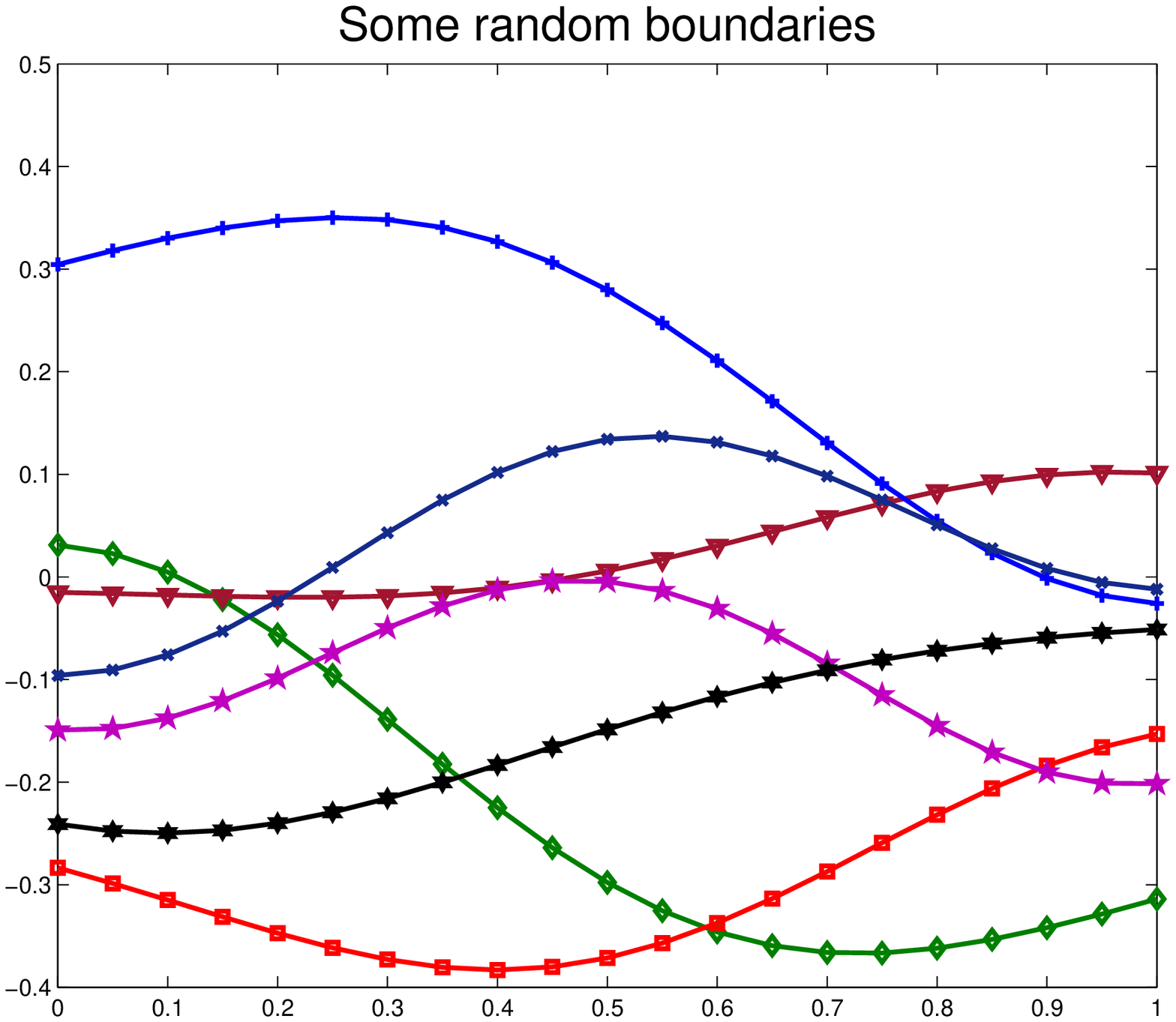}}
\subfigure[]{\includegraphics[width=2in, height=2in,angle=0.3]{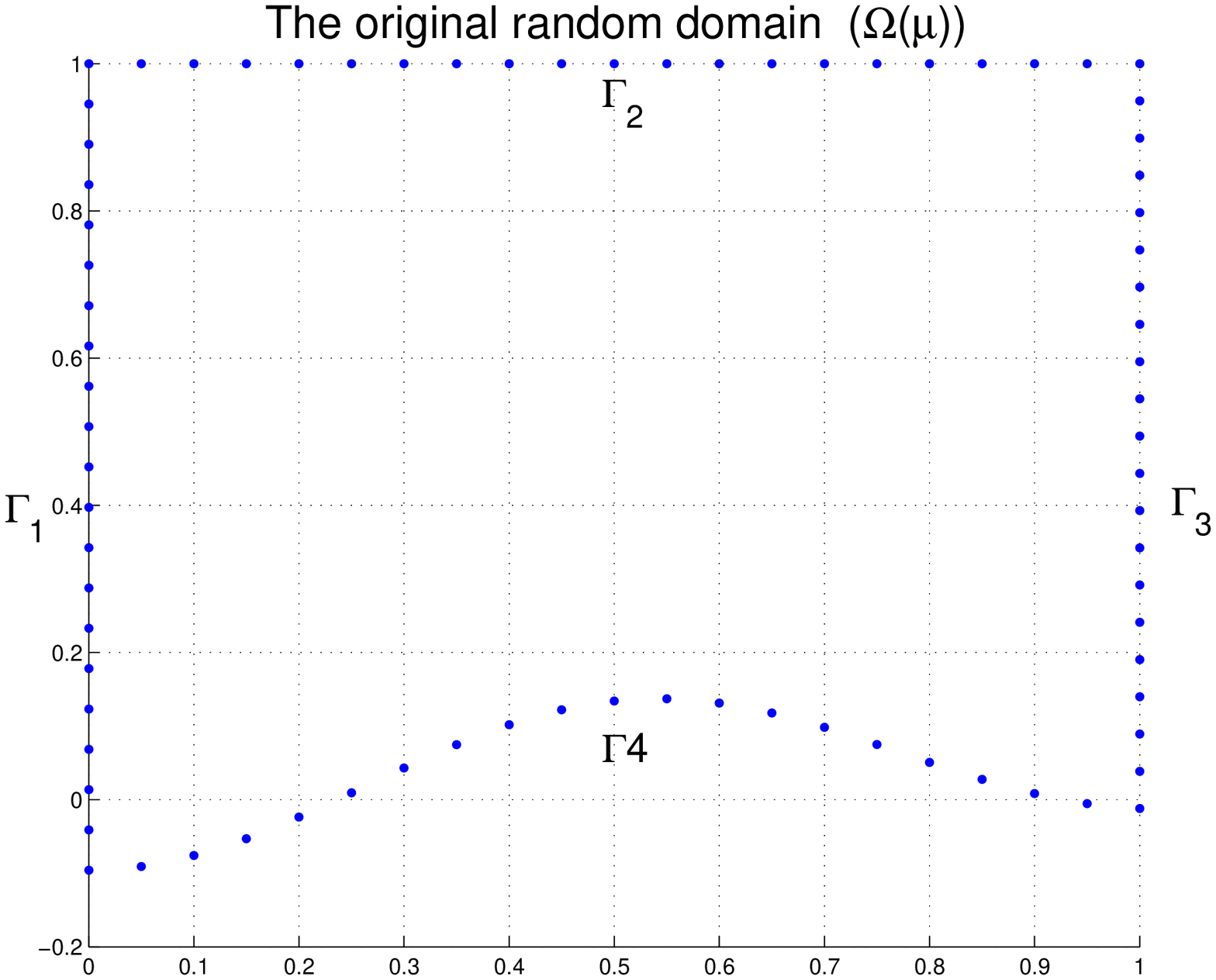}}
\subfigure[]{\includegraphics[width=2in, height=2in,angle=0.3]{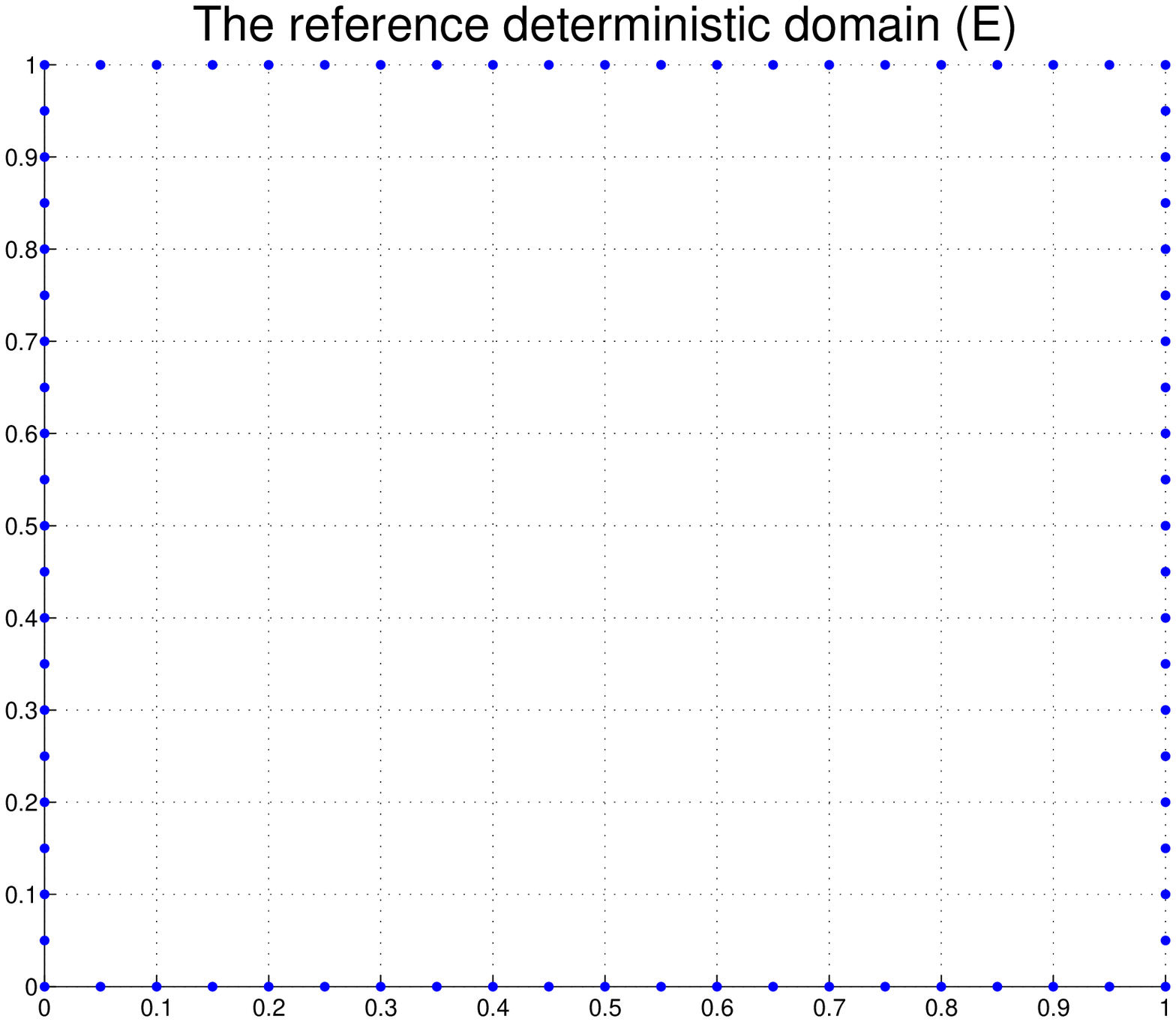}}
\caption{\em{(a) 7  realizations of the bottom boundary $s(x_{1},\mu)$ generated by five-term K-L expansion;  (b) a boundary realization  in the physical domain $(x_{1},x_{2})$; (c) the mapped domain $(\xi_{1},\xi_{2})$}.}
\label{random-domain}
\end{figure}
To treat the random domain, we need to formulate a stochastic map \cite{xt06}. The stochastic mapping of $\Omega(\mu)$ onto $E$ is constructed via the solutions of the Laplace equations,
\[\frac{\partial^{2}x_{1}}{\partial \xi^{2}_{1}}+\frac{\partial^{2}x_{1}}{\partial \xi^{2}_{2}}=0,\quad
\frac{\partial^{2}x_{2}}{\partial \xi^{2}_{1}}+\frac{\partial^{2}x_{2}}{\partial \xi^{2}_{2}}=0 \quad in~E,\]
subject to the boundary conditions
\begin{equation*}
\begin{split}
&x_{1}(0,\xi_{2})=x_{1}|_{\Gamma_{1}},\quad x_{1}(1,\xi_{2})=x_{1}|_{\Gamma_{3}}\\
&x_{1}(\xi_{1},0)=x_{1}|_{\Gamma_{4}},\quad x_{1}(\xi_{1},1)=x_{1}|_{\Gamma_{2}}\\
\end{split}
\end{equation*}
and
\begin{equation*}
\begin{split}
&x_{2}(0,\xi_{2})=x_{2}|_{\Gamma_{1}},\quad x_{2}(1,\xi_{2})=x_{2}|_{\Gamma_{3}}\\
&x_{2}(\xi_{1},0)=s(x_{1},\mu),\quad x_{2}(\xi_{1},1)=x_{2}|_{\Gamma_{2}},\\
\end{split}
\end{equation*}
where $x_{i}|_{\Gamma_{j}}$ denotes the $x_{i}$ coordinate along the boundary segment $\Gamma_{j}$. One can choose different distributions of boundary coordinates in $x$ as boundary conditions to achieve better computational results.  With various methods to construct the stochastic map, the map is not bijective in general. Thus, we will describe
all the numerical results on the mapped domain $E$.

We make $100\times 100$ fine grid to compute the reference solution and the number of degrees freedom $N_{f}=10201$ for fine scale FE method. The local model reduction computation (GMsFEM) is performed on the coarse grid $5\times 5$ with $180$ basis functions. For offline procedure, we  select $100$  optimal samples for snapshot functions, i.e.,  $n_{train}=100$. Similar to the definition of $L^{2}$ error, we define the energy error by
\[ e^{H}_{u}=\frac{1}{N}\sum\limits^{N}_{i=1}\frac{\|u_{ref}(x,\mu_{i})-u_{lg}(x,\mu_{i})\|_{a(E)}}
{\|u_{ref}(x,\mu_{i})\|_{a(E)}}
\]
The energy error $e^{H}_{\lambda}$ for the adjoint variable $\lambda$ can be defined similarly.

To evaluate the approximation for the local-global model reduction method, we randomly choose $1000$  parameter samples and compute the average relative $L^{2}$ errors and energy  errors. For a fixed regularization parameter ($\beta=10^{-4}$) and five local-global basis functions, the  mean and the variance of state and control variables are all shown in Fig.~\ref{ex-var}, where  the first row represents for reference solutions of optimal control problem (\ref{model-rdom1}). With local-global model reduction method, the mean and variance for solutions are plotted in second row. From the figure, the moments  of  optimal control $f$ and the moments  of  state $u$ using local-global reduced method have a good agreement with the reference results.

\begin{figure}[!h]
\centering
\subfigure[ref-mean-state]{\includegraphics[width=1.5in, height=1.7in,angle=0]{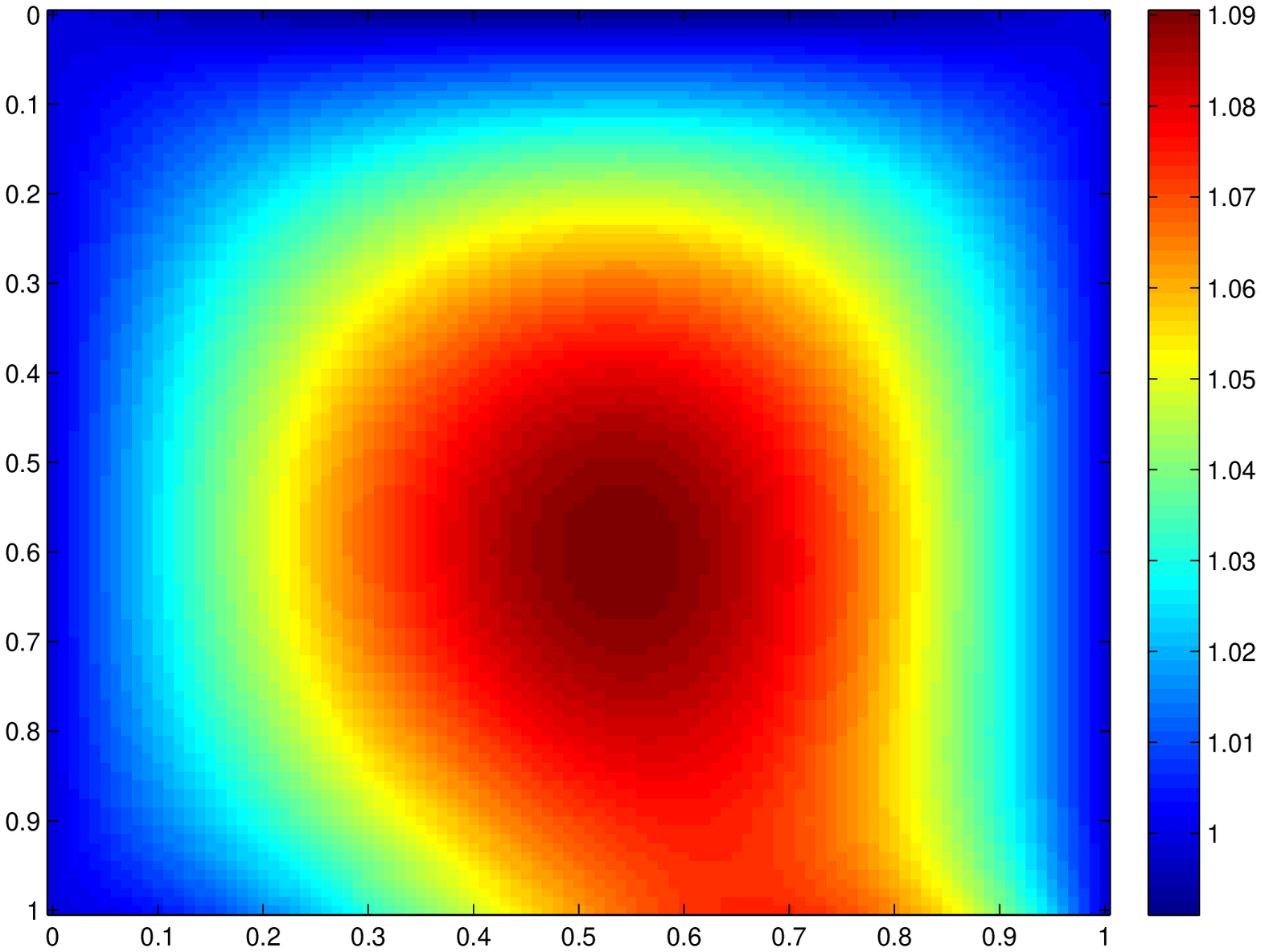}}
\subfigure[ref-variance-state]{\includegraphics[width=1.5in, height=1.7in,angle=0]{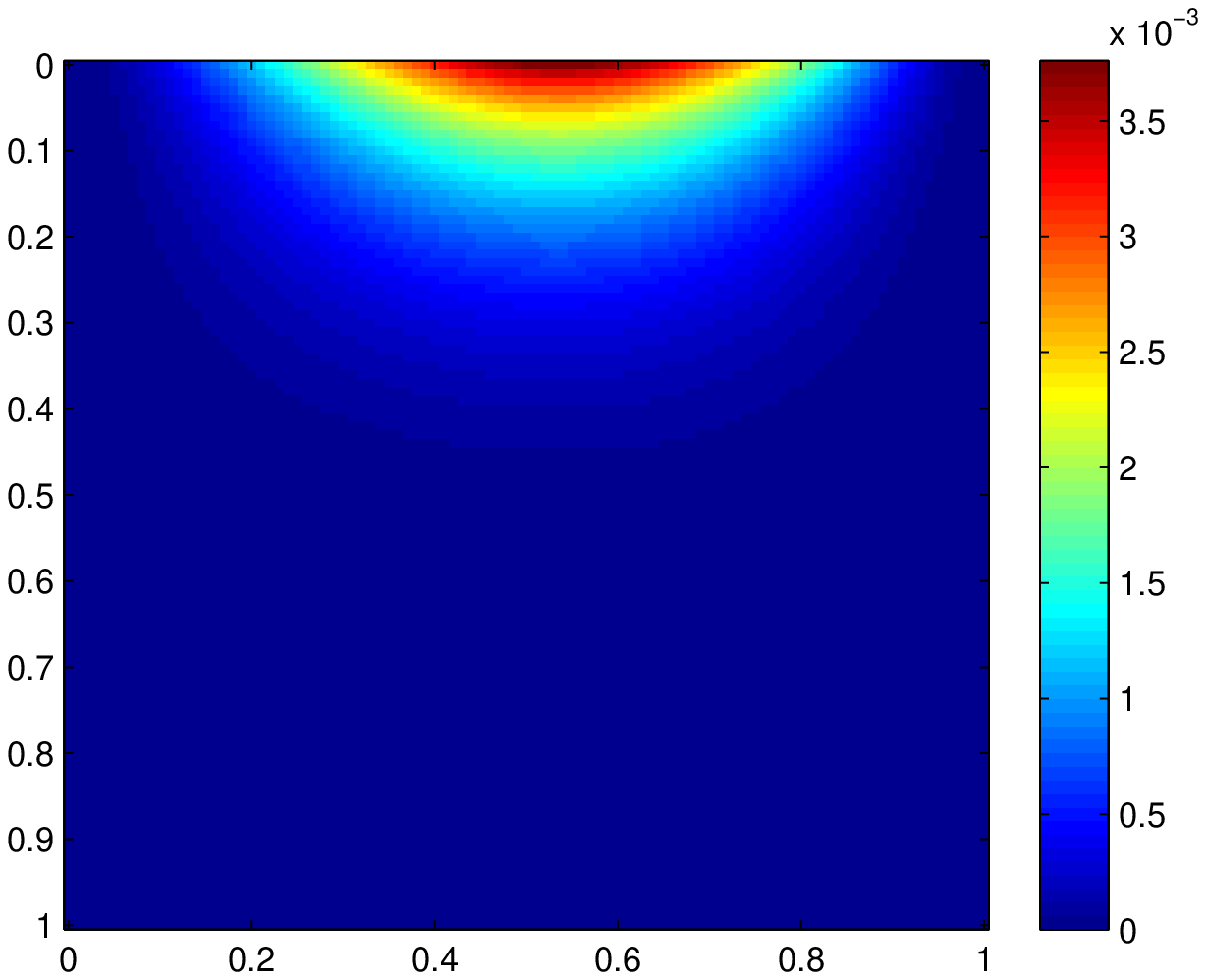}}
\subfigure[ref-mean-control]{\includegraphics[width=1.5in, height=1.7in,angle=0]{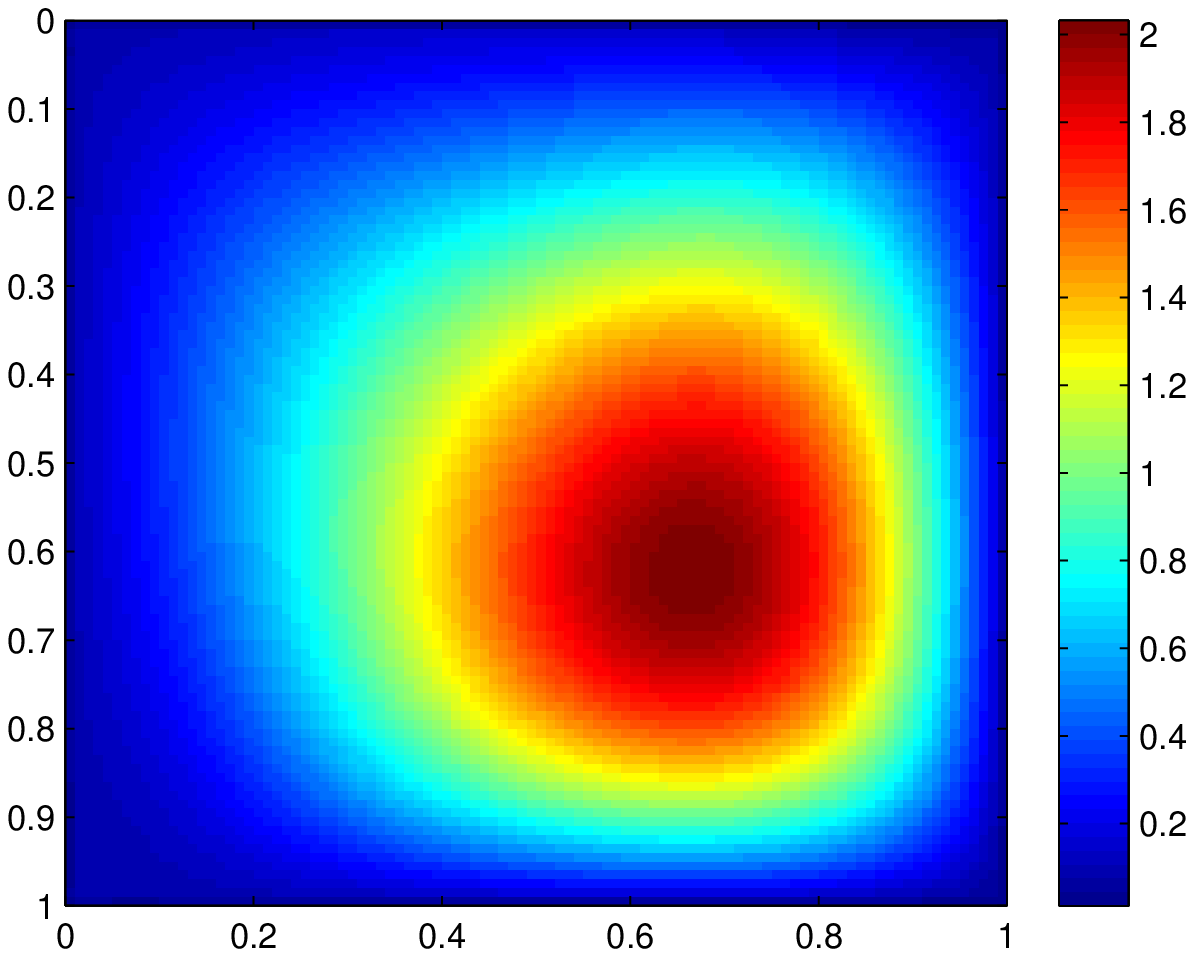}}
\subfigure[ref-variance-control]{\includegraphics[width=1.5in, height=1.7in,angle=0]{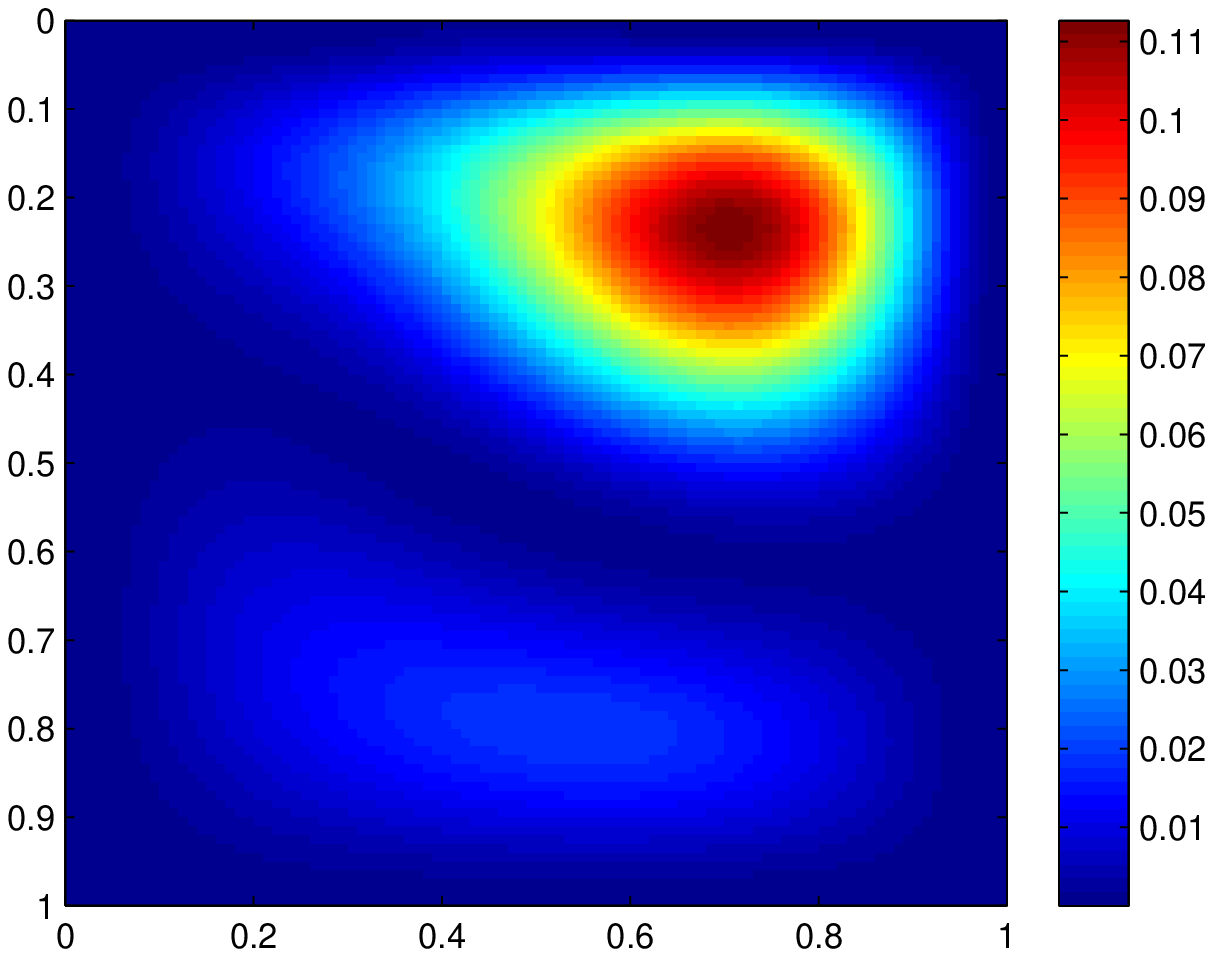}}
\subfigure[mean of state]{\includegraphics[width=1.5in, height=1.7in,angle=0]{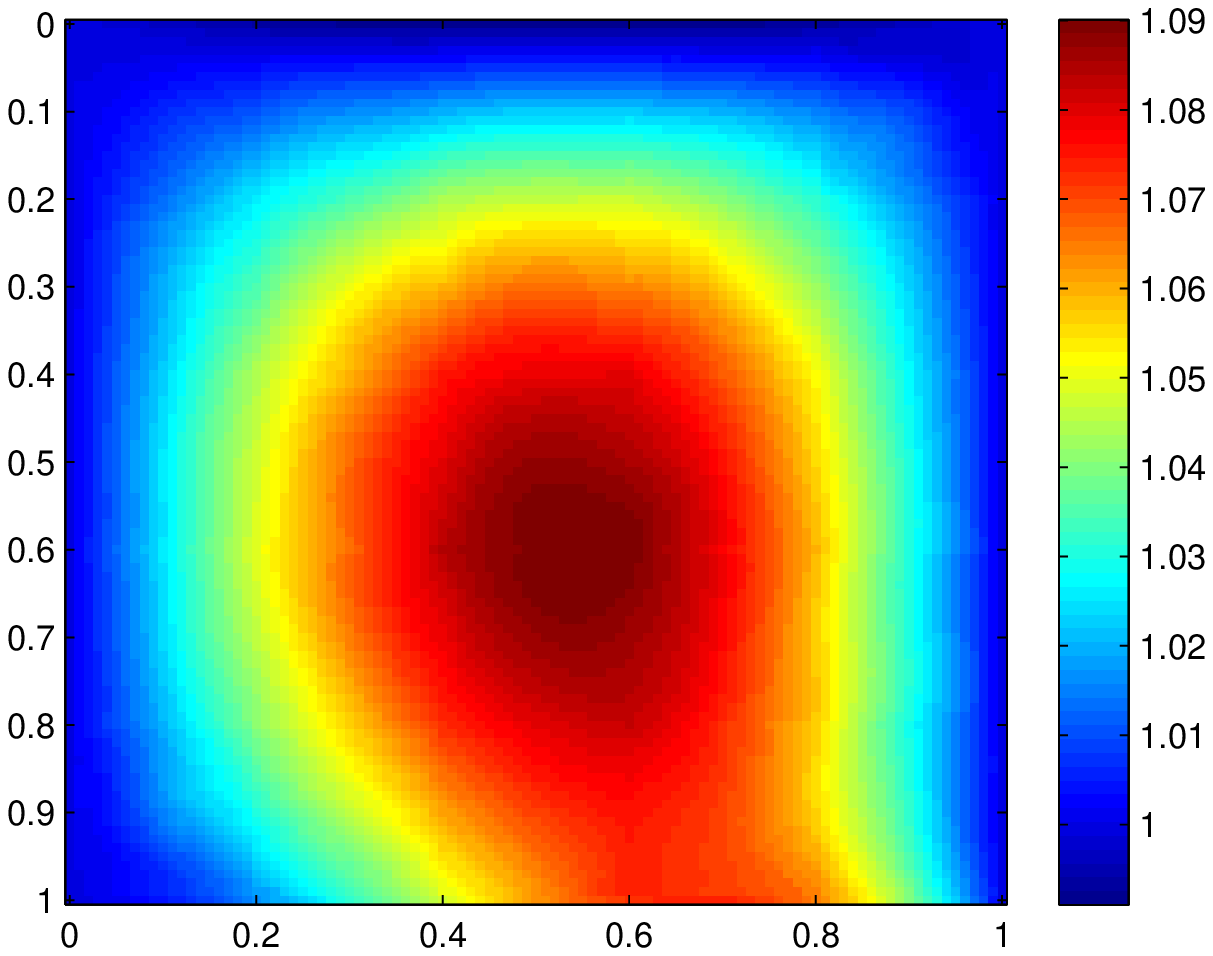}}
\subfigure[variance of state]{\includegraphics[width=1.5in, height=1.7in,angle=0]{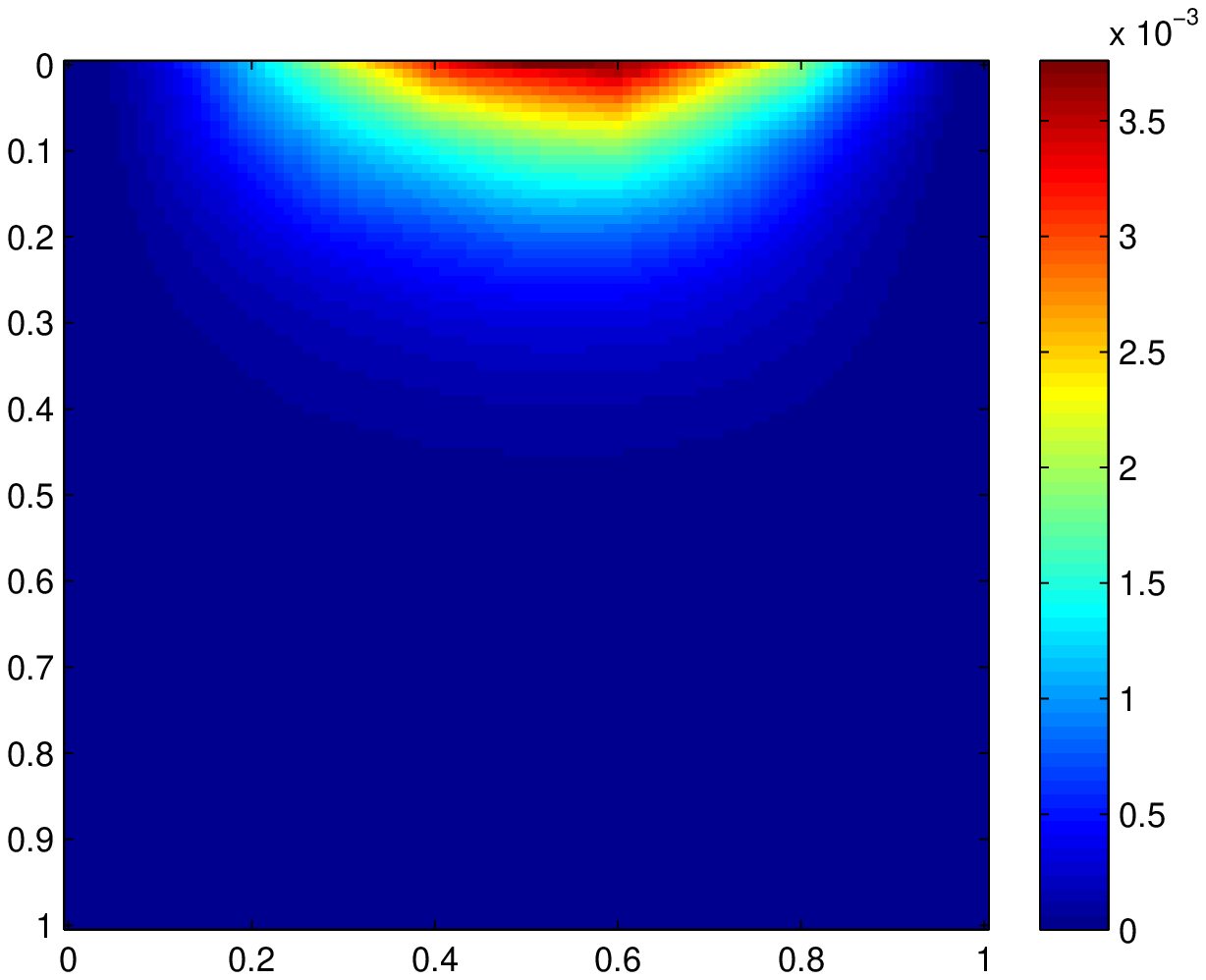}}
\subfigure[mean of control]{\includegraphics[width=1.5in, height=1.7in,angle=0]{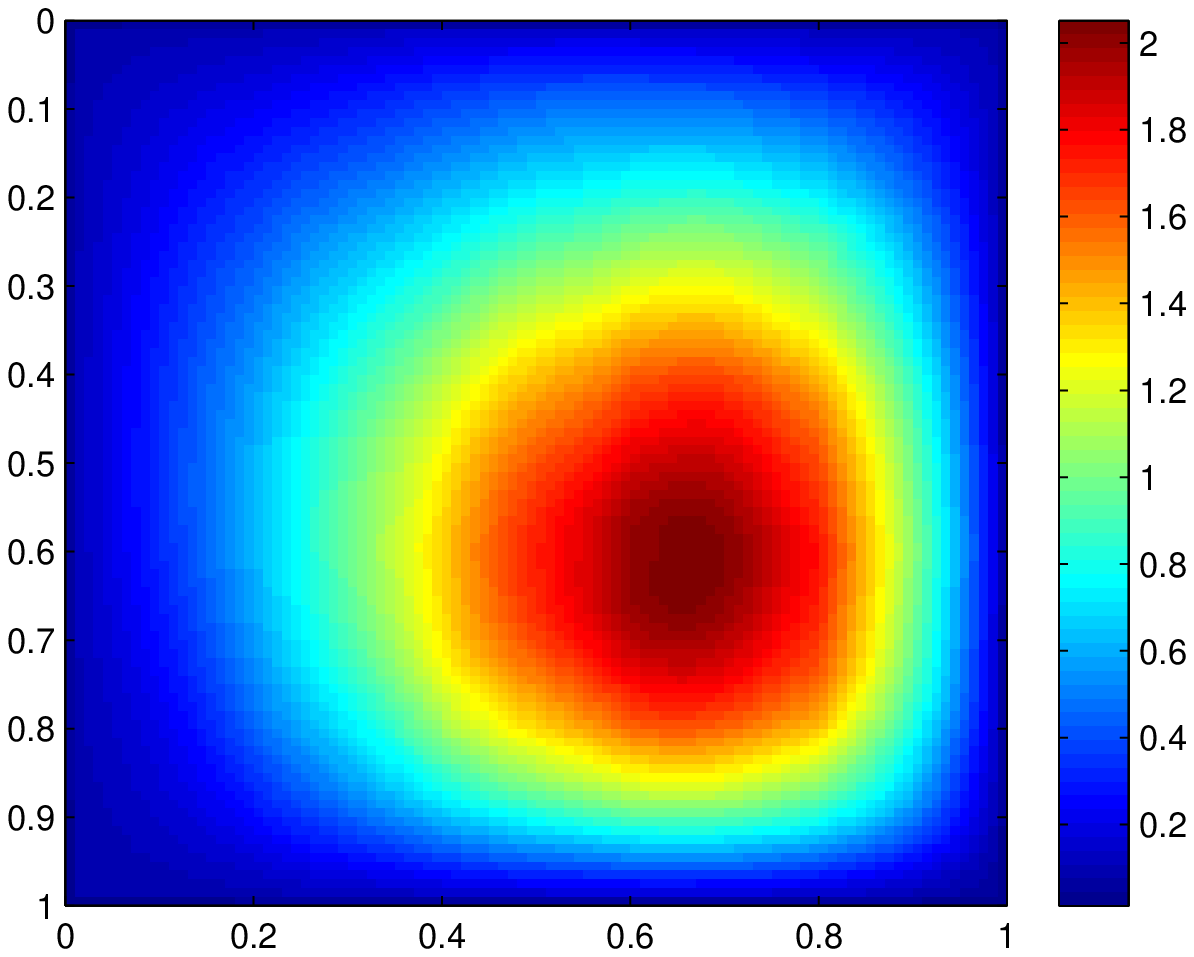}}
\subfigure[variance of control]{\includegraphics[width=1.5in, height=1.7in,angle=0]{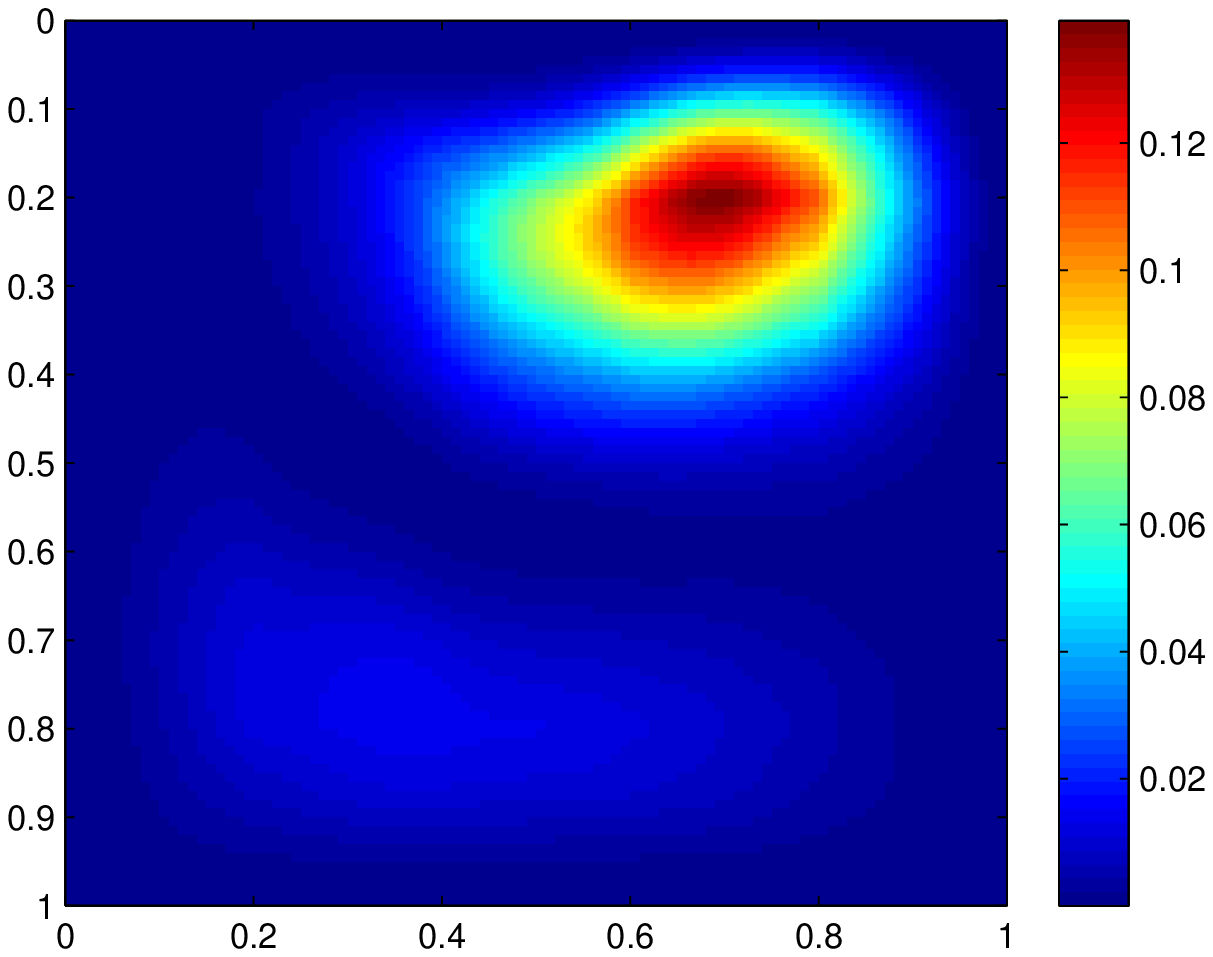}}
\caption{\em{Mean and variance of reference solutions of (\ref{model-rdom1}) (1st row).
Mean and variance of local-global reduced solutions (2nd row).}}
\label{ex-var}
\end{figure}
In Table \ref{Table-greedy-L} and Table \ref{Table-greedy-H}, we list the number of optimal parameter samples for global model reduction method versus the relative error in the sense of $L^{2}$-norm and energy norm, respectively, for the state, control and adjoint variables.
\begin{table}[!htbp]\center
\renewcommand\arraystretch{1.2}
\scriptsize
  \caption{\em{Average relative $L^{2}$ errors with different number of optimal parameter samples.}}
  \label{Table-greedy-L}
    \begin{tabular}{c| c c c c c c c}
    \hline
    $N_{max}$ & $2$      & $3$       & $4$       & $5$       &$6$        &$7$ &$8$\\
   \hline
    $e^{2}_{u}$            & 4.757E-03 & 4.264E-03 & 2.901E-03 & 1.373E-03 & 7.852E-04 & 7.610E-04 & 7.060E-04 \\
    $e^{2}_{f}$            & 1.807E-01 & 1.816E-01 & 1.553E-01 & 3.908E-02 & 2.764E-02 & 2.698E-02 & 2.485E-02 \\
    $e^{2}_{\lambda}$      & 1.808E-01 & 1.743E-01 & 1.663E-01 & 3.927E-02 & 2.704E-02 & 2.604E-02 & 2.422E-02 \\
    \hline
    \end{tabular}
\end{table}%
\begin{table}[!htbp]\center
\renewcommand\arraystretch{1.2}
\scriptsize
  \caption{\em{Average relative $H^{1}$ errors with different number of optimal parameter samples.}}
  \label{Table-greedy-H}
    \begin{tabular}{c| c c c c c c c}
    \hline
    $N_{max}$ & $2$      & $3$       & $4$       & $5$       &$6$        &$7$        &$8$\\
   \hline
    $e^{H}_{u}$        &5.992E-01 & 5.831E-01 & 3.908E-01 & 1.841E-01 & 1.533E-01 & 1.528E-01 & 1.510E-01 \\
    $e^{H}_{\lambda}$  &2.758E-01 & 2.606E-01 & 2.626E-01 & 1.270E-01 & 1.154E-01 & 1.147E-01 & 1.124E-01 \\
    \hline
    \end{tabular}
\end{table}%
We see that, for a fixed number of local basis functions (L=5), the accuracy  will improve as the number of
optimal parameter samples (the number of global basis functions $N_{max}$) increases gradually.
On the other hand, for the same optimal parameter samples ($N_{max}$=5), Table \ref{Table-basis-L} shows that
more  local basis functions render a better approximation.

\begin{table}[!htbp]\center
\scriptsize
  \caption{\em{Average relative $L^{2}$ errors for local-global model reduction method vs. different number of local basis functions.}}
  \label{Table-basis-L}
    \begin{tabular}{c| c c c c c c c c c}
    \hline
    $L$ & $2$      & $3$       & $4$       & $5$       &$6$        &$7$         \\
    \hline
    $u_{lg}$        & 1.863E-03 & 1.720E-03 & 1.442E-03 & 1.373E-03 & 1.352E-03 & 1.304E-03 \\
    $f_{lg}$        & 4.958E-02 & 4.397E-02 & 4.303E-02 & 3.908E-02 & 3.880E-02 & 3.832E-02 \\
    $\lambda_{lg}$  & 4.829E-02 & 4.349E-02 & 4.303E-02 & 3.927E-02 & 3.893E-02 & 3.851E-02 \\
    \hline
    \end{tabular}
\end{table}%

\begin{table}[!htbp]\center
\renewcommand\arraystretch{1.2}
\scriptsize
\caption{\em{Computational details for the local-global reduced model for the stochastic optimal control problem defined on the random domain.}}
\label{CPU-time-2}
    \begin{tabular}{l l| l l}
    \hline
    $\mathbf{Computation ~ setting}$                 &$ $       &$\mathbf{local-global~reduced~model}$  & $$\\
    \hline
    Number of FE dofs $\mathcal{N}$              &10201     & Number of RB dofs                & 25\\
    Number of parameters                         & 5        & Dofs reduction                   & 1216:1\\
    Error tolerance greedy $\epsilon^{*}_{tol}$  &$10^{-5}$ & Offline greedy time              & 1.704E+03 s \\
    Number of local basis functions                 & 5        & Offline time for snapshot spaces & 9.364E+02 s \\
    Number of test parameters                    &1000      & Online average time for optimal solutions   & 2.490E-02 s\\
    \hline
    \end{tabular}
\end{table}

In Table \ref{CPU-time-2}, we list the CPU time for the optimal control problem defined on the random domain with the high-fidelity model (FEM in fine grid) and the local-global reduced model. Before implementing the global reduced method, we need to compute the nested snapshots. As a sharp comparison, the local model reduction method only needs about 15 minutes to get the snapshot spaces, while the FE method in fine grid  requires more than  65 hours for the snapshots.  At the online stage, it  takes 2.121s to get the optimal solutions for per  parameter sample using the  FE method. However, the online average time is only $2.490\text{E}-02s$ per sample.
This shows that the local-global model reduction method  can significantly improve the computational efficiency for the stochastic optimal control problem defined on random domain.

\subsection{Stochastic optimal Neumann boundary control problem}
\label{Ex3}
Compared with   the  distributed control problems applied on the entire domain,
 boundary control problems have  the control applied only on the boundary.
Such boundary control problems are perhaps  more physically realistic because for real-world applications,
it may  be possible to only  control the physical property along  the boundary
of the domain. In this section, we consider  the following Neumann boundary control  problems using  local-global model reduction method, i.e.,
\begin{equation}
\label{model-bd}
\begin{cases}
\begin{split}
 &\min\limits_{u,g} J=\frac{1}{2}\|u(x,\mu)-\hat{u}(x,\mu)\|
 ^2_{\mathscr{L}^2(\Omega)}+\frac{\beta}{2}\|g(x)\|^2_{\mathscr{L}^2(\Omega)},\\
&s.t.~~-div (\kappa(x,\mu)\nabla u)=f(x) ~~\text{in}~\Omega\\
&\quad \quad \kappa(x,\mu)\frac{\partial u}{\partial n}=g(x) ~~\text{on}~\partial \Omega.
\end{split}
\end{cases}
\end{equation}
In the simulation, we set
\[
\kappa(x,\mu)= \exp(-\frac{(x_{1}-\mu_{1})^2}{4}-\frac{(x_{2}-\mu_{2})^2}{4}),\quad \hat{u}(x,\mu)=(x_{1}-\mu_{1})^2+(x_{2}-\mu_{2})^2.
 \]
The source term $f(x, \mu)$ is defined by
 \[
 f(x)=\frac{1}{2}\sin(\pi x_{1})\cos(2\pi x_{2})+x_{1}x_{2}+(\frac{x_{1}}{6}+\sin(\pi x_{2})+1)^2.
  \]
  Here the physical domain is still the unit square domain and $x=(x_{1},x_{2})$. The parameter $\mu=(\mu_{1},\mu_{2})$ and $\mu_{i} \sim Beta(1,1)$ ($i=1, 2$). Both functions $\kappa(x,\mu)$ and $\hat{u}(x,\mu)$ are not affine with respect to the parameter sample $\mu$, we  employ EIM to get affine approximations for $\kappa(x,\mu)$ and $\hat{u}(x,\mu)$.

For the simulation, we choose a uniform $100 \times 100$ fine grid to compute the reference solution for the Neumann boundary control problem. In the procedure of local model reduction, we set the coarse mesh size as $H=1/5$ and select $L=5$ multiscale basis functions at each coarse block. To construct the global RB spaces, we will select five optimal parameter samples by greedy algorithm. We set the number of training samples $n_{train}=100$ and take 2500 test samples to compute average errors and moments.

The expectation and standard deviation of the state variable $u$ are depicted  in
Fig.~\ref{ex-standard}, which shows that  the state approximation  using local-global reduced model matches the reference solution very well
for  both the mean and the standard deviation.
\begin{figure}[!h]
\centering
\subfigure[]{\includegraphics[width=1.55in, height=1.75in,angle=0]{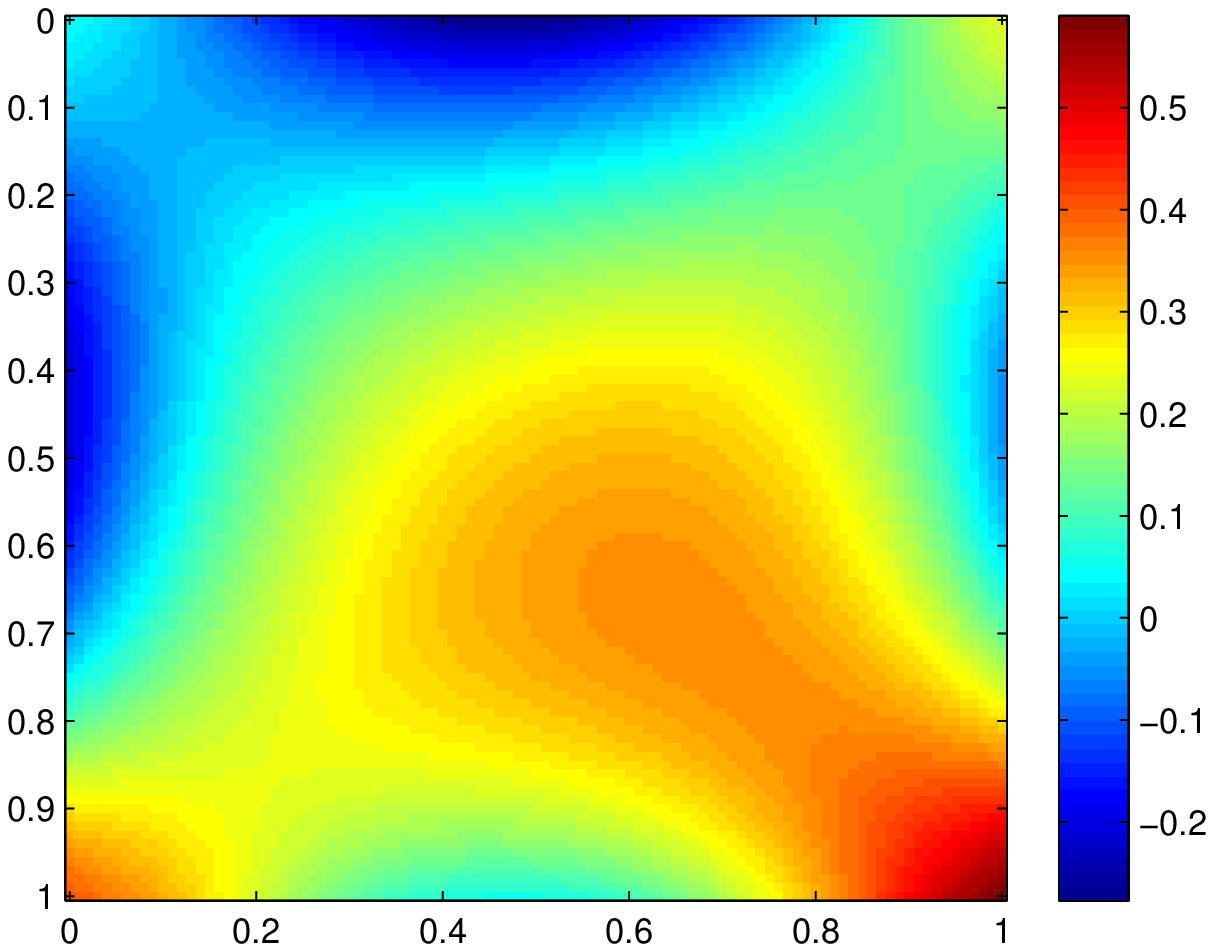}}
\subfigure[]{\includegraphics[width=1.55in, height=1.75in,angle=0]{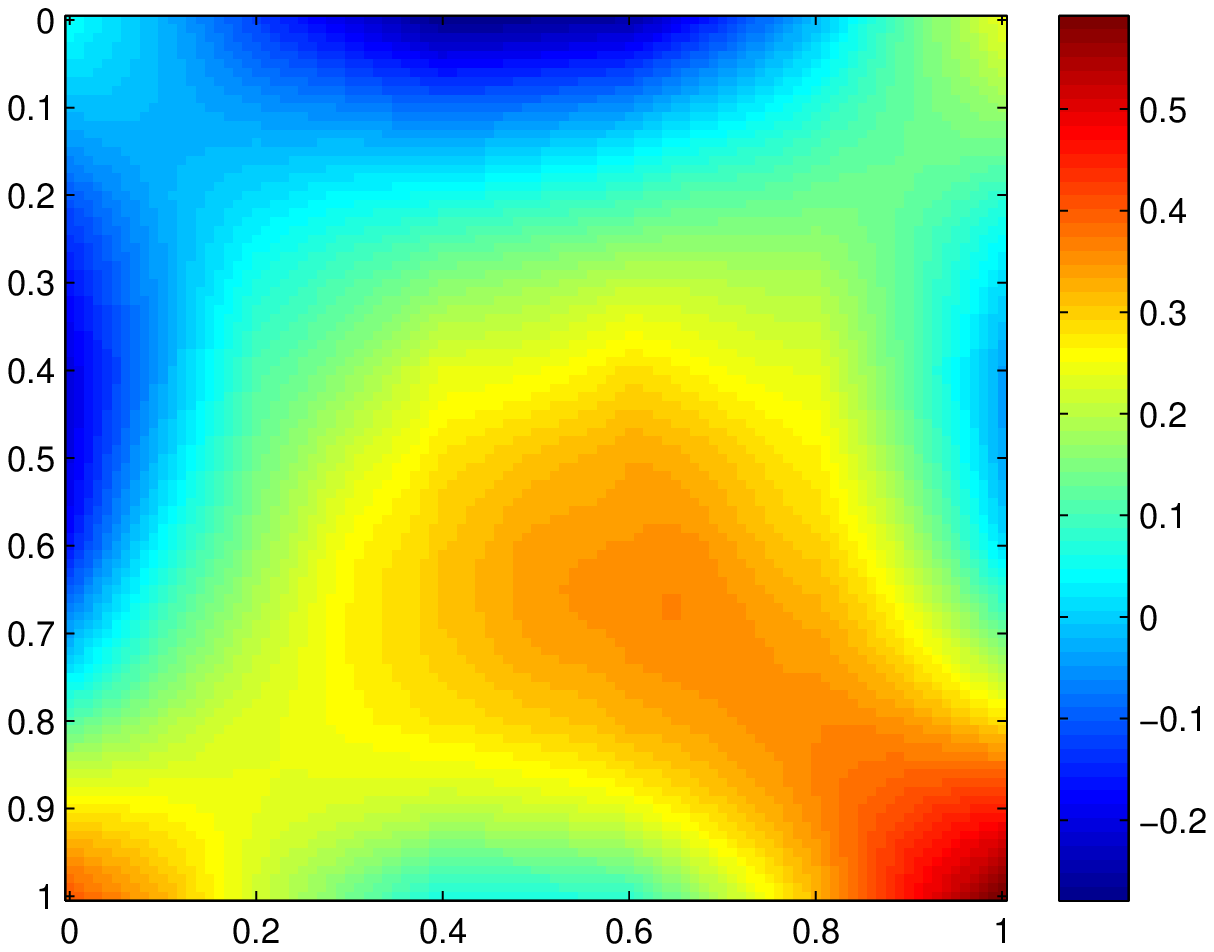}}
\subfigure[]{\includegraphics[width=1.55in, height=1.75in,angle=0]{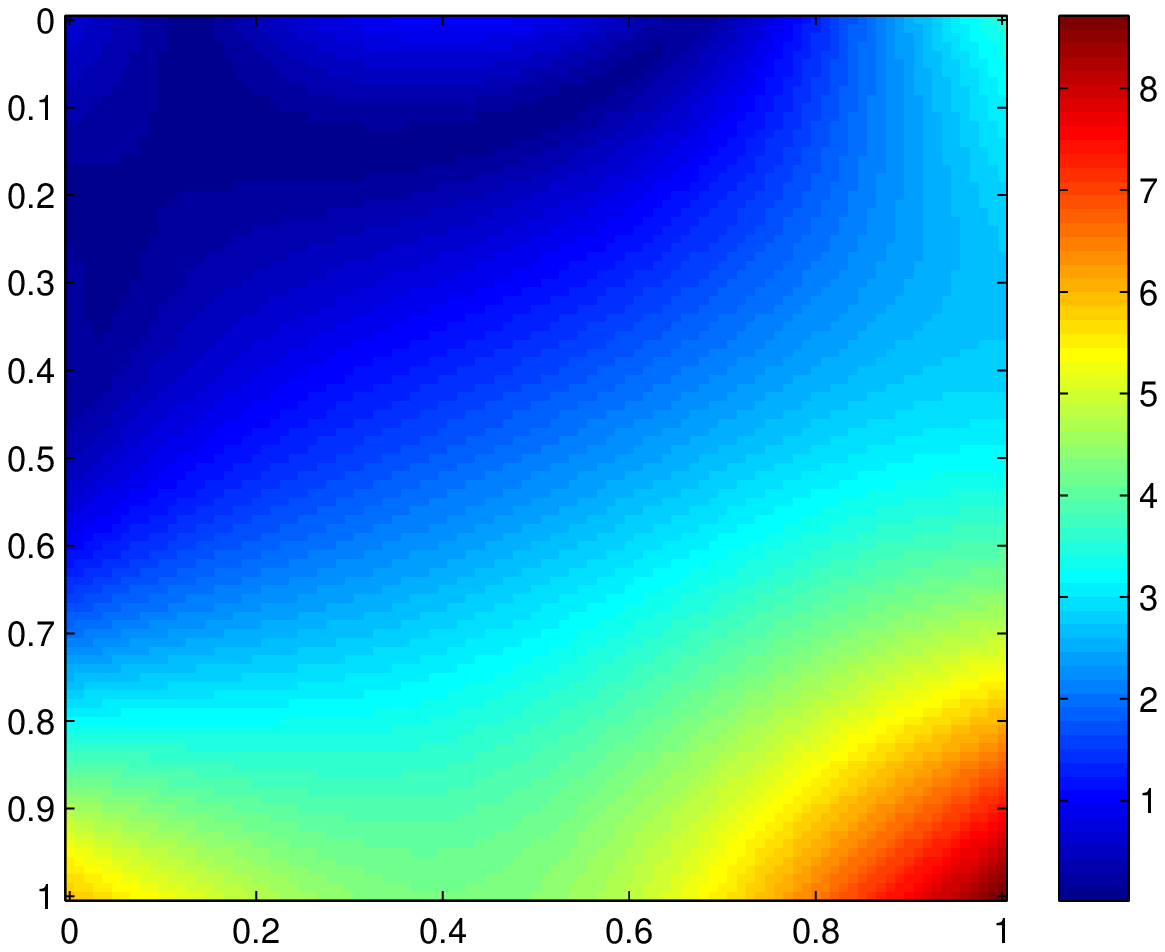}}
\subfigure[]{\includegraphics[width=1.55in, height=1.75in,angle=0]{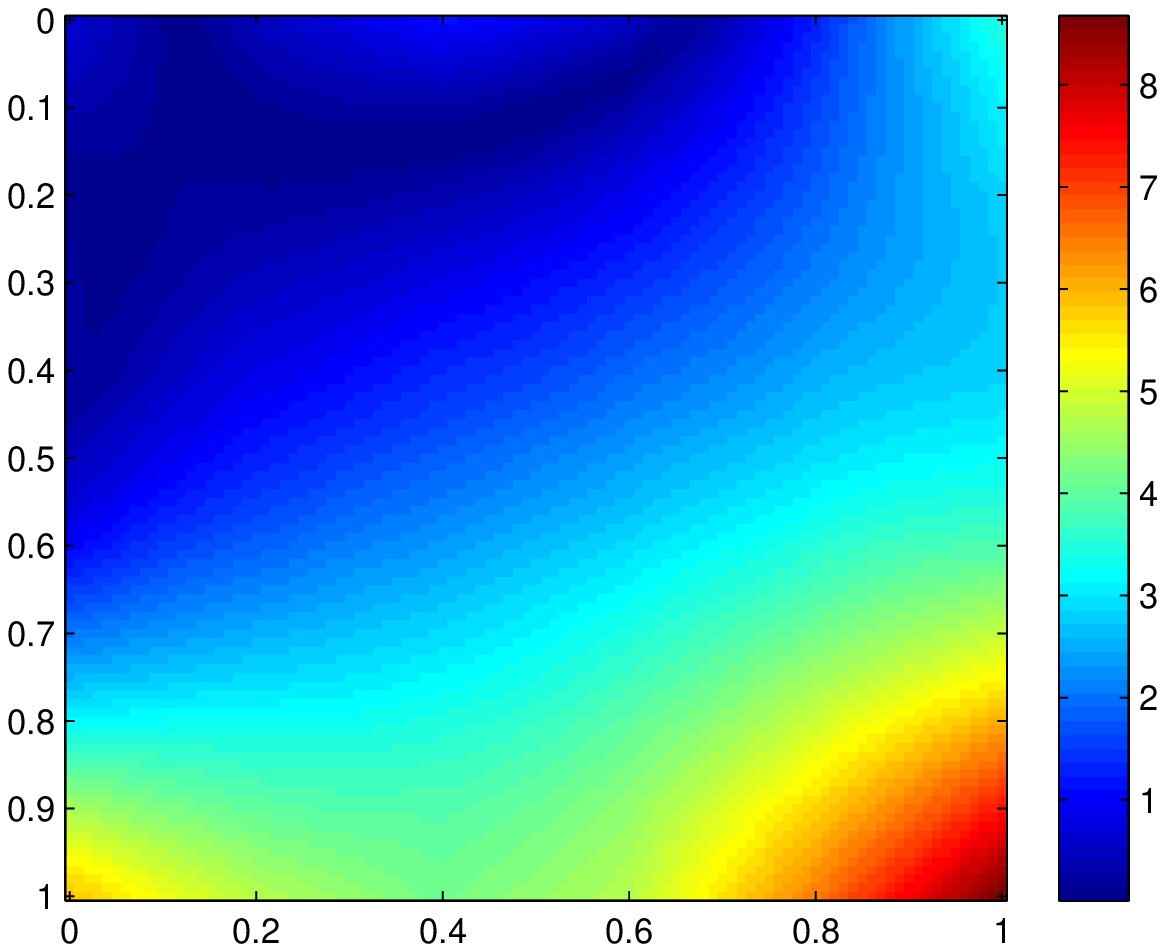}}
\caption{\em{(a) expectation of reference state; (b) expectation of state using local-global model reduction method; (c) standard deviations of state obtained by FE method;  (d) standard deviations of state using local-global model reduction method.}}
\label{ex-standard}
\end{figure}
In Fig.~\ref{bd-control}, we plot  the expectation $p_{lg}$ of the boundary control function $g$  using the  local-global reduced method.  We denote the expectation  and standard deviation of reference control by $p_{ref}$ and $\sigma_{ref}$, respectively. We see that the expectation of control variable $g$ using the local-global model reduction method  lies in the  region between the line $p_{ref}+\sigma_{ref}$ and the line $p_{ref}-\sigma_{ref}$.
This shows the credibility of the approximation based on  the model reduction.

\begin{figure}[!h]
\centering
\subfigure[Left boundary]{\includegraphics[width=1.58in, height=2.1in,angle=0]{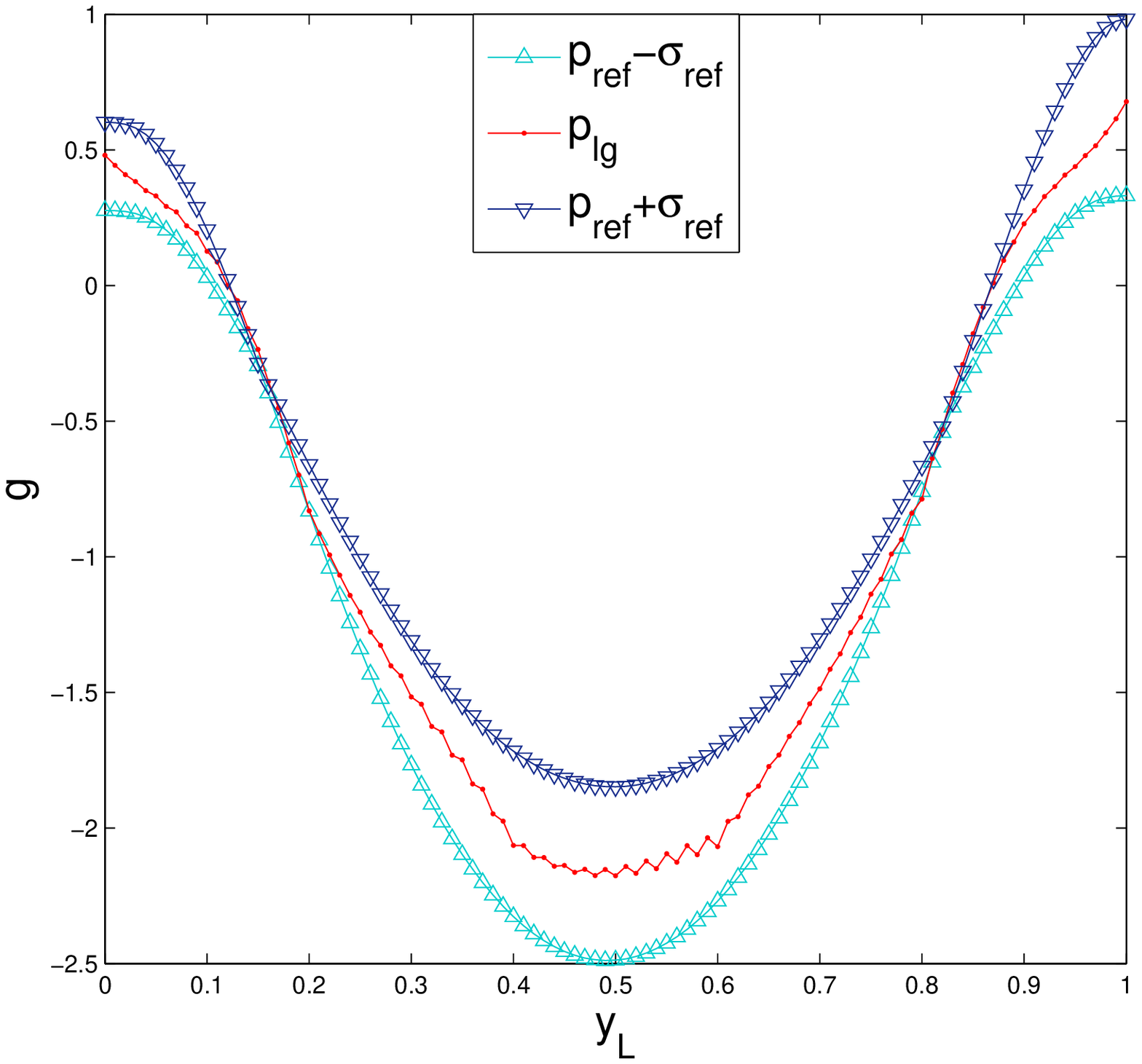}}
\subfigure[Bottom boundary]{\includegraphics[width=1.58in, height=2.1in,angle=0]{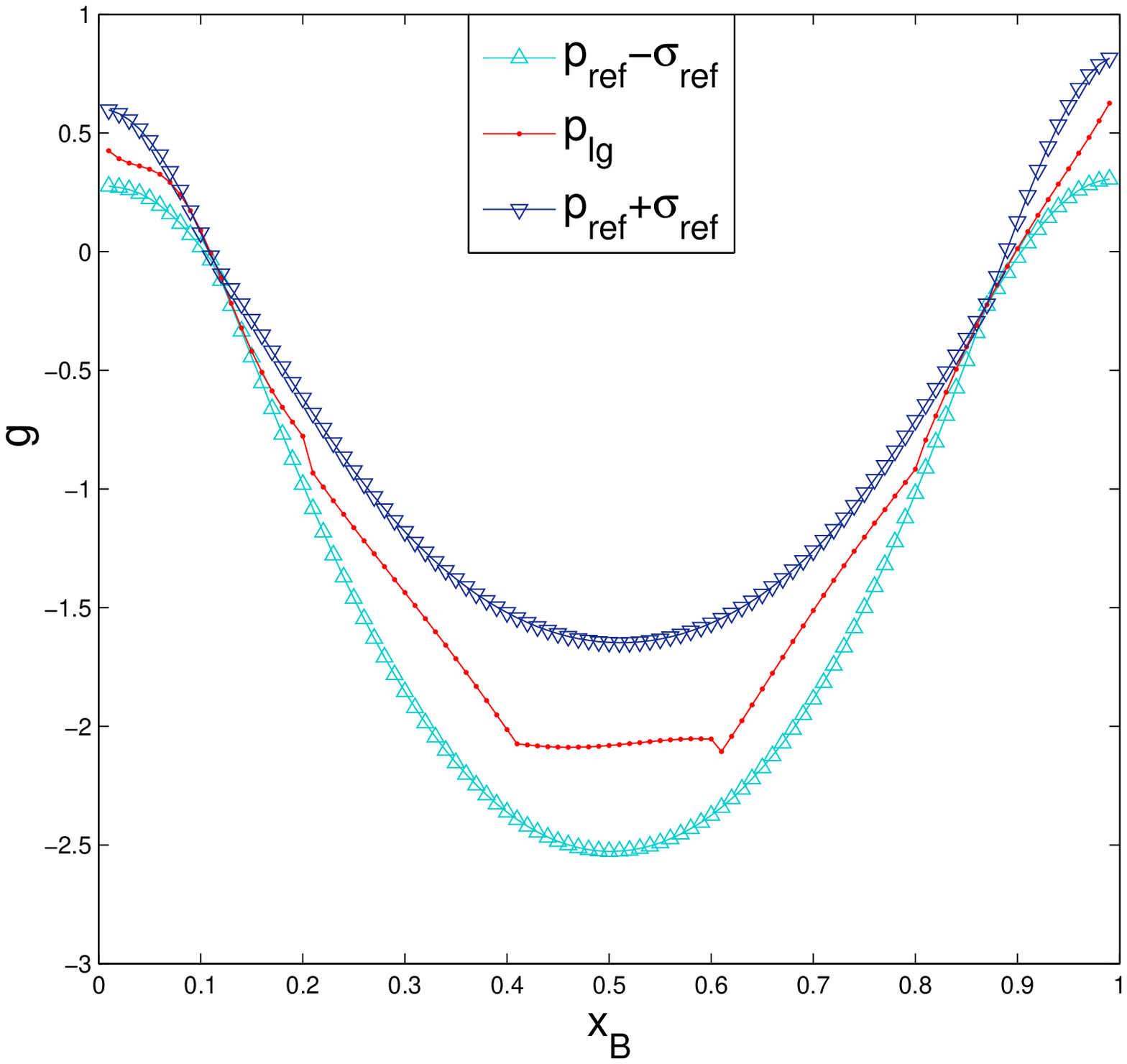}}
\subfigure[Upper boundary]{\includegraphics[width=1.58in, height=2.1in,angle=0]{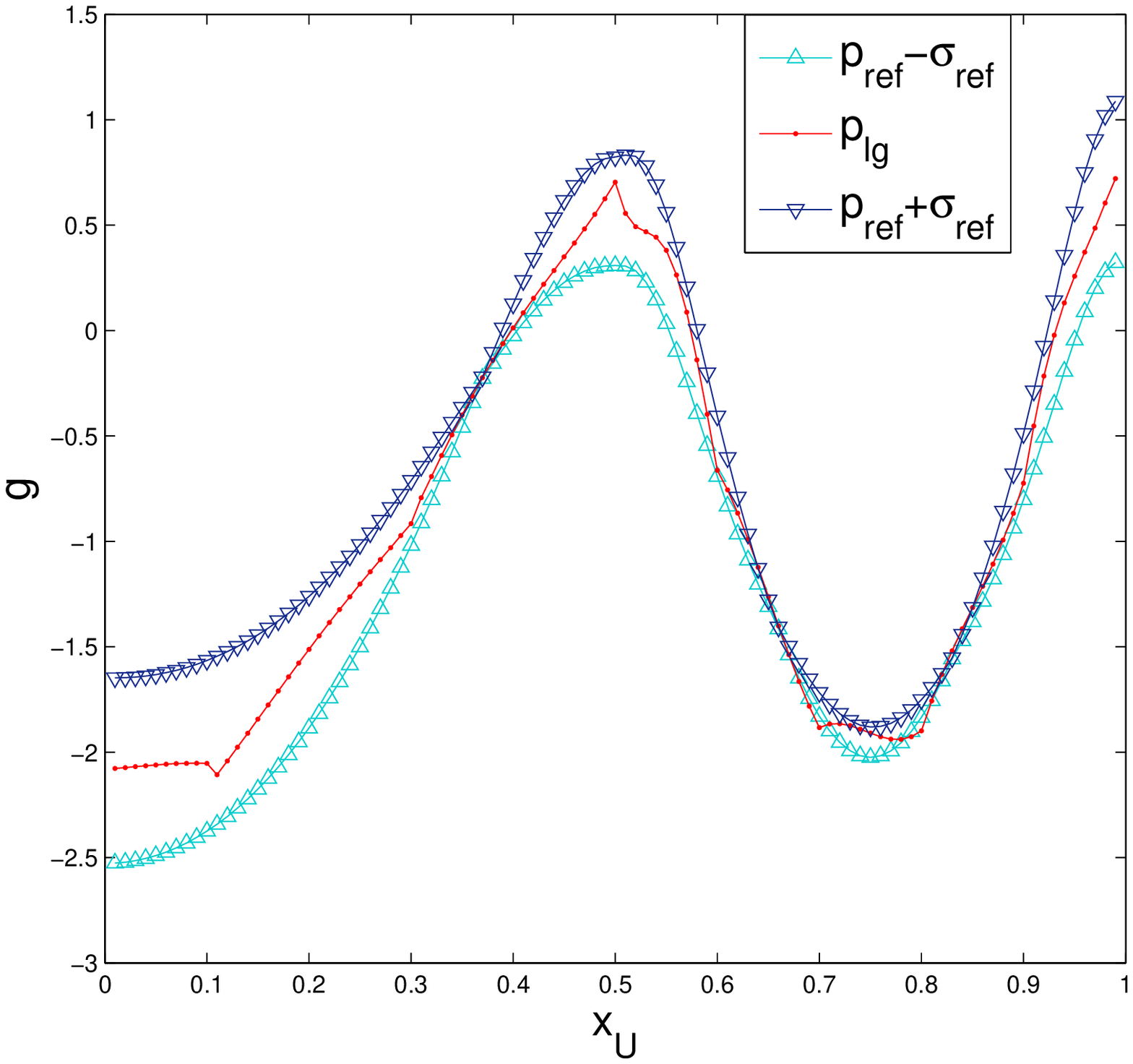}}
\subfigure[Right boundary]{\includegraphics[width=1.58in, height=2.1in,angle=0]{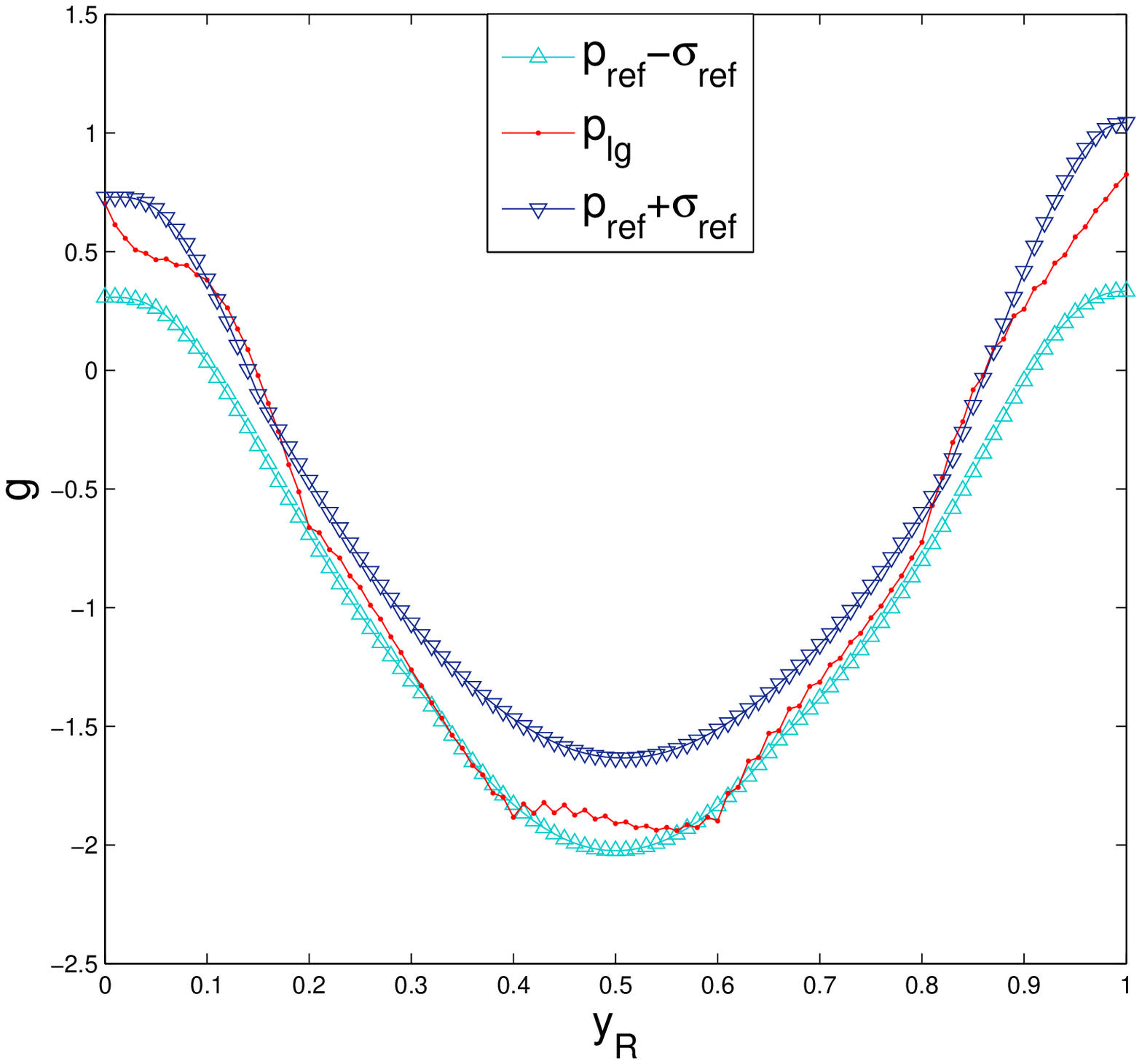}}
\caption{\em{Expectation $p$ and standard $\sigma$ of the Neumann boundary condition $g$ on the four boundaries of the domain $\Omega$.}}
\label{bd-control}
\end{figure}

We fix the regularization parameter $\beta$ with $\beta=10^{-4}$ and coarse mesh size with $H=1/5$. For different number of local basis functions at each coarse block and different number of global basis functions,
 we compute the difference between the state function and target function, and the minimum value of cost functional.  The results are listed in  Table \ref{first-cost}.
  From the table, we can find that the approximation of the cost functional is not very sensitive to the model fidelity for the boundary control problem.

\begin{table}[!htbp]\center
\renewcommand\arraystretch{1.4}
\scriptsize
\setlength{\belowcaptionskip}{5pt}
  \caption{\em{Numerical  results with different number of global basis functions $N_{max}$ and different number of local basis functions $L$.}}
   \label{first-cost}
    \begin{tabular}{c c c c c c}
    \hline
    \multirow{2}{2cm}{$\sharp~of~global~basis~functions$}&\multirow{2}{1cm}{}&\multirow{2}{1cm}{}{$L=3$}
    &\multirow{2}{1cm}{}&\multirow{2}{1cm}{}{$L=6$}&\multirow{2}{1cm}{}\\
    \cline{3-6}
    &$(N_{max})$& $\|u_{lg}(x,\mu)-\hat{u}(x,\mu)\|^{2}_{\mathscr{L}^{2}}$   &$J_{min}$ & $\|u_{lg}(x,\mu)-\hat{u}(x,\mu)\|^{2}_{\mathscr{L}^{2}}$   &$J_{min}$\\
   \hline
    2     && 1.099863E-01 & 5.531052E-02 & 1.098623E-01 & 5.524519E-02 \\
    4     && 1.098805E-01 & 5.525738E-02 & 1.097684E-01 & 5.519797E-02 \\
   \hline
    \end{tabular}
\end{table}%

To compare the computation efficiency,   we describe the computation setting and list the CPU time in Table \ref{CPU-time-3} when we use  the local-global model reduction method and FE method (time in the brackets).
By the table, we can find that the local-global model reduction method significantly improves the efficiency compared with standard FE method in fine grid.

\begin{table}[!htbp]\center
\renewcommand\arraystretch{1.2}
\scriptsize
\caption{\em{Comparison of the CPU time for Neumann boundary control problem with FE method and local-global model reduction method.}}
\label{CPU-time-3}
    \begin{tabular}{l l | l l}
    \hline
    $\mathbf{Computation~setting}$        &$ $    & $\mathbf{Local-global~reduced~model}$          &$$\\
    \hline
    Number of FE dofs $\mathcal{N}$     &10201  & Linear system size reduction              & 1100:1\\
    Number of optimal parameter samples & 4     & Offline greedy time                       & 2.910E+03 s (3.337E+03 s)\\
    Number of local basis functions     & 5     & Offline time for snapshot spaces          & 3.997E+02 s (5.397E+01 h)\\
    Number of test parameter samples    &2500   & Online average time for optimal solutions & 1.774E-01 s (1.131 s)\\
     \hline
    \end{tabular}
\end{table}

\section{Conclusion}
In this paper, we have presented a local-global model reduction method  for stochastic optimal control problems constrained by  PDEs. The possible  uncertainty  we considered in the paper  arises from the PDE coefficient, the target function,  the physical domain and the boundary condition.  We used reduced basis method and GMsFEM to develop the local-global model reduction.
We recast the optimal control problem into
   a stochastic saddle point formulation and proved the global existence and uniqueness for the stochastic optimal control solution.
   The local-global model reduction  is very suitable for many-query situations. This can significantly enhance the computation efficiency to solve the   stochastic optimal control problems.
   A few numerical examples have been carefully implemented for different stochastic optimal control problems: distributed control on deterministic domain and random domain, boundary control.
   The numerical  results showed the efficacy of the proposed model reduction method   and its  promising  application in  stochastic optimal control problems governed by complex models.

\section*{Acknowledgments}
We acknowledge the support of Chinese NSF 11471107.


\end{document}